\documentclass[12pt,english]{smfart}
\usepackage{mathrsfs}
\usepackage{smfthm}
\usepackage{amssymb,esint}
\usepackage{latexsym}
\usepackage{verbatim}
\usepackage{amsmath}
\usepackage{graphicx,color}
\usepackage[all]{xy}
\usepackage{pstricks}
\usepackage{pst-plot}
\usepackage{epsfig}
\usepackage{url}
\usepackage{paralist}

\usepackage[right=1.8cm,left=1.8cm,bottom=4cm]{geometry}

 \theoremstyle{plain}
 \newtheorem{thm}{Theorem}[section]
 \newtheorem{cor}[thm]{Corollary}
 \newtheorem{lem}[thm]{Lemma}
 \newtheorem{sublem}[thm]{Sublemma}
 \newtheorem{propo}[thm]{Proposition}
 \theoremstyle {definition}
 
 \theoremstyle{remark}
 \newtheorem{rem}[thm]{Remark}

\def\polhk#1{\setbox0=\hbox{#1}{\ooalign{\hidewidth
    \lower1.5ex\hbox{`}\hidewidth\crcr\unhbox0}}}

\renewcommand{\Im}{\operatorname{Im}} 
\renewcommand{\Re}{\operatorname{Re}}
\newcommand{\Dom}{\operatorname{Dom}\,}
\newcommand{\diam}{\operatorname{diam}\,}
\newcommand{\area}{\operatorname{area}}
\newcommand{\daron}{\operatorname{int}\,}
\newcommand{\ex}{\operatorname{\mathbb{E}xp}}
\newcommand{\eps}{\varepsilon}
\newcommand{\vtet}{\vartheta}
\newcommand{\irr}{\operatorname{HT}_N}
\newcommand{\pc}{\mathcal{PC}}
\newcommand{\alf}{\alpha}
\newcommand{\ff}{\mathcal{I}\hspace{-1pt}\mathcal{S}}
\newcommand{\QIS}{\mathcal{Q}\hspace{-1pt}\mathcal{I}\hspace{-1pt}\mathcal{S}}

\newcommand{\orb}{\mathcal{O}}
\newcommand{\RS}{\mathbb{\hat{C}}}
\newcommand{\vfi}{\varphi}
\newcommand{\cc}{\mathbb{C}}
\newcommand{\ra}{\rightarrow}
\newcommand{\ea}{e^{2 \pi \mathbf{i} \alpha}}
\newcommand{\p}{\mathcal{P}}
\newcommand{\C}{\mathcal{C}}
\newcommand{\Csh}{\mathcal{C}^\sharp}
\newcommand{\rr}{\mathcal{R}}
\newcommand{\D}{\mathcal{D}}
\newcommand{\G}{\mathcal{G}}

\newcommand{\cp}{\textup{cp}}
\newcommand{\cv}{\textup{cv}}
\newcommand{\Bi}{\textbf{i}}
\newcommand{\Bk}{\boldsymbol{k}}
\newcommand{\co}[1]{^{\circ {#1}}}
\newcommand{\dhyp}{\operatorname{d_{hyp}}}

\begin{document}

\title[orbits of quadratic polynomials with a neutral fixed point]{Typical orbits of quadratic polynomials
with a neutral fixed point: Non-Brjuno type}

\author{Davoud Cheraghi}

\address {Department of Mathematics, Imperial College London, UK} 
\email{d.cheraghi@Imperial.ac.uk}


\begin{abstract}
We investigate the quantitative and analytic aspects of the near-parabolic renormalization scheme introduced by 
Inou and Shishikura in 2006.
These provide techniques to study the dynamics of some holomorphic maps of the form 
$f(z) = \ea z + \mathcal{O}(z^2)$, including the quadratic polynomials $\ea z+z^2$, for some irrational values of 
$\alpha$. 
The main results of the paper concern fine-scale features of the measure-theoretic attractors of these maps, 
and their dependence on the data.
As a bi-product, we establish an optimal upper bound on the size of the maximal linearization domain
in terms of the Siegel-Brjuno-Yoccoz series of~$\alpha$.  
\end{abstract}
\maketitle
\renewcommand{\thethm}{\Alph{thm}}
\section{Introduction}
\subsection{Neutral fixed points}
Let $f$ be a holomorphic map of the form 
\[f(z)=e^{2\pi \Bi \alpha}z+ a_2z^2+a_3 z^3+ \dots,\]
defined on a neighborhood of $0 \in \mathbb{C}$, and $\alf \in \mathbb{R}\setminus \mathbb{Q}$.  
Asymptotically near $0$, the orbits are governed by the rotation of angle $\alf$ and are highly recurrent. 
Away from zero, the influence of non-linearity increases, eventually reaching the scale where the behavior is 
governed by the global topological structure of the map.
For systems with unstable behavior near zero, the transition from local to global and back may occur 
infinitely often. 
This creates a delicate interplay among the arithmetic nature of $\alpha$, the non-linearities 
of the large iterates of $f$, and the global covering structure of the iterates of $f$.    
In this paper we study this problem. 

The ideal scenario is when the map $f$ is conformally conjugate to the linear map $w \mapsto \ea \cdot w$ 
on a neighborhood of $0$. 
To discuss this further, let us denote the best rational approximatants of $\alpha$ with $p_n/q_n$, $n\geq 1$. 
By a landmark result of Siegel and Brjuno \cite{Sie42,brj71}\nocite{Cre38}, if the series
\[\mathcal{B}(\alpha)= \textstyle\sum_{n=1}^{\infty} q_n^{-1} \log q_{n+1}\]
is finite, then near $0$ the map $f$ is confomally conjugate to a linear map. \nocite{Pe11,ChFa96,Her87,Pf17}
When $f$ is linearizable near $0$, the maximal domain on which the conjugacy exists is called 
the \textit{Siegel disk} of $f$. 
The geometry of the Siegel disks as well as the dynamics of $f$ near their boundaries have been 
the subject of extensive studies over the last few decades. 
These involve a wide range of methods, with consequences often depending on the arithmetic nature of $\alpha$. 
See for instance, \cite{Her85,Mc04,GrJo02,GrSw03,PZ04,ABC04,BC07,Ya08,Zak10,Zha11,Che11}. 
We note that for almost every $\alf \in \mathbb{R}\setminus \mathbb{Q}$, $\mathcal{B}(\alpha)< \infty$, 
while for generic choice of $\alf \in \mathbb{R}\setminus \mathbb{Q}$, $\mathcal{B}(\alpha)=\infty$. 
 
 
On the other hand, by a celebrated result of Yoccoz \cite{Yoc95}, if $\mathcal{B}(\alpha)=\infty$, the polynomial 
\[P_\alf(z):= \ea z+z^2\]
is not linearizable at zero. 
Although this optimality result has been further extended by similar ideas to special families of maps 
\cite{PM93,Ge01,Ok04}, it remains widely open in families of polynomials and rational maps.
Also, due to a non dynamical step in those arguments, very little has been understood about the local dynamics of 
non-linearizable maps. 
In \cite{PM97}, Perez-Marco constructs non-trivial local invariant compact sets containing $0$ for 
non-linearizable maps. 
But the necessary control on the geometry of these objects and the dynamics of the map on them has remained 
out of reach.

In 2006, Inou and Shishikura introduced a renormalization scheme that provides a powerful tool to study 
the dynamics of near parabolic maps, \cite{IS06}. 
This involves an infinite-dimensional class of maps $\mathcal{F}$, and a nonlinear operator 
$\rr : \mathcal{F} \to \mathcal{F}$, called near-parabolic renormalization. 
Every map in $\mathcal{F}$ is defined on a Jordan neighborhood of $0$, has a neutral fixed point at $0$, 
and a unique critical point of local degree two in its domain of definition.
Given $f\in \mathcal{F}$, $\rr(f)$ is defined as a sophisticated notion of the return map of $f$ about $0$ 
to a region in the domain of $f$, viewed in a certain canonically defined coordinate on that region.
Precise definitions appear in Section~\ref{sec:prelim}.    

In this paper we carry out a quantitative analysis of the near-parabolic renormalization scheme. 
This involves proving a number of foundational results on the combinatorial and analytic aspects of the scheme. 
In particular, we have slightly modified the definition of renormalization to make it suitable for applications.

Successive iterates of $\rr$ at some $f\in \mathcal{F}$ produces a renormalization tower; a sequence of maps 
$\rr\co{j}(f)$ which are related by the changes of coordinates. 
The general theme in theories of renormalization is that large iterates of $f$ often break down into compositions of a 
small number of the changes of coordinates and the maps $\rr\co{j}(f)$. 
However, due to the ``semi-local'' nature of near-parabolic renormalization, there are a number of issues which require 
careful consideration.

The maps in $\mathcal{F}$ have a partial covering structure, involving a branched covering of local degree two. 
The change of coordinate in the definition of renormalization also has a partial covering structure with a branch point. 
We prove some detailed orbit relations on the renormalization tower, relating the combinatorial aspects of the orbits 
of $f$ to the ones of $\rr\co{j}(f)$, for $j\geq 1$. 

The change of coordinates in the definition of renormalization involves transcendental mappings with highly 
distorting nature. 
Substantial part of the paper (Section~\ref{sec:perturbed}) is devoted to proving uniform (distortion) estimates on these 
maps and their dependence on $\alpha$.
To this end we have introduced a new approach to compare the changes of coordinates to some model maps 
using quasi-conformal mappings.

The interplay between the arithmetic of  $\alpha$ and the non-linearities of the iterates of $f$ is manifested in 
geometric aspects of the renormalization tower. 
We present a systematic approach to employing the above combinatorial and analytic tools 
to study the dynamics of the maps in $\mathcal{F}$. 

\subsection{Statements of the results}
Define the set of irrational numbers 
\[\irr:=\{[0;a_1,a_2,\dots]\in \mathbb{R}\mid \forall i\geq 1,  a_i\geq N\},\] 
where $N\in \mathbb{N}$ and $[0;a_1,a_2,\dots]=1/(a_1+1/(a_2+1/(\dots)))$ denotes the continued fraction expansion. 
For technical reasons, in this paper we require $\alf$ to be in $\irr$, for some fixed constant $N \in \mathbb{N}$. 
\footnote{This is also required in the near-parabolic renormalization scheme. 
However, it is conjectured that there exists a scheme with similar qualitative features for which $N=1$. 
So, we hope that the arguments presented here will be eventually applied to all irrational rotation numbers.}

The class of maps $\mathcal{F}$ fibers over $\irr$ as 
\[\mathcal{F}= \cup_{\alpha\in \irr} \mathcal{F}_\alpha, \quad  \mathcal{F}_\alpha= \{P_\alpha\} \cup \ff_\alpha, \]
where $\ff_\alpha$ is the Inou-Shishikura class of maps defined precisely in Section~\ref{Inou-Shishikura-class}. 
For $f \in \ff_\alpha$, $f'(0)=\ea$. 
We note that for each $\alpha \in \irr$ there are polynomials and rational maps of arbitrarily large degree 
whose restriction to some neighborhood of $0$ belong to $\ff_\alpha$. 
Also, $P_\alf \notin \ff_\alpha$, but $\mathcal{R}(P_\alpha)$ is defined and belongs to 
$\ff_{1/\alpha} \subset \mathcal{F}$. 


By classical results, the post-critical set of a holomorphic map provides key information about the dynamics of that 
map,  in particular, its measurable dynamics. 
A map $f \in \mathcal{F}$ has a unique critical point in its (restricted) domain of definition, say $\cp_f$. 
The \textit{post-critical} set of $f$ associated to $\cp_f$ is defined as  
\[\pc(f)= \overline{\cup_{i=1}^\infty f\co{i}(\cp_f)}.\]

The main aim of this paper is to describe the geometry of  the post-critical set and the iterates of the 
map near it. 
To this end, we build a decreasing nest of simply connected sets containing $\pc(f)$, denoted by 
$\Omega_0^n$, $n\geq 0$. 
Each $\Omega_0^n$ is formed of about $q_{n+1}+q_n$ (topological) sectors landing at $0$, which are 
ordered by the arithmetic of $\alpha$, and are mapped to one another by the map. 
Roughly speaking, the rotation element leads to a tangential action on each $\Omega_0^n$, 
while the nonlinearity of the map results in a radial action on each $\Omega_0^n$. 
The arithmetic of $\alpha$ and the non-linearities of the large iterates of $f$ characterizes 
the relative geometry of each $\Omega_0^{n+1}$ in $\Omega_0^n$, and the shapes of the sectors in each 
$\Omega_0^n$. 
See Figure~\ref{easy-renormalization}.  

A large $a_n$ in the expansion of $\alpha$ (or some $p_{n+1}/q_{n+1}$ very close to $\alpha$) 
results in $q_n$ ``relatively thick fjords'' in $\Omega_0^n \setminus \Omega_0^{n+1}$.  
By a delicate analysis of the geometry of the renormalization tower we show that the Siegel-Brjuno-Yoccoz 
arithmetic condition corresponds to the tip of the fjords reaching $0$ in the limit. 

\begin{thm}\label{pc-area}
For all $\alf\in\irr$ with $\mathcal{B}(\alpha)=\infty$ and every $f\in \mathcal{F}_\alpha$, 
$\pc(f) \setminus \{0\}$ is non-uniformly porous\footnote{A set $E\subseteq \cc$ is called non-uniformly 
porous, if there is $\lambda >0$ satisfying the following property. 
For every $z\in E$ there is a sequence of real numbers $r_n\to 0$ such that each 
ball of radius $r_n$ about $z$ contains a ball of radius $\lambda r_n$ disjoint from $E$.}. 
In particular, it has zero area.
\end{thm}

We establish a uniform contraction principle with respect to certain hyperbolic metrics on the renormalization tower.  
The map $f$ permutes the sectors in each $\Omega_0^n$ according to the rotation of angle $\alpha$. 
These are used to prove that the dynamics of $f$ on $\pc(f)$ is highly recurrent, with the combinatorics of the 
returns given by the rotation of angle $\alpha$. 

\begin{thm}\label{like-rotation-thm}
There are constants $M$ and $\mu < 1$ such that for every $\alf$ in $\irr$ and every $f \in \mathcal{F}_\alpha$, 
on the set $\pc(f)$ we have 
\[\vert f\co{q_n}(z)-z \vert\leq M \mu^n.\]  
\end{thm}


By a general result, the orbit of almost every point in the Julia set accumulates on a subset of the post-critical 
set \cite{Su83,Ly83b}. 
Thus, Theorem~\ref{pc-area} allows us to obtain the following. 

\begin{cor}\label{non-recurrent}
For all $\alf\in\irr$ with $\mathcal{B}(\alpha)=\infty$, the orbit of Lebesgue almost every point in the Julia set of 
$P_\alf$ is \text{non-recurrent}. 
In particular, there is no absolutely continuous invariant probability on the Julia set of $P_\alf$.  
\end{cor}

Let $\Delta(f)$ denote the Siegel disk of $f$ when $f$ is linearizable at $0$, and otherwise, let $\Delta(f)=\{0\}$.
Using the uniform contraction principle along the renormalization tower we show the relation 
$\cap_n \Omega_0^n = \pc(f) \cup \Delta(f)$, see Proposition~\ref{equal-to-PC}. 
This allows us to establish some topological properties of $\pc(f)$. 

\begin{thm}\label{thm:connectivity-Siegel-boundary}
For all $\alf\in \irr$ and all $f \in \mathcal{F}_\alpha $, $\pc(f)$ is a connected set. 
\end{thm}


For small perturbations of $\alpha$, 
the sets $\Omega_0^n$, up to some finite level $n$, move continuously as a function of $\alpha$. 
Using $\cap_n \Omega_0^n = \pc(f) \cup \Delta(f)$ we conclude a semi-continuity property of $\pc(f)$. 

\begin{thm}\label{thm:semicontinuous}
Let $f_\alf$, $\alf \in [0,1]$, be a continuous family of maps such that for $\alpha\in \irr$ we have 
$f_\alpha \in \mathcal{F}_\alpha$. 
Then, for every $\alpha_0 \in \irr$ and every $\epsilon>0$ there is $\delta>0$ such that for every $\alpha\in \irr$ with 
$|\alpha_0 -\alpha|< \delta$, 
$\pc(f_\alpha) \cup \Delta(f_{\alf})$ is contained in $\epsilon$-neighborhood of 
$\pc(f_{\alpha_0}) \cup \Delta(f_{\alf_0})$. 
\end{thm} 

When $\mathcal{B}(\alpha_0) < \infty$, $\pc(f_\alpha)$ may not depend continuously on $\alpha$ at $\alpha_0$, 
due to nearby non-linearizable maps.
The above theorem states that the post-critical set of the perturbed map can only explode into the 
Siegel disk of the limiting map. 
On the other hand, when $\mathcal{B}(\alpha_0)=\infty$, $\Delta(f_{\alf_0})=\{0\} \subset \pc(f_{\alpha_0})$, and 
the above theorem boils down to the continuity of $\pc(f_{\alpha})$ at $\alpha_0$. 

Theorem~\ref{thm:semicontinuous} plays a key role in constructions of examples based on successive 
small perturbations. 
A special case of the above theorem for $P_\alpha$ and when $\alpha_0$ is of bounded type 
was proved earlier by Buff and Ch\'eritat in \cite{BC12}.
That is a key step in their remarkable construction of quadratics $P_\alpha$ with positive area Julia sets.
The flexibility of the arguments presented here allows one to perturb the parameter $\alpha$ away from the real line, 
and gain control on the post-critical sets of nearby maps. 
This forms an essential part of a recent construction of Feigenbaum quadratic polynomials with 
positive area Julia sets by Avila and Lyubich reported in \cite{AvLy15}. 


\medskip

Each sector in $\Omega_0^n$, for $n\geq 1$, forms a ``fundamental domain''  for the dynamics of $f$. 
That is, the orbit of every point in $\pc(f)$ visits each such sector. 
When $f=P_\alpha$, the orbit of almost every point in the Julia set of $f$ must visit all those sectors; 
see Proposition~\ref{visit}. 
We show that in each $\Omega_0^n$ there is a sector whose diameter is bounded by a uniform constant times 
$exp(-\sum_{i\leq n} q_i^{-1} \log q_{i+1})$. 

\begin{thm} \label{acc-on-fixed}
Let $\alf\in \irr$ with $\mathcal{B}(\alpha)=\infty$ and $f\in \mathcal{F}_\alpha$. 
Then, the orbit of every point in $\pc(f)$ visits every neighborhood of $0$. 
In particular, there is no periodic point in $\pc(f)$ except $0$. 

When $f=P_\alpha$, for Lebesgue almost every $z\in \mathbb{C}$, the orbit of $z$ under $P_\alpha$ 
either tends to infinity, or visits every neighborhood of $0$ infinitely often.
\end{thm} 

As the critical orbit may never enter the linearization domain, the size of the smallest sector in each $\Omega_0^n$
provides an upper bound on the size of the Siegel disk. 

\begin{thm}\label{non-lin-0}
There exists $C> 0$ such that for every $\alf \in \irr$ and every $f \in \mathcal{F}_\alpha$ we have 
\[\operatorname{d}(\partial \Delta(f), 0) \leq C \cdot e^{-\mathcal{B}(\alpha)},\]
where $\operatorname{d} (\partial \Delta(f), 0)$ denotes the distance from $0$ to the boundary of $\Delta(f)$. 
\end{thm}

On the other hand, Yoccoz in \cite{Yoc95} proves that there is a constant $C'>0$ such that for normalized maps $f$ 
that are defined and one-to-one on $B(0,1)$, $\Delta(f)$ contains the ball of radius $C'\cdot  e^{-\mathcal{B}(\alpha)}$
about $0$.  
By an alternative (and beautiful) approach specific to the quadratic polynomials, Buff and Ch\'eritat \cite{BC04} 
had already established the bound in Theorem~\ref{non-lin-0} for the quadratic polynomials $P_\alpha$ for all  
$\alpha \in \mathbb{R}\setminus \mathbb{Q}$.

The above theorem gives a direct proof of the optimality of the Siegel-Brjuno-Yoccoz arithmetic condition in 
$\mathcal{F}$. 
However, for $\alf \in \irr$, there is a holomorphic motion of the orbit of the critical point over $\ff_\alpha$. 
Thus, the optimality of the arithmetic condition for $P_\alpha$ by Yoccoz, and the classical $\lambda$-lemma 
\cite{Ly83,MSS83}, may be used to derive the optimality of the arithmetic condition in $\mathcal{F}$.

There are points in $\pc(f)$ with dense orbits. 
When $f$ is not linearizable at $0$ (and even for some linearizable $f$) there is an abundance of non-trivial 
invariant sets in $\pc(f)$ in the form of hedgehogs introduced by Perez-Marco. 
Do those sets have a non-trivial basin of attraction in the Julia set. 
This has been answered in \cite{Ch10-II} for all rotations in $\irr$. 
That is, for Lebesgue almost every $z$ in the Julia set of $P_\alpha$, the set of accumulation points of the orbit of 
$z$ under $P_\alpha$ is equal to $\pc(P_\alpha)$. 
This provides a complete description of the topological behavior of the typical orbits of $P_\alpha$, 
modulo the topology of $\pc(P_\alpha)$. 
The topological description of $\pc(f)$ will appear in a forthcoming paper. 

\medskip

The analogue of Theorem~\ref{pc-area} when $\mathcal{B}(\alpha)< \infty$ is proved in \cite{Ch10-II}. 
The reason for the distinction is that the study of the linearizable maps requires finer (distortion) estimates on 
the changes of coordinates that were not available at the time of writing this paper. 

There has been recent advances on the dynamics of quadratic polynomials using the near-parabolic 
renormalization technique and the methods developed in this paper. 
The statistical behavior of the orbits of the maps $f\in \mathcal{F}$ is described in \cite{AC12}. 
The 1/2-h\"older continuity of a relation between the conformal radius of the Siegel disks and the Brjuno series is 
confirmed in \cite{CC13}. 
It is also employed in \cite{ChSh14} to prove the local connectivity of the Mandelbrot set on a Cantor set of parameters 
where the fine scale dynamics degenerates. 

This paper is a step towards developing a theory based on near-parabolic renormalization in order to 
provide a comprehensive description of the dynamics of holomorphic maps with a neutral fixed point. 
One hopes to eventually build a unified language to treat problems of different nature associated with such maps.

\setcounter{tocdepth}{2}
\tableofcontents
\subsection{Frequently used notations}
\begin{itemize} 
\setlength{\leftmargin}{-1em}
\item$:=$ is used when a notation appears for the first time. 
\item $\mathbb{Z}$, $\mathbb{Q}$, $\mathbb{R}$, and $\mathbb{C}$ denote the integer, rational, real, and 
complex numbers, respectively. $\RS:=\cc \cup \{\infty \}$ denotes the Riemann sphere.
\item $\Bi$ denotes the imaginary unit complex number, and $i$ is used as an 
integer index.
\item $\Re z$, $\Im z$, and $|z|$ denote the real part, the imaginary part, and the absolute value of a complex number
$z$, respectively.
\item $B(y,\delta)\subset \cc$ denotes the ball of radius $\delta$ around $y$ in the Euclidean metric, and 
$B_\delta(X):=\cup_{x\in X}B(x,\delta)$, for a given $X\subseteq\cc$.     
\item  $\diam(S)$ and  $\daron(S)$ denote the Euclidean diameter and the interior of a set $S\subset \cc$.
\item Given a map $f$, $f\co{n}$ denotes the $n$ times composition of $f$ with itself.
\item $\Dom f$, $J(f)$, and $\pc(f)$ denote the domain of definition, the Julia set, and the post-critical 
set of a map $f$, respectively.
\item Univalent map refers to a one-to-one holomorphic map.
\item Given $g\colon\!\!\Dom g\ra\cc$, with only one critical point in its domain of definition, $\cp_g$ 
and $\cv_g$ denote the critical point and the critical value of $g$, respectively. 
\item  For $x\in\mathbb{R}$, $\lfloor x \rfloor$ denotes the largest integer less than or equal to $x$.
\item Unless otherwise stated, $\arg$ denotes the principal branch of argument with values in $(-\pi, \pi]$. 
\end{itemize}   
\renewcommand{\thethm}{\thesection.\arabic{thm}}
\section{Inou-Shishikura class and near-parabolic renormalization}\label{sec:prelim} 
\subsection{Preliminary definitions} Let $f\colon U \subseteq \RS \ra \RS $ be a holomorphic map. 
Given $z \in U$, if $f(z)\in U$ we can define $f\co{2}(z):=f\circ f(z)$. 
Similarly, if $f\co{2}(z)$ also belongs to $U$, $f\co{3}(z)$ is defined, and so on. 
The \textit{orbit of $z$}, denoted by $\orb(z)$, is the sequence, $z,f(z),f\co{2}(z),\dots$ , as long as it 
is defined. 
So it may be a finite or an infinite sequence. 
Given an infinite orbit $\orb(z)$, we say that $\orb(z)$ \textit{eventually stays} in a given set $E\subset\RS$, 
if there exists an integer $k$ such that $\orb(f\co{k}(z))$ is contained in $E$.

The \textit{Fatou set} of a rational map $f:\RS\ra\RS$ is defined as the largest open set $F(f)\subseteq \RS$ 
on which the sequence of iterates $\langle f\co{n}\rangle_{n=0,1,\dots}$ forms a pre-compact family in the 
compact-open topology. 
Its complement, $J(f)$, is the \textit{Julia set} of $f$. 

The \textit{distortion} of $f$ on $U$ is defined as the supremum of $\log (\vert f'(z)/f'(w) \vert)$, for all 
$z$ and $w$ in $U$, in the spherical metric, (which may be finite or infinite). 
We frequently use the following distortion bounds due to Koebe and Grunsky, 
see \cite{Pom75} or \cite[Theorem 3.5]{Dur83}. 

\begin{thm}[Distortion Theorem]\label{T:Distortions}
Suppose that $f\colon B(0,1)\rightarrow\cc$ is a univalent map with $f(0)=0$, and $f'(0)=1$. 
At every $z\in B(0,1)$ we have
\begin{itemize}
 \setlength{\itemsep}{.3em}
\item [1)] $\frac{|z|}{(1+|z|)^2}\leq |f(z)| \leq\frac{|z|}{(1-|z|)^2},$ 
\item [2)] $\frac{1-|z|}{(1+|z|)^3}  \leq |f'(z)| \leq\frac{1+|z|}{(1-|z|)^3},$
\item [3)] $\frac{1-|z|}{1+|z|} \leq |zf'(z)/f(z)|\leq \frac{1+|z|}{1-|z|},$
\item [4)] $|\arg (zf'(z)/f(z))|\leq \log \frac{1+|z|}{1-|z|}.$ 
\end{itemize}
This implies the $1/4$-theorem: the image $f(B(0,1))$ contains $B(0,1/4)$.  
\end{thm}

Here we summarize the results of \cite{IS06} in Theorems~\ref{Ino-Shi0}, \ref{Ino-Shi1} and \ref{Ino-Shi2}, 
that we use in this paper. 
They follow from Theorem 2.1 and Main Theorems 1--3 in \cite{IS06}

\subsection{Inou-Shishikura class of maps}\label{Inou-Shishikura-class}
Consider a map $h\colon \!\!\Dom h \ra \cc$, where $\Dom h\subseteq \cc$ denotes the domain of definition (always 
assumed to be open) of $h$. 
Given a compact set $K\subset \Dom h$ and an $\eps>0$, a neighborhood of $h$ is defined as 
\[\mathcal{N}(h;K,\eps):=\{g\colon\!\! \Dom g\ra \cc \mid K\subset \Dom g , \textrm{ and } \sup_{z\in K} 
|g(z)-h(z)|<\eps\}.\]
By ``the sequence $h_n:\Dom h_n\ra \cc$ converges to $h$'' we mean that given an arbitrary neighborhood of 
$h$ defined as above, $h_n$ is contained in that neighborhood for large enough $n$. 
Note that the maps $h_n$ are not necessarily defined on the same set.

Consider the cubic polynomial 
\[P(z):=z(1+z)^2.\] 
It has a \textit{parabolic} fixed point at $0$, that is, $P'(0)=1$. 
Also, it has a critical point at $\cp_P:=-1/3$ which is mapped to the critical value at $\cv_P:=-4/27$, and another 
critical point at $-1$ which is mapped to $0$.
See Figure~\ref{F:cubic-IS}.

\begin{figure}[ht]\label{F:cubic-IS}
\begin{center}
\begin{pspicture}(8.2,8)
\epsfxsize=8cm 
\rput(4,4){\epsfbox{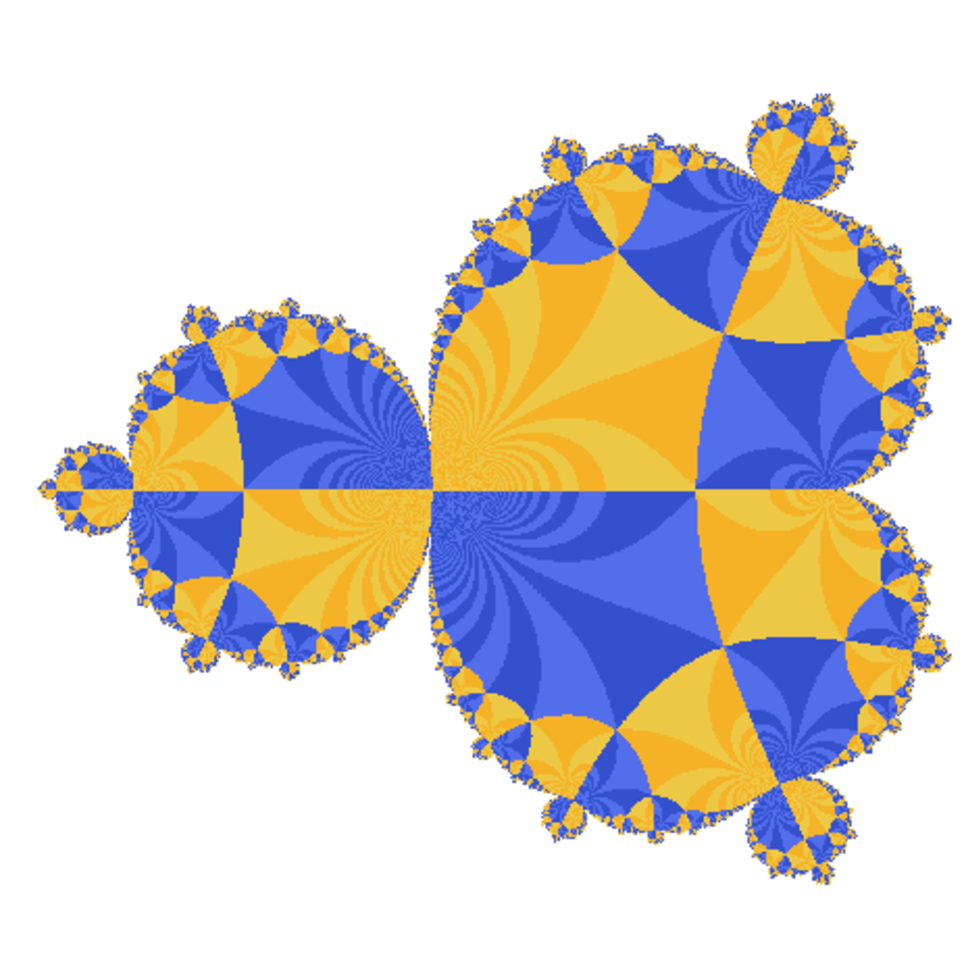}}

\psline[linewidth=.5pt](1.7,1)(3.5,3.8) 
\rput(1.7,.8){\tiny $-1$}

\psline[linewidth=.5pt](3.9,1)(5.7,3.8) 
\rput(3.9,.8){\tiny $-1/3$}

\psline[linewidth=.5pt](7.6,1)(6.45,3.8) 
\rput(7.6,.8){\tiny $-4/27$}

\psline[linewidth=.5pt](8.15,1)(7,3.8) 
\rput(8.15,.8){\tiny $0$}

\end{pspicture}
\end{center}
\caption{The dynamic plane of the polynomial $P(z)=z(1+z)^2$, and the special points $-1$, $-1/3$, $-4/27$, and $0$. 
The shades of blue are mapped by the Fatou coordinate to strips of width one in the upper half plane, while the shades of 
yellow are mapped by the Fatou coordinate to strips of width one in the the lower half pane.}
\end{figure}
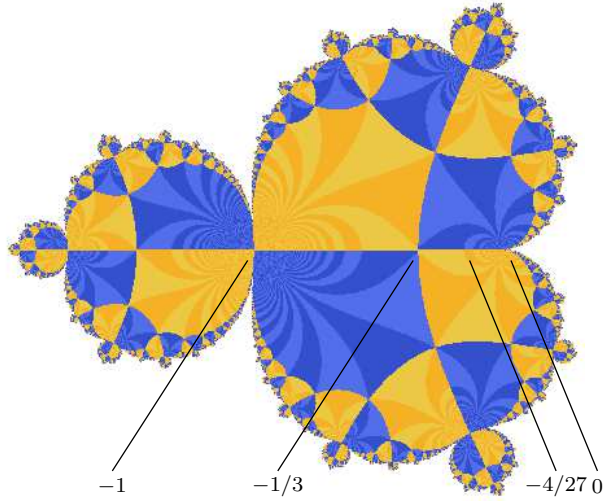

Consider the ellipse 
\[E:= \Big\{x+\Bi y\in \cc\mid (\frac{x+0.18}{1.24})^2+(\frac{y}{1.04})^2\leq 1 \Big\},\]
and let 
\begin{equation}\label{U}
U:= g(\RS \setminus E), \text{ where } g(z):=\frac{-4z}{(1+z)^2}.
\end{equation}
The domain $U$ contains $0$ and $\cp_P$, but not the other critical point of $P$ at $-1$.

Following \cite{IS06}, we define the class of maps 
\begin{displaymath}
\ff_0\!:=\Big\{f:=P\circ \vfi^{-1}\!\!:U_f \rightarrow \cc \;\Big|%
\begin{array}{l} 
\text{$\vfi\colon U \ra U_f$ is univalent, $\vfi(0)=0$, $\vfi'(0)=1$.}
\end{array}
\!\!\!\Big\}.
\end{displaymath}
Every map in this class has a parabolic fixed point at $0$ and a unique critical point at 
$\cp_f:=\vfi(-1/3)\in U_f$.

The class $\ff_0$ corresponds to the class $\mathcal{F}_1$ (and also $\mathcal{F}^P_1$) in the notations of 
\cite{IS06}. 
An extra condition on quasi-conformal extendibility of $\vfi^{-1}: U_f\to \mathbb{C}$ onto 
$\mathbb{C}$ is assumed in that paper. 
However, they have imposed this extra condition only to derive the hyperbolicity of the renormalization operator, which we 
shall introduce a moment. 
As we do not use the hyperbolicity of the renormalization operator in this paper, we have dropped the extra condition 
on the extendibility. 

\begin{thm}[Inou--Shishikura]\label{Ino-Shi0} 
For all $h\in \ff_0$ there exist a domain $\p_h \subset U_h$ and a univalent map 
$\Phi_h\colon \p_h \ra \cc$ satisfying the following:
\begin{itemize}  
\item[(1)] $\p_h$ is bounded by piece-wise analytic curves and is compactly contained in $U_h$. 
It contains $\cp_h$ and $0$ on its boundary.
\item[(2)] $\Phi_h(\p_h)= \{w\in \cc ; 0< \Re w \}$ and when $z\in \p_h\ra 0$, $|\Phi_h(z)| \ra +\infty$, 
\item[(3)] $\Phi_h(h(z))=\Phi_h(z)+1$, for all $z\in \p_h$,  
\item[(4)] the map $\Phi_h$ is unique once normalized by $\Phi_h(\cp_h)=0$. 
Moreover, the normalized map $\Phi_h$ depends continuously on $h$.     
\end{itemize}
\end{thm}

The map $\Phi_h: \p_h\to \cc$ in the above theorem is called the \textit{Fatou coordinate} of $h$. 
The existence of such coordinate for the quadratic map $z\mapsto z+z^2$ was already known to Fatou, 
see for example \cite{Sh00}.

Given $\alf\in \mathbb{R}$, let
\[\ff_\alf:=\{ z\mapsto f(\ea z): e^{-2\pi \alf\Bi}\cdot U_f \to \cc\mid f\in \ff_0\}.\]
All maps in $\ff_\alf$ have a critical value at $-4/27$. 
For the sake of simplicity of notations, we define and work with the quadratic family
\[Q_\alf(z):=\ea z+ \frac{27}{16} e^{4\pi \alf \Bi}z^2,\]
that enjoys the same normalization $\cv_{Q_\alf}=-4/27$. 
Let us combine the two classes under the notation 
\[\QIS_\alf:=\ff_\alf\cup \{Q_\alf\}.\]

The class $\cup_{\alf\in \mathbb{R}}\ff_\alf$ naturally embeds into the space of univalent maps on the unit disk with 
a neutral fixed point at $0$. 
Hence, by the distortion theorem, it is a pre-compact class in the compact-open topology.
Furthermore, it is an application of the area Theorem and the choice of $P$ and $U$ (see Main Theorem 1-a in \cite{IS06} for 
details) that 
\begin{equation}\label{E:bound-on-second-derivative}
\{|h''(0)| ; h\in \ff_0\} \subset [2,7].
\end{equation} 
 
Any map $h=f_0(\ea \cdot)\in\ff_{\alf}$ has a fixed point at $0$ with $h'(0)=\ea$. 
Moreover, if $\alf$ is small, $h$ has another fixed point $\sigma_h\neq 0$ near $0$ in $U_h$. 
The $\sigma_h$ fixed point depends continuously on $h$ and has asymptotic expansion 
$\sigma_h=-4\pi \alf \Bi/f_0''(0)+o(\alf)$, when $h$ converges to $f_0$ in a fixed neighborhood of $0$. 
Clearly $\sigma_h \ra 0$ as $\alf \ra 0$.  

\begin{thm}[Inou--Shishikura]\label{Ino-Shi1} 
There exists a constant $r_1 >0$ such that  for every map $h\colon U_h \ra \cc$ in $\QIS_\alf$ 
with $\alf \in (0,r_1]$, there exist a domain $\p_h \subset U_h$ and a univalent map $\Phi_h\colon \p_h \ra \cc$ 
satisfying the following properties:
\begin{itemize}  
\item[(1)] $\p_h$ is a simply connected region bounded by piece-wise analytic curves and is compactly contained in $U_h$. 
Also, it contains $\cp_h$, $0$, and $\sigma_h$ on its boundary.
\item[(2)] we have 
\[\Phi_h(\p_h) \supseteq \{w\in \cc ; 0< \Re w \leq 1\},\]
with $\Im \Phi_h(z) \ra +\infty$ as $z\in \p_h\ra 0$, and $\Im \Phi_h(z)\ra-\infty$ as $z \in \p_h \ra \sigma_h$.
\item[(3)] $\Phi_h$ satisfies the Abel functional equation, that is, 
\[\Phi_h(h(z))=\Phi_h(z)+1, \text{ whenever $z$ and $h(z)$ belong to $\p_h$}.\] 
\item[(4)] $\Phi_h$ is unique once normalized by $\Phi_h(\cp_h)=0$. 
Moreover, the normalized map $\Phi_h$ depends continuously on $h$.     
\end{itemize}
\end{thm}

In Section~\ref{sec:perturbed} we shall analyze the coordinates $\Phi_h$ introduced in the above theorem.
In particular, we prove the following proposition in Section~\ref{sec:geometry-petals}. 
It is frequently used in this paper. 
There is an alternative proof of this given in \cite[Proposition 12]{BC12}. 

\begin{propo}\label{P:uniformly-bounded-width-spiral} 
There exist a positive constant $r_2$ and integers $\Bk, \hat{\Bk}$ such that for all 
$h\in \QIS_\alf$ with $\alf \in (0,r_2]$, the domain $\p_h$ and the map 
$\Phi_h\colon \p_h \ra \cc$ may be chosen to satisfy the following: 
\begin{itemize}  
\item[(1)] there exists a continuous branch of argument defined on $\p_h$ such that 
\[\max_{w,w'\in \p_h} |\arg(w)-\arg(w')|\leq 2 \pi \hat{\Bk};\]
\item[(2)] $\Phi_h(\p_h)=\{w \in \cc \mid 0 < \Re(w) < \lfloor 1/\alf \rfloor-\Bk\}$.
\end{itemize}
\end{propo}


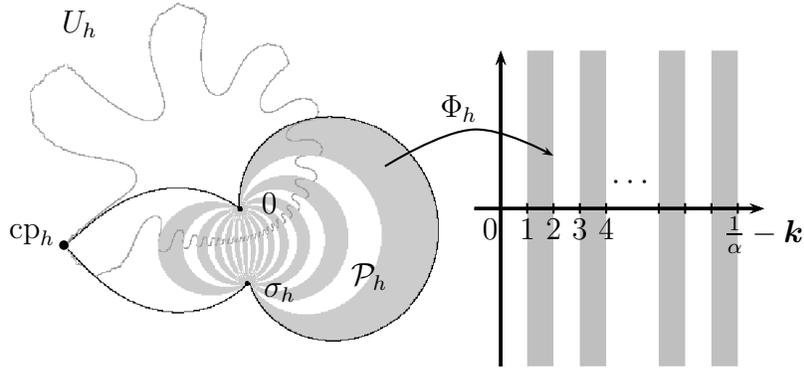
\begin{figure}[ht]
\begin{center}
  \begin{pspicture}(10.4,5.3)
\epsfxsize=5.6cm 
\rput(2.85,2.81){\epsfbox{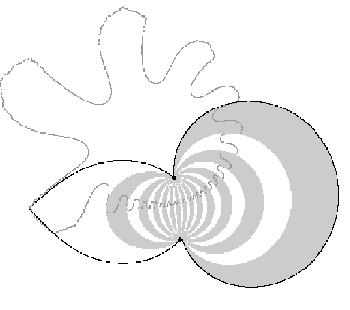}}
  \psset{xunit=.7cm}
  \psset{yunit=.7cm}
     \pspolygon*[linecolor=lightgray,fillcolor=lightgray](10.5,.5)(10.5,6.5)(11,6.5)(11,.5)
     \pspolygon*[linecolor=lightgray,fillcolor=lightgray](9.5,.5)(9.5,6.5)(10,6.5)(10,.5)
     \pspolygon*[linecolor=lightgray,fillcolor=lightgray](13,.5)(13,6.5)(13.5,6.5)(13.5,.5)
     \pspolygon*[linecolor=lightgray,fillcolor=lightgray](12,.5)(12,6.5)(12.5,6.5)(12.5,.5) 
     \psline(13.5,3.4)(13.5,3.6)
     \psline(13,3.4)(13,3.6)
     \psline(12.5,3.4)(12.5,3.6)
     \psline(12,3.4)(12,3.6)

     \psline(9.5,3.4)(9.5,3.6)
     \psline(10,3.4)(10,3.6)
     \psline(10.5,3.4)(10.5,3.6)
     \psline(11,3.4)(11,3.6)
    
     \rput(8.8,3.1){$0$}
     \rput(9.5,3.1){$1$}
     \rput(10,3.1){$2$}
     \rput(10.5,3.1){$3$}
     \rput(11,3.1){$4$}

     \rput(11.5,4){$\cdots$}
     \psaxes*[linewidth=1.2pt,labels=none,ticks=none]{->}(9,3.5)(8.5,.5)(14,6.5)
     \pscurve{->}(6.8,4.3)(8.5,5)(10,4.5)
     \rput(8.2,5.4){$\Phi_h$}
     \rput(4.6,3.6){$0$}
     \rput(4.8,1.9){$\sigma_h$}
     \psdots(.7,2.8)
     \rput(.1,3){$\text{cp}_h$}
     \rput(14,3){$\frac{1}{\alf}-\Bk$}
     \rput(6.5,2.2){$\mathcal{P}_h$}
     \rput(1,7){$U_h$}
   \end{pspicture}
\caption{A perturbed Fatou coordinate $\Phi_h$ and its domain of definition $\mathcal{P}_h$.
Similar colors are mapped on one another under $\Phi_h$. 
The gray curve (amoeba) approximates the first few iterates of $\cp_h$ under $h$.}
\label{petal}
\end{center}
\end{figure}
The map $\Phi_h: \mathcal{P}_h\to \cc$ obtained in Theorem~\ref{Ino-Shi1} is called the 
\textit{perturbed Fatou coordinate} of $h$. 
In this paper, by the \textit{perturbed Fatou coordinate} of $h$, or sometimes \textit{Fatou coordinate} of $h$ 
for short, we mean the coordinate that satisfies Proposition~\ref{P:uniformly-bounded-width-spiral} 
or Theorem~\ref{Ino-Shi0}. 
See Figure~\ref{petal}.

\subsection{Near-parabolic renormalization}\label{SS:renormalization-def} \label{sec:npr}
Let $h\colon U_h \ra \cc$ be in $\QIS_\alf$, with $\alf \in (0,r_2]$, where $r_2$ is 
the constant obtained in Proposition~\ref{P:uniformly-bounded-width-spiral}. 
Let $\Phi_h\colon \p_h \ra \cc$ denote the normalized Fatou coordinate of $h$. 
Define 
\begin{equation}\label{sector-def}
\begin{gathered}
\C_h:=\{z\in \p_h : 1/2 \leq \Re(\Phi_h(z)) \leq 3/2 \: ,\: -2< \Im \Phi_h(z) \leq 2 \}, \\
\Csh_h:=\{z\in \p_h : 1/2 \leq \Re(\Phi_h(z)) \leq 3/2 \: , \: 2\leq \Im \Phi_h(z) \}.
\end{gathered}
\end{equation}
By definition, the critical value of $h$, $\cv_h$, belongs to $\daron(\C_h)$, and $0\in \partial(\Csh_h)$. 

Assume for a moment that there exists a positive integer $k_h$, depending on $h$, with the following properties:
\begin{itemize}
\item For every integer $k\in \{1,2,\dots, k_h\}$, there exists a unique connected component of $h^{-k}(\Csh_h)$ 
which is compactly contained in $\Dom h$ and contains $0$ on its boundary. We denote this component by 
$(\Csh_h)^{-k}$. 
\item For every integer $k\in \{1,2,\dots, k_h\}$, there exists a unique connected component of $h^{-k}(\C_h)$ which has 
non-empty intersection with $(\Csh_h)^{-k}$, and is compactly contained in $\Dom h$. 
This component is denoted by $\C_h^{-k}$. 
\item The sets $\C_h^{-k_h}$ and $(\Csh_h)^{-k_h}$ are contained in 
\[\{z\in\p_h \mid  0< \Re \Phi_h(z) <\lfloor 1/\alf\rfloor -\Bk-1/2\}.\] 
\item The maps $h: \C_h^{-k}\to \C_h^{-k+1}$, for $2\leq k \leq k_h$, and $h: (\Csh_h)^{-k}\to (\Csh_h)^{-k+1}$, 
for $1\leq k \leq k_h$ are one-to-one onto. 
The map $h: \C_h^{-1}\to \C_h$ is a two-to-one branched covering.
 
\end{itemize}
Let $k_h$ denote the smallest positive integer for which the above conditions hold, and define
\[S_h:=\C_h^{-k_h}\cup(\Csh_h)^{-k_h}.\]
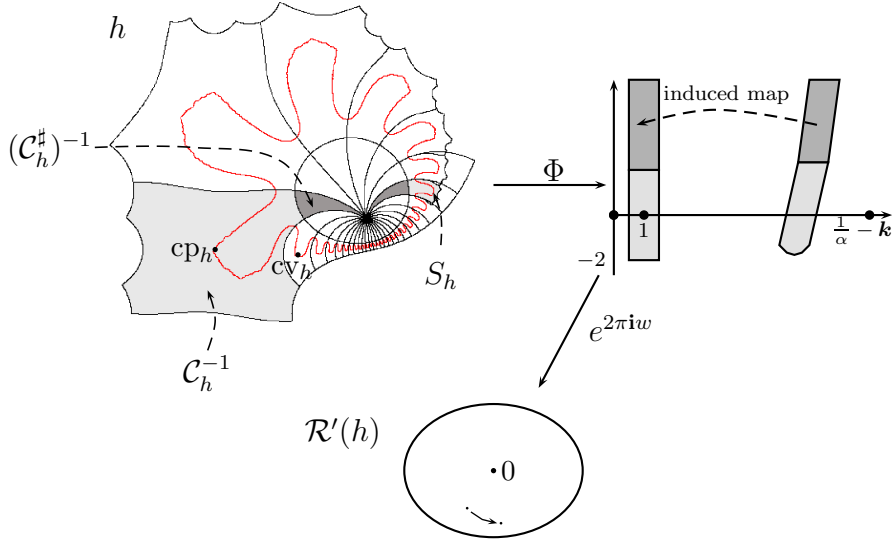
\begin{figure}[ht]
\begin{center}
 \begin{pspicture}(-.5,1.2)(11.4,9)
\epsfxsize=6.3cm
\rput(3.5,5.9){\epsfbox{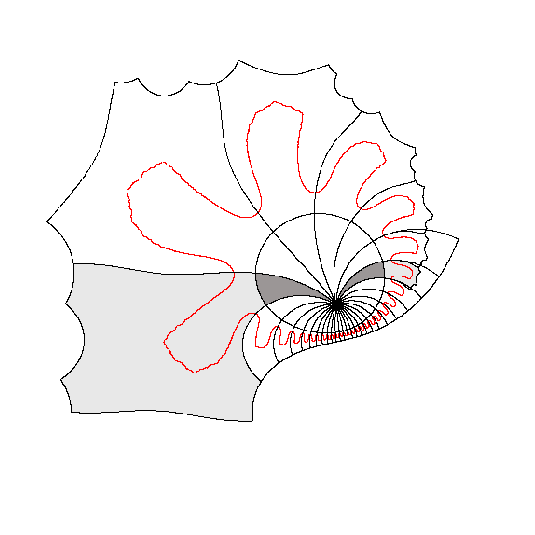}}  
  \psset{xunit=1cm}
  \psset{yunit=1cm}
    \pscurve[linewidth=.6pt,linestyle=dashed,linecolor=black]{->}(5.3,5.1)(5.3,5.5)(5.2,5.8)
    \pscurve[linewidth=.6pt,linestyle=dashed,linecolor=black]{->}(2.2,3.7)(2.3,4.1)(2.2,4.5)
    \pscurve[linewidth=.6pt,linestyle=dashed,linecolor=black]{->}(.7,6.4)(2,6.4)(3.2,6.2)(3.6,5.6)
    \rput(5.3,4.7){$S_h$}
    \rput(2.2,3.4){$\C_h^{-1}$}
    \rput(.1,6.4){$(\Csh_h)^{-1}$}
    \psdots[dotsize=2pt](2.3,5.04)(3.4,4.97)
    \rput(2.,5){\small{$\cp_h$}}
    \rput(3.33,4.77){\small{$\cv_h$}}
\newgray{Lgray}{.99}
\newgray{LLgray}{.88}
\newgray{LLLgray}{.70}
\psdots(7.6,5.5)(8,5.5)(11,5.5)
\pspolygon[fillstyle=solid,fillcolor=LLLgray](7.8,7.3)(7.8,6.1)(8.2,6.1)(8.2,7.3)
\pspolygon[fillstyle=solid,fillcolor=LLgray](7.8,6.1)(7.8,4.9)(8.2,4.9)(8.2,6.1)

\pspolygon[fillstyle=solid,fillcolor=LLLgray](10.2,7.3)(10.05,6.2)(10.45,6.2)(10.6,7.3)
\pspolygon[fillstyle=solid,fillcolor=LLgray](10.05,6.2)(9.8,5.1)(9.9,5)(10,4.97)(10.1,5)(10.2,5.1)(10.45,6.2)
\psdots(7.6,5.5)(8,5.5)
\psline{->}(7.6,5.5)(11.3,5.5)
\psline{->}(7.6,4.7)(7.6,7.3)
\rput(8,5.3){\tiny{$1$}}
\rput(10.9,5.3){\tiny{$\frac{1}{\alf}-\Bk$}}
\rput(7.3,4.9){\tiny{$-2$}}

\psline{->}(6,5.9)(7.5,5.9)
\rput(6.8,6.1){$\Phi$}

\pscurve[linestyle=dashed]{<-}(7.9,6.7)(8.9,6.9)(10.3,6.7)
\rput(9.1,7.1){\tiny{induced map}}

\psline{->}(7.4,4.7)(6.6,3.2)
\rput(7.7,4.){$e^{2\pi \Bi w}$}
\psellipse(6,2.1)(1.2,.9)
\psdots[dotsize=2pt](6,2.1)
\rput(4,2.6){$\rr' (h)$}
\rput(6.2,2.1){\small{$0$}}
\rput(1,8){$h$}
\psline[linewidth=.5pt]{->}(5.7,1.55)(5.85,1.45)(6.05,1.4)
\NormalCoor
\psdot[dotsize=1pt](5.65,1.6)
\psdot[dotsize=1pt](6.1,1.4)
 \end{pspicture}
\caption{The sets $\C_h$, $\Csh_h$,..., $\C_h^{-k_h}$, and $(\Csh_h)^{-k_h}$. 
The ``induced map'' projects via $e^{2\pi \Bi w}$ to a well defined map $\rr (h)$ on a neighborhood of $0$.}
\label{sectorpix}
\end{center}
\end{figure}

Consider the map 
\begin{equation}\label{renorm-def}
\Phi_h \circ h\co{k_h} \circ \Phi_h^{-1}:\Phi_h(S_h) \ra \cc. 
\end{equation}
By the Abel functional equation, this map commutes with the translation by one, and hence projects via 
$z=\frac{-4}{27}e^{2 \pi \Bi w}$ to a map $\rr' (h)$ defined on a set punctured at zero. 
However, it extends across zero and has the form $z \mapsto e^{2 \pi \frac{-1}{\alf}\Bi}z+ O(z^2)$ near there.
See Figure~ \ref{sectorpix}. 

The conjugate map $s\circ \rr' (h)\circ s^{-1}$, where $s(z):=\bar{z}$ denotes the complex conjugation map, 
has the form $z \mapsto e^{2 \pi \frac{1}{\alf}\Bi}z+O(z^2)$ near $0$. 
The map  $\rr (h):= s\circ \rr' (h)\circ s^{-1}$, restricted to the interior of 
$s(\frac{-4}{27}e^{2\pi \Bi(\Phi_h(S_h))})$, is called the \textit{near-parabolic renormalization} of $h$ by 
Inou and Shishikura. 
We simply refer to it as the \textit{renormalization} of $h$. 
One can see (Lemma~\ref{renorm}) that one time iterating $\rr(h)$ corresponds to several times iterating $h$, through 
the changes of coordinates. 
For some applications of closely related renormalizations (Douady-Ghys renormalization) one may  refer to 
\cite{Do86,Do94,Yoc95,Sh98} and the references therein.

It is a non-trivial task to control the shapes and the locations of the sets $\C_h^{-k}$ and ${\Csh_h}^{-k}$ for a given map in 
$\ff_\alpha$. 
This is the key content of \cite{IS06}, which is carried out using a remarkable series of estimates on univalent mappings. 
In this paper we do not use the many statements proved on the geometry and locations of these sets in that paper, 
but need the following. 

\begin{propo}\label{P:location-critical-piece}
There are positive constants $r_2'>0$ and $C < 2 \pi$ such that for every $\alpha \in (0, r_2']$ and every $h\in \QIS_\alpha$, 
\[\sup \{\arg z_1 - \arg z_2  \mid z_1, z_2 \in \mathcal{C}_h^{-1} \} \leq C,\]
for every continuous branch of argument defined on $\mathcal{C}_h^{-1}$.
\end{propo}

\begin{proof}
According to \cite{IS06}, for every $h\in \QIS_0$ the sets $\C_h^{-k}$ and ${\Csh_h}^{-k}$ are defined for all 
$k\geq 0$. 
That is, for large enough $k$ these are contained in the repelling Fatou coordinate of the map $h$ and then further 
pre-images are defined by the general properties of the Fatou coordinates. 
Comparing to their notations, $\C_h^{-1}$ is contain in the union 
\[\psi_0(D_0) \cup \psi_0(D_0') \cup \psi_0(D_{-1}) \cup \psi_0(D_{-1}''), \]  
where $\psi_0(z)=-4/z$. 
See Section 5.A--Outline of the proof. 
They prove in Proposition~5.7-(e) that the close of the set $D_0 \cup D_0' \cup D_{-1} \cup D_{-1}''$ does not intersect the 
negative real axis. 
In particular, it follows that $\sup \arg z_1/z_2 < 2\pi$, for $z_1, z_2 \in \mathcal{C}_h^{-1}$, for each $h\in \QIS_0$. 
By the pre-compactness of the class of maps $\QIS_0$, there is a constant $C'<2\pi$ such that the supremum is bounded 
from above by $C'$ over all maps $h\in \QIS_0$. 
Then, by the continuous dependence of the Fatou coordinate on the map, there are $r_2'>0$ and $C<2\pi$ satisfying the 
conclusion of the proposition. 
\end{proof}

The following theorem \cite[Main theorem 3]{IS06} states that the above definition of  
renormalization $\rr$ can be carried out for certain perturbations of maps in $\ff_0$. 
In particular, this implies the existence of $k_h$ satisfying the four properties listed in the definition of renormalization. 
There is also a detailed argument on this given in \cite[Proposition 13]{BC12} 
\footnote{The sets $\C_h^{-k}$ and $(\Csh_h)^{-k}$ defined here are (strictly) contained in the closure of the sets 
denoted by $V^{-k}$ and $W^{-k}$ in \cite{BC12}. 
The set $\Phi_h(\C_h^{-k}\cup (\Csh_h)^{-k})$ is contained in the closure of the union 
\[D^\sharp_{-k} \cup D_{-k} \cup D''_{-k} \cup D'_{-k+1} \cup D_{-k+1} \cup D^\sharp_{-k+1}\] 
in the notation used in \cite[Section 5.A]{IS06}.}.

Define 
\begin{equation}\label{V}
V:=P^{-1}(B(0,\frac{4}{27}e^{4\pi}))\setminus((-\infty,-1]\cup \overline{B})
\end{equation}
where $B$ is the component of $P^{-1}(B(0,\frac{4}{27}e^{-4\pi}))$ containing $-1$ (see Figure~\ref{poly}). 
By an explicit calculation (see \cite[Proposition 5.2]{IS06}) one can see that $\overline{U}\subset V$.   
\begin{figure}[ht]
\begin{center}
  \begin{pspicture}(8,3.2)
  \rput(4.5,1.6){\epsfbox{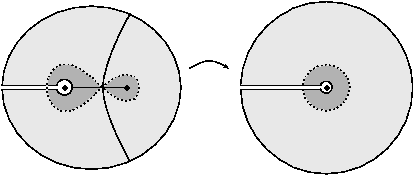}}
      \rput(6.12,1.6){{\small $\times$}}
      \rput(5.9,1.4){{\small $cv_P$}}
      \rput(6.7,1.8){{\small $0$}} 
      \rput(2.7,1.6){{\small $\times$}}
      \rput(2.7,1.3){{\small $cp_P$}}
      \rput(3.2,1.8){{\small $0$}}
      \rput(1.85,1.3){{\small $-1$}}
      \rput(4.5,2.2){{\small $P$}}
      \rput(1.1,2.7){{\small $V$}} 
\end{pspicture}
\caption{A schematic presentation of the polynomial $P$; its domain, and its range. 
Similar colors and line styles are mapped on one another.}
\label{poly}
\end{center}
\end{figure}
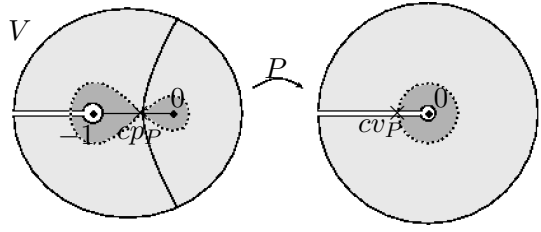

\begin{thm}[Inou-Shishikura]\label{Ino-Shi2} 
There exist a constant $r_3>0$ such that if $h \in \ff_\alf$ with $\alf \in (0, r_3]$, then $\rr(h)$ is 
well-defined and belongs to the class $\ff_{1/\alf}$, that is, $\rr(h)(z):=P \circ \psi^{-1}(e^{\frac{2\pi}{\alf}\mathbf{i}}\cdot z)$ 
for a univalent map $\psi:U\to \cc$. 
Moreover, $\psi$ extends to a univalent map on $V$.    

The same conclusion holds for the map $Q_{\alf}(z)=\ea z+\frac{27}{16}e^{4\pi \alf \mathbf{i}} z^2$. 
That is, $\rr(Q_\alf)$ is well-defined and belongs to $\ff_{1/\alf}$ provided $\alf \in (0, r_3]$.
\end{thm}

A uniform bound on $k_h$ is established in Section~\ref{sec:width-of-L_h(X)} .

\begin{propo}\label{P:turning}
There is $\Bk''\in \mathbb{N}$ such that for all $h\in \ff_\alf$ with $\alf\in (0, r_3]$, $k_h\leq \Bk''$. 
\end{propo}

Let $[0; a_1,a_2, \dots]$ denote the continued fraction expansion of $\alf$ as in the introduction. 
Define $\alf_0:=\alf$, and inductively for $i\geq1$ define the sequence of real numbers $\alf_i\in (0,1)$ as 
\[\alf_i:=1/\alf_{i-1} \pmod 1.\] 
Then each $\alf_i$ has expansion $[0; a_{i+1}, a_{i+2}, \dots]$. 
If we fix a constant $N\geq 1/r_3$, then $\alf\in \irr$ implies that $\alf_{j}\in (0,r_3]$, for $j=0,1,2,\dots$.  
We use this constant $N$ throughout the rest of this article.

Let $\alf\in \irr$ and $f_0\in \QIS_\alf$. 
Then, using Theorem~\ref{Ino-Shi2}, we may inductively define the sequence of maps 
\[f_{n+1}:=\rr (f_n): U_{f_{n+1}} \to \cc.\]
Let $U_n:=U_{f_n}$ denote the domain of definition of $f_n$, for $n\geq 0$. 
Hence, for every $n$,  
\[f_n:U_n\to \cc,\; f_n(0)=0,\; f_n'(0)= e^{2\pi \alf_n \Bi}, \;  \text{and } \cv_{f_n}=-4/27.\]
\section{Dynamically defined neighborhoods of the post-critical set}\label{sec:neighborhood}
Recall the constants $\Bk, \hat{\Bk}$ introduced in Proposition~\ref{P:uniformly-bounded-width-spiral} and the constant $N$ introduced 
at the end of the previous section.

\begin{rem}
To slightly simplify the technical details of proofs, we assume that 
\begin{equation}\label{alf*-1}
N\geq \Bk+\hat{\Bk}+2.
\end{equation}
The reason to impose this is to make $\Phi_{f_n}(\p_{f_n})$ wide enough to contain a set defined later. 
However, one can avoid this condition by extending $\Phi_{f_n}$ and $\Phi_{f_n}^{-1}$ to larger domains, using the 
dynamics of $f_n$. 
We postpone this argument to Section \ref{S:Going-Up}.
\end{rem}

\subsection{Changes of coordinates, renormalization tower}\label{sec:sectors-introduced}
For $n\geq 0$, let $\Phi_n:=\Phi_{f_n}$ denote the Fatou coordinate of $f_n\colon U_n\ra \cc$ 
defined on the set $\p_n:=\p_{f_n}$. 
For our convenience we use the notation  
\[\ex(\zeta):= \zeta \longmapsto \frac{-4}{27} s (e^{2\pi \Bi \zeta}):\cc \ra \cc^*, \text{ where } s(z)=\bar{z}.\]
By Proposition~\ref{P:uniformly-bounded-width-spiral}, Inequality~\eqref{alf*-1}, and that $\p_n$ is simply connected, 
there is an (anti-holomorphic) inverse branch 
\[\eta_n: \p_n\to \Phi_{n-1}(\p_{n-1})\] 
of $\ex$.
There may be several choices for this map but we choose one of them (for each $n$) such that  
\begin{equation}\label{WidthofLift}
\Re (\eta_n(\p_n))\subset [0,\hat{\Bk}+1] 
\end{equation}
holds, and fix this choice for the rest of this article. 
Now define 
\begin{equation}\label{E:psi-injection}
\psi_n :=\Phi_{n-1}^{-1} \circ \eta_n  \colon \p_{n}\ra \p_{n-1}.
\end{equation}
Each $\psi_n$ extends continuously to $0\in \partial \p_n$ by mapping it to $0$. 

For $n\geq 2$ we can form the compositions  
\[\Psi_n:=\psi_1 \circ \psi_2 \circ \dots \circ \psi_n\colon \p_n \ra \p_0\subset U_0.\] 

For every $n\geq 0$, let $\C_n$ and $\Csh_n$ denote the corresponding sets for $f_n$ defined in \eqref{sector-def} 
(i.e., replace $h$ by $f_n$). 
Denote by $k_n$ the smallest positive integer with  
\[S_n^0:=\C_n^{-k_n}\cup (\Csh_n)^{-k_n}\subset \{z\in \p_n \mid 0< \Re \Phi_n(z)< \lfloor 1/\alf_n \rfloor -\Bk-1/2\}.\]
By definition, the critical value of $f_n$ is contained in $f_n\co{k_n}(S_n^0)$.

For every $n\geq 0$ and $i\geq 2$, define the sectors 
\begin{gather*}
S_n^1:=\psi_{n+1}(S_{n+1}^0)\subset \p_n, \;  S_n^i:=\psi_{n+1}\circ \dots \circ\psi_{n+i}(S_{n+i}^0)\subset \p_n.
\end{gather*}
All these sectors contain $0$ on their boundaries. 
\subsection{Orbit relations on the renormalization tower}

\begin{lem}\label{renorm}
Let $z\in\p_n$ be a point with $w\!\!:=\ex\circ \Phi_n(z)\in U_{n+1}$.
There exists an integer $\ell_z$ with  $1\leq \ell_z \leq \lfloor 1/\alf_n\rfloor-\Bk-1+k_n$, 
such that
\begin{itemize}
\item the finite orbit $z,f_{n}(z),f_{n}\co{2}(z),\dots,f_{n}\co{\ell_z}(z)$ is defined, $f_{n}\co{\ell_z}(z)\in \C_n\cup \Csh_n$;
\item $\ex\circ \Phi_{n}(f_{n}\co{\ell_z}(z))= f_{n+1}(w)$;
\item if in addition $w\in f_{n+1}(U_{n+1})$, then 
\[z,f_{n}(z),f_{n}\co{2}(z),\dots,f_{n}\co{\ell_z}(z) \in \bigcup_{i=0}^{k_{n}+\lfloor 1/\alf_{n} \rfloor-\Bk-2}f_{n}\co{i}(S_{n}^0).\]
\end{itemize}
\end{lem}

\begin{proof}
As $w\in \Dom f_{n+1}$, by the definition of renormalization $\rr(f_{n})=f_{n+1}$, there are 
$\zeta \in \Phi_{n}(S^0_{n})$ and $\zeta' \in \Phi_{n}( \C_{n}\cup\Csh_{n})$, 
such that
\begin{equation*} 
\ex(\zeta)=w,\quad \ex(\zeta')=f_{n+1}(w),\text { and} \quad \zeta'=\Phi_{n}\circ f_{n}\co{k_{n}}\circ\Phi_{n}^{-1}(\zeta).
\end{equation*}
Since $\ex (\Phi_{n}(z))=w$, there exists an integer $\ell$ with 
\[\Phi_{n}(z)+\ell=\zeta \text{ and } -k_{n}+1\leq\ell\leq \lfloor 1/\alf_{n}\rfloor-\Bk-1.\]
By the Abel functional equation for $\Phi_{n}$, we have
\begin{equation*}
\zeta' =\Phi_{n}\circ f_{n}\co{k_{n}}\circ\Phi_{n}^{-1}(\zeta)
          =\Phi_{n}\circ f_{n}\co{k_{n}}\circ\Phi_{n}^{-1}(\Phi_{n}(z)+\ell)
          =\Phi_{n}\circ f_{n}\co{k_{n}+\ell}(z).
\end{equation*}
Letting $\ell_z:=k_{n}+\ell$, we have 
\begin{gather*} 
1 \leq  \ell_z \leq k_{n}+\lfloor 1/\alf_{n} \rfloor-\Bk-1, \; f_{n}\co{\ell_z}(z)=\Phi_{n}^{-1}(\zeta') \in \C_n\cup \Csh_n,\\
\ex\circ\Phi_{n}(f_{n}\co{\ell_z}(z))=\ex \circ \Phi_{n}(\Phi_{n}^{-1}(\zeta')) =\ex (\zeta') =f_{n+1}(w).
\end{gather*}
This proves the first two parts. 

For the last part, first note that by the assumption on $w$, $\Im \Phi_n(z) >-2$. 
Now, if $\ell>0$, then
\begin{gather*}
 z,f_{n}(z), \dots, f_{n}\co{(\ell-1)}(z) \in \bigcup_{i=k_{n}-1}^{k_{n}+\lfloor 1/\alf_{n} 
\rfloor-\Bk-2} f_{n}\co{i}(S_{n}^0), \;
f_{n}\co{\ell}(z), \dots,f_{n}\co{\ell_z}(z) \in \bigcup_{i=0}^{k_{n}} f_{n}\co{i}(S_{n}^0).
\end{gather*}
If $\ell\leq 0$, then
\[z,f_{n}(z), \dots,f_{n}\co{\ell_z}(z) \in \bigcup_{i=-\ell}^{k_{n}} f_{n}\co{i}(S_{n}^0).\]
\end{proof}

Define 
\[\p_n':=\{w\in\p_n\mid 0 <\Re \Phi_{n}(w) <\lfloor 1/\alf_n\rfloor-\Bk-1\}.\]

\begin{lem}\label{one-level}
For every $n\geq 1$ we have 
\begin{itemize}
 \setlength{\itemsep}{.2em}
\item[(1)]for every $w\in \p_n'$, $f_{n-1}\co{\lfloor 1/\alf_{n-1}\rfloor}%
\circ \psi_n(w)=\psi_n\circ f_n(w)$, 
\item[(2)]for every $w\in S^0_n$, $f_{n-1}\co{(k_n\lfloor 1/\alf_{n-1}\rfloor+1)}
\circ \psi_n(w)=\psi_n\circ f_n\co{k_n}(w)$. 
\end{itemize}
\end{lem}

This is summarized in the following two diagrams
\begin{align*}\label{diagram}
\xymatrix{
  \p_{n-1} \ar[rr]^{f_{n-1}\co{\lfloor 1/\alf_{n-1}\rfloor}}& & \p_{n-1} \\
  \p_n' \ar[u]^{\psi_n} \ar[rr]^{f_n}& & \p_n\ar[u]^{\psi_n}} \qquad\qquad
\xymatrix{
  \p_{n-1}\ar[rr]^{f_{n-1}\co{k_n\lfloor 1/\alf_{n-1}\rfloor+1}} & \hspace{1em} & \p_{n-1}\\
  S_n^0 \ar[u]^{\psi_n} \ar[rr]^{f_n\co{k_n}} & \hspace{1em} &\C_n\cup\Csh_n \ar[u]^{\psi_n}} 
\end{align*}

\begin{proof}
\textsl{Part (1):} The proof is given in three steps.
\medskip

{\em Step 1:} 
For every $w\in\p_n'$ there exists a positive integer $m_w$ with 
\[f_{n-1}\co{m_w}\circ \psi_n(w)=\psi_n\circ f_n(w).\] 
By the definition of renormalization $\rr f_{n-1}=f_n$, there are $\zeta\in \Phi_{n-1}(S^0_{n-1})$ and 
$\zeta'\in \Phi_{n-1}(\C_{n-1}\cup \Csh_{n-1})$ as well as integers $t_1$ and $t_2$ with 
\begin{gather*}
\zeta'=\Phi_{n-1}\circ f_{n-1}\co{k_{n-1}}\circ\Phi_{n-1}^{-1}(\zeta), 
\zeta=\eta_n(w)+t_1, \zeta'=\eta_n(f_n(w))+t_2\\
\vert t_i \vert \leq \lfloor 1/\alf_{n-1}\rfloor-\Bk, \text{ for } i=1,2. 
\end{gather*}
This implies that 
\[\eta_n(f_n(w))=\Phi_{n-1}\circ f_{n-1}\co{(k_{n-1}+t_1-t_2)}\circ \Phi_{n-1}^{-1}(\eta_n(w)).\] 
Hence, $f_{n-1}\co{m_w} \circ \psi_n(w)=\psi_n\circ f_n(w)$, for $m_w=k_{n-1}+t_1-t_2$.
\medskip

{\em Step 2:} 
$m_w$ is a constant independent of $w\in\p_n'$. 
\noindent We use the connectivity of $\p_n'$.  
For $j\in A:=\{1,2,\dots, k_{n-1}+2(\lfloor 1/\alf_{n-1}\rfloor-\Bk)\}$ set
\[X_j:=\{w\in\p_n'\mid f_{n-1}\co{j}(\psi_n(w))\text{ is defined and }f_{n-1}\co{j}\circ\psi_n(w)-\psi_n\circ f_n(w)=0\}.\]
It follows from Step $1$ that $\p_n'=\cup_{j\in A} X_j$.
Let $m$ be the smallest element of $A$ such that $\daron(X_m)$ is  non-empty. 
We claim that $S:=\cup_{j\in A, j\geq m} X_j$ is connected. 
Otherwise, $\p_n'\setminus S$ is an uncountable set contained in $\cup_{j=1}^{m-1} X_j$. 
This implies that at least one of $X_1, X_2,\dots, X_{m-1}$, say $X_i$, is uncountable, and hence has an accumulation point 
in itself. 
As the set of points where $f_{n-1}\co{i}(\psi_n(w))$ is defined is open, and $f_{n-1}\co{i}\circ\psi_n-\psi_n\circ f_n$ 
is anti-holomorphic, $\daron(X_i)$ must be non-empty.
Therefore, $S$ must be connected. 

The anti-holomorphic map $f_{n-1}\co{m}\circ\psi_n-\psi_n\circ f_n$ is defined on the connected set $S$ and 
is equal to $0$ on an open subset of $S$. 
Hence, it must be $0$ on all of $S$. 
Finally, since $\p_n'\setminus S$ is discrete, the equality holds on all of $\p_n'$.     
\medskip

{\em Step 3:} $m_w=\lfloor 1/\alf_{n-1}\rfloor$.

\noindent By virtue of Step $2$, it is enough to find the asymptotic value of $m_w$ as $w\in \p_n'$ tends to $0$.
By an analysis of the Fatou coordinates that will be carried out in Section~\ref{sec:perturbed} 
(see Equation~\eqref{E:asymptote-of-L_h^{-1}}), we shall see that for all 
$w_1, w_2\in \p_n$, $\arg (\psi_n(w_2)/\psi_n(w_1))+ \alf_{n-1}\arg(w_2/w_1)\to 0$ $\pmod{2\pi}$ 
as $w_1, w_2\to 0$, (for any branches of $\arg$ defined on $\p_n$ and $\p_{n-1}$). 
Note that the change in sign is because $\eta_n$ is anti-holomorphic.
Now, as $w\in \p_n'$ tends to $0$, $\arg (f_n(w)/w)\to 2\pi \alf_n$ $\pmod{2\pi}$. 
Hence, $\arg (\psi_n(f_n(w))/\psi_n(w))\to -2\pi \alf_n\alf_{n-1} \pmod{2\pi}$.   
On the other hand, for the irrational number $\alf_{n-1}$, $\lfloor 1/\alf_{n-1}\rfloor$ is the unique 
positive integer $j$ for which 
$\arg (f_{n-1}\co{j}(w')/w')\to -2\pi \alf_n\alf_{n-1}$ $\pmod{2\pi}$, as $w'\to 0$. 
\medskip

\textsl{Part (2):} The above steps work to prove this part as well. In step $1$, one needs to use 
Lemma~\ref{renorm} $k_n$ times. 
In step $2$, one only replaces $\p_n'$ by $S^0_n$, and uses connectivity of $S_n^0$. 
For the last step, one has $\arg (w/f_n\co{k_n}(w))\to 2\pi (1-k_n\alf_n) \pmod{2\pi}$, as $w\to 0$ in $S_n^0$.
As in the previous case, $\arg (\psi_n(w)/\psi_n(f_n\co{k_n}(w)))\to -2\pi (1-k_n\alf_n)\alf_{n-1} \pmod{2\pi}$.   
This uniquely determines the number of iterates of $f_{n-1}$ required to map 
$\psi_n(w)$ to $\psi_n(f_n\co{k_n}(w))$.   
\end{proof}

\begin{lem}\label{conjugacy}
For every $n\geq 1$ we have 
\begin{itemize}
 \setlength{\itemsep}{.2em}
\item[(1)]for every $w\in \p_n'$, $f_0\co{q_n}\circ\Psi_n(w)=\Psi_n \circ f_n(w)$,
\item[(2)]for every $w\in S^0_n$, $f_0\co{(k_nq_n+q_{n-1})}\circ\Psi_n(w)=\Psi_n\circ f_n\co{k_n}(w)$,
\item[(3)]similarly, for every $m<n$, $f_n\colon \p_n' \ra \p_n$ and $f_n\co{k_n}:S_n^0\ra (\C_n\cup\Csh_n)$ 
are conjugate to some iterates of $f_m$ on the set $ \psi_{m+1} \circ \dots \circ \psi_n(\p_n)$. 
\end{itemize}
\end{lem}

Parts ($1$) and ($2$) of the lemma are illustrated in the following diagrams 
\begin{align*}
\xymatrix{
  \p_0 \ar[r]^{f_0\co{q_n}} & \p_0 \\
  \p_n' \ar[u]^{\Psi_n} \ar[r]^{f_n} & \p_n\ar[u]^{\Psi_n}} \qquad\qquad
\xymatrix{
  \p_0 \ar[rr]^{f_0\co{(k_nq_n+q_{n-1})}} & & \p_0 \\
  S_n^0 \ar[u]^{\Psi_n} \ar[rr]^{f_n\co{k_n}} & &\C_n\cup\Csh_n\ar[u]^{\Psi_n}} 
\end{align*}

\begin{proof}
We give a proof for the first part in three steps. 
The other parts can be proved by the same arguments.
\medskip

{\em Step 1:}  
For every $w\in\p_n'$ there exists a positive integer $m_w$ with 
\[f_0\co{m_w}\circ\Psi_n(w)=\Psi_n \circ f_n(w).\] 
By Lemma~\ref{one-level}, $\psi_n(w)$ is mapped to $\psi_n(f_n(w))$ under the iterate 
$f_{n-1}\co{\lfloor 1/\alf_{n-1}\rfloor}$. 
The orbit 
\[\psi_n(w), f_{n-1}(\psi_n(w)), \dots, f_{n-1}\co{\lfloor 1/\alf_{n-1}\rfloor}(\psi_n(w))=\psi_n(f_n(w))\] 
has a subset of the form 
\begin{multline*}
\psi_n(w), f_{n-1}(\psi_n(w)), \dots, f_{n-1}\co{j}(\psi_n(w))\\ 
, f_{n-1}\co{(j+k_{n-1})}(\psi_n(w)), f_{n-1}\co{(j+k_{n-1}+1)}(\psi_n(w)), 
\dots,  f_{n-1}\co{\lfloor 1/\alf_{n-1}\rfloor}(\psi_n(w))
\end{multline*}
contained in $\p_{n-1}$, where $f_{n-1}\co{j}(\psi_n(w))\in S^0_{n-1}$. 
Using Lemma~\ref{one-level} (with $n-1$) for each consecutive pair in the above list, one concludes that 
$\psi_{n-1}(\psi_{n}(w))$ is mapped to $\psi_{n-1}(\psi_{n}(f_n(w)))$ under some iterate of $f_{n-2}$. 
By an inverse inductive argument (at levels $n-2, n-3,\dots, 1$), one concludes the claim.   
\medskip

{\em Step 2:}  $m_w$ is a constant independent of $w\in\p_n'$.

\noindent The proof in Step $2$ of the previous lemma works here as well. 
Indeed, as $f_0\co{j}\circ\Psi_n(w)$ is defined for all positive integers $j$ and $w\in \p_n'$, the 
proof is slightly easier here.   
\medskip

{\em Step 3:} $m_w=q_n$.

\noindent Similar to the proof in the previous lemma, we use the property that for every $j$ and $w_1, w_2\in \p_j$,
$\arg (\psi_j(w_2)/\psi_j(w_1))+\alf_{j-1}\arg(w_2/w_1)\to 0$ $\pmod{2\pi}$ as $w_1, w_2\to 0$. 
This will be proved in Section~\ref{sec:perturbed} (see Equation~\eqref{E:asymptote-of-L_h^{-1}}). 
Now, as $w\in \p_n'$ tends to $0$, $\arg (f_n(w)/w)$ tends to $2\pi \alf_n$ $\pmod{2\pi}$. 
Hence, 
\[\arg (\Psi_n(f_n(w))/\Psi_n(w))\to (-1)^{n}2\pi \alf_0 \cdots\alf_n \pmod{2\pi}.\]    
On the other hand, $q_n$ is the unique positive integer for which 
$\arg (f_0\co{q_n}(w')/w')\to (-1)^{n}2\pi \alf_0 \cdots\alf_n$ $\pmod{2\pi}$, as $w'\to 0$. 
\end{proof}
\subsection{A nest of neighborhoods of the post-critical set}\label{sec:neighbors}

For $n\geq 0$, define the positive integers 
\[b_n:=k_n+\lfloor 1/\alf_n\rfloor-\Bk-2,\]
and consider the union
\begin{equation}\label{union-0}
\Omega^0_n:=\cup_{i=0}^{b_n}f_n\co{i}(S^0_n) \cup \{0\}.
\end{equation}
Using Lemma~\ref{conjugacy}, we transfer the iterates in the above union to the dynamic plane of $f_0$ to obtain
\[\Omega^n_0:=\cup_{i=0}^{q_n b_n+q_{n-1}} f_0\co{i}(S^n_0) \cup \big\{0\big\}.\]
The upper bound in the above union is obtained as follows. 
The first $k_n$ iterates in \eqref{union-0} corresponds to $k_nq_n+q_{n-1}$ iterates on level $0$ by 
Lemma~\ref{conjugacy}-$2$.
The remaining $\lfloor 1/\alf_n\rfloor-\Bk-2$ iterates in \eqref{union-0} 
amounts to $q_n(\lfloor 1/\alf_n\rfloor-\Bk-2)$ iterates by Lemma~\ref{conjugacy}-1. 
The neighborhoods $\Omega_n^i$, for $i\geq 1$, may be defined accordingly. 
Using Lemma~\ref{conjugacy}, first choose the unique integer $l_{n,i}$ such that 
$f_{n+i}\co{k_{n+i}+\lfloor 1/\alf_{n+i}\rfloor -\Bk-2}$ on $S_{n+i}^0$ corresponds to $f_n\co{l_{n,i}}$ on $S_n^i$.  
Then, define 
\[\Omega_n^i:= \cup_{j=0}^{l_{n,i}} f_n\co{j} (S_n^i) \cup \{0\}.\]
See Figure~\ref{easy-renormalization}. 

\begin{figure}
\begin{center}
\begin{pspicture}(5,0)(11,4.8)

\rput(8.5,2.4){\epsfbox{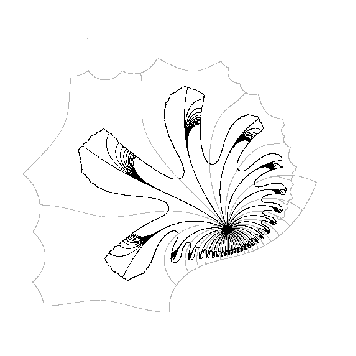}}

\pscurve[linewidth=.2pt]{->}(5.9,2.3)(5.7,1.9)(6,1.6)
\pscurve[linewidth=.2pt]{->}(8.42,.2)(8.65,.2)(8.6,.55)
\pscurve[linewidth=.2pt]{->}(8.67,.61)(8.8,.63)(8.76,.76)
\rput(9,.4){\small{$P_\alf$}}

\rput(7.5,.9){\tiny{$\times$}}

\pscurve[linewidth=.2pt]{->}(7.22,.95)(7.15,1)(7.3,1.2)
\pscurve[linewidth=.2pt]{->}(7.33,1.23)(7.38,1.32)(7.5,1.3)

\psframe[linewidth=.4pt](5.5,0)( 11.1,4.5)
\rput(7.1,1.4){\small{$f_n$}}
\end{pspicture}
\caption{
In the figure, $\Omega^0$ is the union of the sectors with gray boundaries. 
The domain $\Omega^n$ is bounded by the black curve (amoeba) and some of the sectors in $\Omega^n$ are shown.  
The map $f_n$ corresponds to the $n$-th renormalization of $P_\alf$. 
The critical point of $P_\alf$ is denoted by ``$\times$'' here.}
\label{easy-renormalization}
\end{center}
\end{figure}
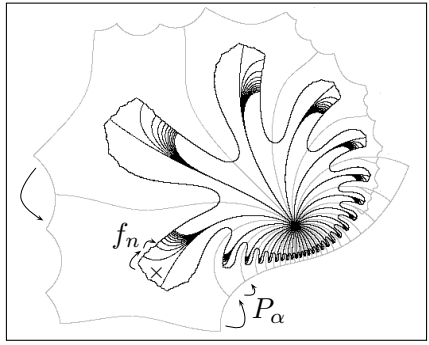

\begin{propo}\label{neighbor}
For every $f_0\in \QIS_\alf$, with $\alf\in \irr$, and every $n\geq0$, 
\begin{enumerate}
\item[(1)] $\Omega^{n+1}_0$ is compactly contained in the interior of $\Omega^n_0$;
\item[(2)] $f_0:\Omega_0^{n+1}\to \Omega_0^n$; 
\item[(3)] $\pc(f_0)$ is contained in the interior of $\Omega^n_0$.\footnote{For further properties of this 
nest see Proposition~\ref{equal-to-PC}.}
\end{enumerate}
\end{propo}

\begin{proof} 
{\em Part (1):} First we prove that $\Omega^{n+1}_0 \subset \Omega^n_0$. 
To do this, it is enough to show that for every $z \in S^{n+1}_0$ there are points $z_1,z_2,\dots, z_{m}$ in 
$S^n_0$ as well as non-negative integers $t_0,t_1,\dots,t_{m}$ with the following properties:
\begin{itemize}
 \setlength{\itemsep}{.2em}
\item[(a)] $f_0\co{t_0}(z_1)=z$,
\item[(b)] $f_0\co{t_j}(z_{j})=z_{j+1}$, for all $j=1,2,\dots, m-1$,
\item[(c)] $f_0\co{t_m}(z_m)=f_0\co{q_{n+1}b_{n+1}+q_n}(z)$,
\item[(d)] $t_j\leq q_n  b_n +q_{n-1}$, for all $j=0,1,\dots,m$. 
\end{itemize} 

Let $m:=b_{n+1}$. 
Given $z\in S^{n+1}_0$, let $\zeta:=\Psi_{n+1}^{-1}(z)\in S^0_{n+1}$ and note that $\Psi_{n}^{-1}(z)$ 
is defined and belongs to $\p_n$. First we show that there are points $\sigma_1,\sigma_2, \dots, \sigma_m$ in $S_n^0$ 
as well as positive integers $\ell_0,\ell_1,\dots, \ell_m$ such that 
\begin{itemize} 
 \setlength{\itemsep}{.2em}
\item[(a')] $f_n\co{\ell_0}(\sigma_1)= \Psi_{n}^{-1}(z)$,
\item[(b')] $f_n\co{\ell_{j}}(\sigma_j)=\sigma_{j+1}$, for all $j=1,2, \dots,m-1$,
\item[(c')] $f_n\co{\ell_m}(\sigma_m)=\psi_{n+1}(f_{n+1}\co{m}(\zeta))$,
\item[(d')] $k_n\leq\ell_j\leq b_n$, for all $j=0,1,\dots, m$.
\end{itemize}

By the definition of $S^0_{n+1}$, the iterates 
\[\zeta,f_{n+1}(\zeta),f_{n+1}\co{2}(\zeta),\dots,f_{n+1}\co{m}(\zeta)\] 
are defined and belong to $U_{n+1} \cap f_{n+1}(U_{n+1})$. 
We use Lemma~\ref{renorm} for each consecutive pair in the above 
orbit (in place of $w$ and $f_n(w)$ in that lemma) to inductively introduce the sequence 
$\sigma_1, \sigma_2,\dots, \sigma_m$ and $\ell_0,\ell_1,\dots,\ell_m$ as follows.

Lemma~\ref{renorm} applied to $\xi_1:=\psi_{n+1}(\zeta)$ produces $\xi_2\in \p_n$ and a positive 
$\ell\in \mathbb{Z}$ with 
\[\ex\circ \Phi_n(\xi_2)=f_{n+1}(\zeta), \text{  and } f_n\co{\ell}(\xi_1)=\xi_2.\] 
Now, there are  $\sigma_1\in S_n^0$ and $\ell_0\in \mathbb{Z}$ with $k_n\leq\ell_0\leq b_n$ 
such that $f_n\co{\ell_0}(\sigma_1)=\xi_1 =\Psi_n^{-1}(z)$.

As $\ex\circ \Phi_n(\xi_1)\in\Dom f_{n+1}$, we can choose a $\sigma_2\in S^0_n$ with 
$\Phi_n(\xi_1)-\Phi_n(\sigma_2)\in \mathbb{Z}$. 
Let $\ell_1$ denote the positive integer with $f_n\co{\ell_1}(\sigma_1)=\sigma_2$. 
The integer $\ell_1$ satisfies (d'). 
That is because to go from $\sigma_1$ to $\sigma_2$, one needs at least $k_n$ iterates to go from $S_n^0$ to 
$\C_n\cup\Csh_n$ and then at most $\lfloor 1/\alf_n \rfloor-\Bk-2$ iterates to reach $\sigma_2$.

Repeating the above paragraph with $\sigma_2$ which satisfies $\ex\circ \Phi_n (\sigma_2)=f_{n+1}(\zeta)$, 
one obtains $\sigma_3\in S_n^0$ and an integer $\ell_2$ with 
$k_n \leq\ell_2\leq b_n+1$, such that
$f_n\co{\ell_2}(\sigma_2)=\sigma_3$ and $\ex\circ \Phi_n(\sigma_3)=f_{n+1}\co{2}(\zeta)$. 

Repeating the above argument inductively, one obtains the sequence of pairs 
$(\sigma_4, \ell_3)$, $(\sigma_5, \ell_4)$, $\dots$, $(\sigma_{m+1}, \ell_m)$ such that
\[\ex\circ \Phi_n(\sigma_{j+1})=f_{n+1}\co{j}(\zeta),\; f_n\co{\ell_j}(\sigma_j)=\sigma_{j+1},\text{ for }j=3, 4, \dots, m.\] 
Finally, change $\ell_m$ to the positive integer $\ell$ with 
$f_n\co{\ell}(\sigma_m)= \psi_{n+1}(f_{n+1}\co{m}(\zeta))$. 
This introduces the points $\sigma_j$ and integers $\ell_j$ satisfying (a')--(d').

Now define $z_j:=\Psi_n(\sigma_j)\in S_0^n$, for $j=1,2,\dots,m$. One can see that (a')--(d') 
implies (a)--(d), using Lemma~\ref{conjugacy}. 
For example, we prove (a), (c), and the inequality for $t_0$ in (d).
\begin{align*}
z&=\Psi_n(\Psi_n^{-1}(z))\\
  &=\Psi_n(f_n\co{\ell_0}(\sigma_1))&\text{(by (a'))}\\
  &=\Psi_{n}(f_{n}\co{(\ell_0-k_{n})}\circ f_{n}\co{k_{n}}(\sigma_1))& \\
  &=f_0\co{(\ell_0-k_{n})q_{n}}\circ\Psi_{n}(f_{n}\co{k_{n}}(\sigma_1))&\text{(by Lemma~\ref{conjugacy}-$1$)}\\ 
  &=f_0\co{(\ell_0-k_{n})q_{n}}\circ f_{0}\co{(k_{n}q_{n}+q_{n-1})}(\Psi_n(\sigma_1))\qquad\qquad&\text{(by Lemma~\ref{conjugacy}-$2$)}\\
  &=f_0\co{(\ell_0 q_n + q_{n-1})}(z_1),&
\end{align*} 
Let $t_0:=\ell_0 q_n + q_{n-1}$, and note that as $\ell_0\leq b_n$, $t_0 $ satisfies the inequality in (d).

Similarly, (c) follows from the following equalities. 
\begin{align*} 
\Psi_n(\psi_{n+1}(f_{n+1}\co{m}&(\zeta))) & \\
&=\Psi_{n+1}(f_{n+1}\co{m}(\zeta))&\\
&=\Psi_{n+1}(f_{n+1}\co{(m-k_{n+1})}\circ f_{n+1}\co{k_{n+1}}(\zeta))& \\
&=f_0\co{(m-k_{n+1})q_{n+1}}\circ\Psi_{n+1}(f_{n+1}\co{k_{n+1}}(\zeta))&\text{(by Lemma~\ref{conjugacy}-$1$)}\\ 
&=f_0\co{(m-k_{n+1})q_{n+1}}\circ f_{0}\co{k_{n+1}q_{n+1}+q_n}(\Psi_{n+1}(\zeta))&\text{(by Lemma~\ref{conjugacy}-$2$)}\\
&=f_0\co{q_{n+1} b_{n+1}+q_n}(z), &\\ 
\end{align*}
and
\begin{align*} 
\Psi_n (f_n\co{\ell_m}(\sigma_m)) &=\Psi_n (f_n\co{\ell_m-k_n} \circ f_n\co{k_n}(\sigma_m)) &\\
&=f_0\co{(\ell_m-k_n)q_n} \circ \Psi_n(f_n\co{k_n}(\sigma_m)) & \text{(by Lemma~\ref{conjugacy}-$1$)} \\
&=f_0\co{(\ell_m-k_n)q_n} \circ f_0\co{k_nq_n+q_{n-1}}(z_m) & \text{(by Lemma~\ref{conjugacy}-$2$)}\\
&=f_0\co{t_m}(z_m). &  \text{(with $t_m:=\ell_m q_n+q_{n-1}$)}
\end{align*}

It remains to show that $\partial \Omega^{n+1}_0 \subset \daron (\Omega^n_0)$. 
First we claim that for all $n\geq 1$, $0\in \daron \Omega_0^n$. 

By the definition of the sectors, for every $n\geq 1$ there is $\eps_n>0$ such that for every $x_n\in B(0,\eps_n)$, there is 
$x_n'\in S_n^0$, and a non-negative integer $s_n\leq b_n -1$ 
with $f_n\co{s_n}(x_n')=x_n$. 
In particular, $B(0,\eps_n)\subset \Omega_n^0$, $\forall n\geq 0$. 
Fix $n\geq 1$. 
For $x_0$ sufficiently close to zero we may obtain a sequence of points $x_j\in B(0, \eps_j)$, $x_j'\in S_j^0$, 
non-negative integers $s_j\leq b_j-1$ such that $f_j\co{s_j}(x_j')=x_j$, $\ex \circ \Phi_j(x_j')=x_{j+1}$, 
for all $j=0,1,\dots, n-1$.
Now, by the definition of renormalization, $\Psi_n(x_n')\in S_0^n$ is mapped to $x_0$ under some iterate of $f_0$. 
To bound the number of iterates needed, let $N(s_0, s_1, \dots , s_n)$ denote the resulting number 
of iterates of $f_0$ for given $s_0, s_1, \dots, s_n$. 
By the upper bound on each $s_j$, we have 
\begin{multline*}
N(s_0, s_1, \dots, s_n) \leq N(0, s_1+1, s_2, s_3, \dots, s_n)
\leq N(0,0, s_2+1, s_3, s_4, \dots, s_n)\\
\leq  \dots \leq N(0,0,\dots,0, s_n+1)= q_n b_n+q_{n-1}
\end{multline*}      
This implies that $x_0\in \Omega_0^n$ and hence, finishes the proof of the claim.

Let $z'\neq 0$ belong to $\partial \Omega^{n+1}_0$. 
To show that $z\in \daron \Omega_0^n$ we continue to use the notations of the earlier arguments. 
There exists $z\neq 0$ in $\partial S^{n+1}_0$ with $f_0\co{t}(z)=z'$, for some \text{non-negative} $t\in \mathbb{Z}$. 
Hence, $\zeta=\Psi_{n+1}^{-1}(z)$ belongs to $\partial S^0_{n+1}$. 
On the other hand, the closure of $S^0_{n+1}$ is contained in $U_{n+1}\cap f_{n+1}(U_{n+1})$.  
But, for the point $\xi = \psi_{n+1}(\zeta)$, $\sigma_1$ may belong to the boundary of $S_n^0$ 
(i.e.\ $\xi \notin \daron S_n^0$). 
To rectify the problem, we slightly ``thicken'' the set $S_n^0$ on the left side. 
That is, there is an open set $\hat{S}_n^0$ such that, the closure of $\hat{S}_n^0$ intersects $S_n^0$, 
$f_n(\hat{S}_n^0)\subset \daron S_n^0$, $\ex\circ \Phi_n(\hat{S}_n^0)\subset f_{n+1}(U_{n+1})$, and 
$\sigma_j \in \daron (\hat{S}_n^0\cup S_n^0)$, for $j=1,2, \dots, m$.    
Now, one uses the open mapping property of holomorphic and anti-holomorphic maps to see that 
$z_i\in \daron(S_0^n\cup \Psi_{n}(\hat{S}^0_n))$, for all $i$.    
Note that since $\hat{S}_n^0 \subset \Omega_n^0$ and 
$f_{n}(\hat{S}_n^0)\subset \daron S_n^0$, $f_0\co{j}(\Psi_{n}(\hat{S}_n^0))$ is define and contained in 
$\daron (\Omega_n^0)$ for all $j$ with $0\leq j\leq q_n b_n+q_{n-1}$. 
By the open mapping property of $f_0$, this implies that those forward iterates of $z_i$  are contained in 
$\daron (\Omega_0^n)$.    
\smallskip

{\em Part (2):}
Clearly, $f(0)=0\in \Omega_0^n$.
Let $z$ be an arbitrary point in $\Omega_0^{n+1}\setminus 0$. 
By the previous part, $z\in \Omega_0^n$. 
If $z \in \Omega_0^n$ is not in the last sector $f_0\co{q_n b_n+q_{n-1}}(S_0^n)$, 
then $f_0(z)$ is defined and belongs to $\Omega_0^n$, by definition.   

Assume that $z\neq 0$ belongs to the last sector of the union $\Omega^n_0$. 
By Lemma~\ref{conjugacy}, the last sector of the union $\Omega_0^n$ is the image of the last sector in the 
union $\Omega_n^0$ under the map $\Psi_n$. 
That is, $\Psi_n^{-1}(z)$ is defined and belongs to $f_n\co{b_n}(S_n^0)\subset \mathcal{P}_n$. 
On the other hand, we claim that $z\in \Omega_0^{n+1} \cap \Psi_n(\p_n)$ implies 
$\ex \circ \Phi_n\circ \Psi_n^{-1}(z)\in \Omega_{n+1}^0\subseteq \Dom(f_{n+1})$. 
Assuming the claim for a moment, combining the two statements, we have 
$\Psi_n^{-1}(z) \in \cup_{l=0}^{k_n-1}f_n\co{l}(S_n^0)$. 
By Lemma~\ref{conjugacy}, this implies that $z\in \cup_{l=0}^{k_n-1}f_0\co{(l q_n)}(S_0^n)$. 
Then, $f_0(z)$ is defined and belongs to $\Omega_0^n$, by the definition of $\Omega_0^n$.

Now we prove the claim. 
Recall the domain $\Omega_n^1$. 
By Lemma~\ref{conjugacy}, the iterates in the union $\Omega_0^{n+1}$ are obtained from the iterates of 
$S_n^1$ on level $n$ to form $\Omega_n^1$. 
In particular, the iterates within $\Omega_0^{n+1} \cap \Psi_n(\mathcal{P}_n)$ are obtained from the 
iterates of $S_n^1$ that lie in $\mathcal{P}_n$.  
Recall that by our choice of the branch of $\psi_{n+1}$ in Section~\ref{sec:sectors-introduced}, 
$S_n^1$ is to the left of the last sector in $\Omega_n^0$. 
Therefore, if $z\in \Omega_0^{n+1} \cap \Psi_n(\p_n)$ and also $z$ belongs to the last sector in $\Omega_0^n$, 
then 
\[\Psi_n^{-1}(z)\in \bigcup_{l=0}^{b_{n+1}} \;\;
\bigcup_{j=0}^{\lfloor 1/\alf_n\rfloor -\Bk -2} f_n\co{\lfloor 1/\alf_n\rfloor l+j} (S_n^1)\cap \p_n.\] 
The set $S_n^1$, and all its consecutive iterates by $f_n$ which lie on $\mathcal{P}_n$, project 
under $\ex \circ \Phi_n$ to the set $S_{n+1}^0$. 
By the definition of renormalization, see also Lemma~\ref{renorm}, 
the iterate of $S_n^1$ that has returned to $\mathcal{P}_n$ after leaving $\mathcal{P}_n$, projects 
under $\ex \circ \Phi_n$ to the set $f_{n+1} (S_{n+1}^0)$. 
Repeating this argument, one concludes that the above union projects under $\ex \circ \Phi_n$ to the 
set $\Omega_{n+1}^0$. 

\smallskip

{\em Part (3):} 
Recall that for every $n\geq 1$, $f_n\colon S_n^0\ra f_n\co{k_n}(S_n^0)$ has a critical point. 
Thus, by Lemma~\ref{conjugacy}-$2$,  
$f_0\co{(k_nq_n+q_{n-1})}\colon  S_0^n\ra \Psi_n(f_n\co{k_n}(S_n^0))$ must also have a critical point. 
Therefore, the critical point of $f_0$ belongs to $\Omega_0^n$, for $n\geq 1$. 
On the other hand, by Part 2, $f_0$ can be iterated infinitely many times on $\cap_{n\geq 1}\Omega_0^n$, 
with values in this intersection. 
Now, the result follows from Part 1.
\end{proof}

By a lemma of Lyubich \cite{Ly83b}, for a rational map $f\colon \RS \ra \RS$, with $J(f)\neq \RS$, and any open set 
$V$ containing the closure of the orbits of the critical values of $f$, 
the orbit of Lebesgue almost every $z\in J(f)$ eventually stays in $V$. 
Combined with Proposition~\ref{neighbor}, the orbit of almost every point in the Julia set of $Q_\alf$, $\alf\in \irr$, 
eventually stays in every $\Omega_0^n$. 

\begin{propo}\label{visit}
For every $\alpha\in \irr$ and every $f_0 \in \QIS_\alpha$ we have the following. 
\begin{itemize}
\item[(1)]When $f_0=Q_\alpha$, for every $n\geq 0$, every integer $\ell$ with $0\leq \ell\leq q_n b_n+q_{n-1}$, 
and almost every $z\in J(Q_\alpha)$, $\orb(z) \cap Q_\alpha \co{\ell}(S_0^n)\neq \emptyset$. 
\item[(2)]For every $n\geq 0$, every integer $\ell$ with $0\leq \ell\leq q_n b_n+q_{n-1}$, 
and every $z \in \pc(f_0)$, $\orb(z) \cap f_0\co{\ell}(S_0^n)\neq \emptyset$.
\end{itemize}
\end{propo}

\begin{proof}
Let $f_0=Q_\alpha$. 
It is enough to prove the proposition for $\ell=0$. 
We claim that for every $n\geq0$, 
\[\left \{z\in J(f_0)\mid \orb(z)\cap \Omega_0^{n+2}\neq \emptyset \right\} \subseteq 
\left \{z\in J(f_0)\mid \orb(z)\cap S_0^n\neq \emptyset \right\}.\]

Assuming the claim for a moment, as the left hand set has full Lebesgue measure by the above paragraph and 
Proposition~\ref{neighbor}, we can conclude Part 1 of the proposition for $\ell=0$.

To prove the claim, let $z$ be an arbitrary point in $J$ with $f_0\co{t_1}(z)\in \Omega^{n+2}_0$ for some 
integer $t_1\geq 0$. 
Choose $t_2\geq t_1$ with $f_0\co{t_2}(z)$ in the last sector $f_0\co{j}(S_0^{n+2})$, 
with $j=q_{n+2} b_{n+2}+q_{n+1}$.
The point $\zeta:=\Psi_{n+2}^{-1}(f_0\co{t_2}(z))\in \p_{n+2}$, and hence $f_{n+2}(\zeta)$ is defined. 
By Lemma~\ref{renorm}, this implies that $\zeta':=\psi_{n+2}(\zeta)$ can be iterated at least two times under $f_{n+1}$. 
That is, $\zeta'$ and $f_{n+1}(\zeta')$ belong to $U_{n+1}$. 
Now, Lemma~\ref{renorm} applied to $\zeta'':=\psi_{n+1}(\zeta')$ implies that there is an orbit 
$\zeta'', f_n(\zeta''), \dots , f_n\co{\ell}(\zeta'')$, with $\ex \circ \Phi_n(f_n\co{\ell}(\zeta''))=f_{n+1}(\zeta')\in U_{n+1}$. 
This implies that there exists a positive integer $\ell'$ with $f_n\co{\ell'}(\zeta'')\in S_n^0$. 
Now, using Lemma~\ref{conjugacy}, $f_0\co{\ell''}(\Psi_n(\zeta''))=f_0\co{\ell''}(z)\in S_0^n$, 
for some positive integer $\ell''$.  
This finishes the proof of the first part. 

By Proposition~\ref{neighbor}, $\pc(f_0)$ is contained in the left hand set in the above equation.    
This implies Part 2 of the lemma. 
Indeed, the proof of Part 2 of the proposition is already present in the proof of Proposition~\ref{neighbor}. 
\end{proof}
\section{Upper bound on  the sizes of linearization domains}\label{sec:acc}
\subsection{Approaches of the critical orbit to the fixed point}
In this section, we estimate the size of a sector (roughly the smallest one) in each union $\Omega^n_0$ in terms of a  
partial sum of the Brjuno series introduced in the Introduction. 
The main technical tool is stated in the next two lemmas. 
They will be proved in Section~\ref{sec:geometry-petals}, once we establish some estimates on the 
Fatou coordinates.

Let $f$ be a map in $\QIS_\alf$, with $\alf \in (0, r_3]$. 
Recall the domain $\p_f$ defined in Proposition~\ref{Ino-Shi1},  the constant $\Bk$ in 
Proposition~\ref{P:uniformly-bounded-width-spiral}, as well as the sector $S_f$ and the constant $k_f$ defined in 
Section~\ref{sec:npr}. 
Moreover, if the rotation of $\rr(f)$ at $0$ belongs to $(0, r_3]$, then $\p_{\rr(f)}$, $\Phi_{\rr(f)}$, and $\psi_{\rr (f)}$ 
are also defined, where $\psi_{\rr(f)}: \p_{\rr(f)} \to \p_f$ is the change of coordinate defined in Section~\ref{renorm}. 

\begin{propo}\label{P:smallsector} 
There is $M_1\geq 1$ such that for all $\alf\in (0, r_3]$ and all $f\in \QIS_\alf$, there exists 
$\eta(f)$ in the set $\{k_f, k_f+1, \dots, \lfloor 1/(2\alf)\rfloor+k_f\}$ 
such that
\[\diam (f\co{\eta(f)}(S_f)) \leq M_1 \alf,  \text{ and } f\co{\eta(f)}(S_f)\subseteq \p_f.\]
\end{propo}

Recall the constant $N$ defined at the end of Section~\ref{Inou-Shishikura-class}. 
That is, $\alf\in \irr$ guarantees that every $h\in \QIS_\alf$ is infinitely near-parabolic renormalizable. 

\begin{propo}\label{P:change of coord}
There is $M_2\geq 1$ such that for all $\alf\in \irr$ and all $f\in \QIS_\alf$, there exists 
$\kappa(f)$ in the set $\{0,1, \dots , \lfloor 1/(2\alf)\rfloor\}$ such that 
\begin{itemize}
\item[(1)] $f\co{ \kappa(f)}\circ \psi_{\rr(f)}(\p_{\rr(f)})\subseteq \p_f$, 
\item[(2)] $\forall w\in \p_{\rr(f)}$, $|f\co{\kappa(f)}\circ \psi_{\rr(f)}(w)|\leq M_2\alf | w | ^\alf.$
\end{itemize}
\end{propo}

Assume $\alf \in \irr$ and $f_0\in \QIS_\alf$. 
By Theorem~\ref{Ino-Shi2}, the sequence of renormalizations $f_n= \rr\co{n}(f_0)$ and rotations 
$ \alf_n$ are defined for $n \geq 0$. 
In particular, we have the petals $\p_n$, the Fatou coordinates $\Phi_n$, the lift maps $\psi_n: \p_n \to \p_{n-1}$, 
and the sectors $S_n^i$ for each $f_n$. 
The latter are defined in Section~\ref{renorm}.
Applying the above Propositions to the maps $f_n$, we obtain the integers $\eta(n)= \eta(f_n)$ 
and $\kappa(n)= \kappa(f_n)$, for $n\geq 0$. 
Recall the integers $b_n$ defined in Section~\ref{sec:neighbors}.
 
\begin{propo}\label{sectorsize}
There is $M_3\in \mathbb{R}$ such that for all $\alf\in \irr$, all $f_0\in \QIS_\alf$, and all $m\geq 1$,
 there exist a non-negative integer $\nu(m)\leq q_m b_m+q_{m-1}$ with
\[\mathrm{diam} (f_0\co{\nu(m)}(S_0^m)) \leq M_3\cdot \alf_0\cdot \alf_1^{\alf_0}\cdot \alf_2^{\alf_0\alf_1}%
\cdot \alf_3^{\alf_0\alf_1\alf_2}\dots \alf_m^{\alf_0 \dots \alf_{m-1}}.\]
\end{propo}

\begin{proof}
Let $M$ be the maximum of the constants $M_1$ and $M_2$ obtain in the above two Propositions.
Given $m\geq1$, by Proposition~\ref{P:smallsector},  
\[\diam(f_m\co{\eta(m)}(S^0_m)) \leq M\cdot\alf_m,\;  f_m\co{\eta(m)}(S^0_m)\subset \p_m .\]  
Using Proposition~\ref{P:change of coord} with $n=m-1$ and $w\in f_m\co{\eta(m)}(S^0_m)$, we obtain
\begin{align*}
\diam(f_{m-1}\co{\kappa(m-1)} \circ \psi_m(f_m\co{\eta(m)}(S^0_m)))&\leq M\cdot \alf_{m-1}(\diam(f_m\co{\eta(m)}(S^0_m)))^{\alf_{m-1}}\\
&\leq M\cdot\alf_{m-1}\cdot(M\cdot\alf_m)^{\alf_{m-1}}.
\end{align*}
By Lemma~\ref{one-level}, the above relation boils down to   
\[\diam(f_{m-1}\co{\kappa(m-1)} \circ f_{m-1}\co{(\eta(m)\lfloor 1/\alf_m\rfloor +1)}(\psi_m(S^0_m)) \leq M\cdot\alf_{m-1}\cdot(M\cdot\alf_m)^{\alf_{m-1}},\] 
which is equivalent to 
\[\diam(f_{m-1}\co{(\kappa(m-1)+\eta(m)\lfloor 1/\alf_m\rfloor+1)}(S^1_{m-1}) \leq M\cdot\alf_{m-1}\cdot(M\cdot\alf_m)^{\alf_{m-1}}.\]
Again applying Proposition~\ref{P:change of coord} with $n=m-2$, the last inequality implies that 
\begin{multline*}
\diam(f_{m-2}\co{\kappa(m-2)}\circ \psi_{m-1}(f_{m-1}\co{(\kappa(m-1)+\eta(m)\lfloor 1/\alf_m\rfloor+1)}(S^1_{m-1})) \\
\quad\leq M\cdot\alf_{m-2}\cdot \big(M\cdot\alf_{m-1}\cdot(M\cdot\alf_m)^{\alf_{m-1}}\big)^{\alf_{m-2}},
\end{multline*}
which, by Lemma~\ref{conjugacy}, gives us  
\begin{multline*}
\diam(f_{m-2}\co{(\kappa(m-2)+(\kappa(m-1)+\eta(m)\lfloor1/\alf_m\rfloor+1)\lfloor1/\alf_{m-1}\rfloor+1)}(\psi_{m-1}(S^1_{m-1})))\\
\quad\leq M\cdot\alf_{m-2}\cdot(M\cdot\alf_{m-1}\cdot(M\cdot\alf_m)^{\alf_{m-1}})^{\alf_{m-2}}.
\end{multline*}

Inductively, repeating Proposition~\ref{P:change of coord} with $m-3, m-4,\dots,0$, one obtains
\begin{align*}
\diam &(f_0\co{\nu(m)}(S^{m}_0)) &\\
&\leq M\cdot\alf_0\cdot[M\cdot\alf_1[M\cdot\alf_2[\dots[M\cdot\alf_m]^{\alf_{m-1}}]^{\alf_{m-2}}\dots]^{\alf_1}]^{\alf_0}\\
&\leq M^{1+\alf_0 +\alf_0\alf_1+\cdots +\alf_0\alf_1 \cdots \alf_{m-1}}\cdot  \alf_0\cdot \alf_1^{\alf_0}\cdot \alf_2^{\alf_0\alf_1}\cdot \alf_3^{\alf_0\alf_1\alf_2}\dots \alf_m^{\alf_0 \dots \alf_{m-1}} \\
&\leq M^4\cdot \alf_0\cdot \alf_1^{\alf_0}\cdot \alf_2^{\alf_0\alf_1}\cdot \alf_3^{\alf_0\alf_1\alf_2}\dots \alf_m^{\alf_0 \dots  \alf_{m-1}} 
\end{align*}
for some integer $\nu(m)$.  
Here we have used that $\alf_i\alf_{i+1} \leq 1/2$, for $i\geq 0$.
This finishes the proof of the estimate. 

The bound on $\nu(m)$ follows from the upper bounds on $\eta(j)$ and $\kappa(j)$ 
(see the discussion on $N(s_0, s_1, \dots, s_n)$ in the proof of Proposition~\ref{neighbor}-1).  
\end{proof}

Let $\beta_{-1}:=1$, and $\beta_n:=\Pi_{j=0}^{n}\alf_j$, for $n\geq 0$. 
Using elementary properties of continued fractions one can show that (see \cite[Section~1.5]{Yoc95} for further details)
\begin{equation}\label{equivalence} 
\Big |\sum_{j=0}^\infty \beta_{j-1} \log \alf_{j}^{-1} - \sum_{n=0}^{\infty} \frac{\log q_{n+1}}{q_n}\Big |\leq C
\end{equation}
for some constant $C$ independent of $\alf_0\in (0,1)$.

\begin{proof}[Proof of Theorem~\ref{acc-on-fixed}] 
The proof is immediate using Proposition~\ref{visit}, Proposition~\ref{sectorsize}, and the uniform bound in 
Equation~\eqref{equivalence}.    
\end{proof}

\begin{thm}\label{non-lin}
There exists a constant $M$ such that for every $\alf \in \irr$ and every $f\in \QIS_\alf$, the conformal radius of 
the Siegel disk centered at $0$ is bounded from above by $M\exp (- \sum_{n=0}^{\infty} q_n^{-1} \log q_{n+1})$.   
\end{thm}

\begin{proof}
Recall that each $U_n=\Dom f_n$ contains a non-zero fixed point $\sigma_{f_n}$. 
By Lemma~\ref{renorm}, this fixed point lifts to a periodic point of $f_{n-1}$, whose orbit crosses 
the set $S_{n-1}^0$. 
Then by the conjugacy relations in Lemma~\ref{conjugacy} this periodic point is sent by $\Psi_{n-1}$ to a periodic point of $f_0$ 
whose orbit must cross $S_0^{n-1}=\Psi_{n-1}(S_{n-1}^0)$. 
Hence, every sector in the union $\Omega_0^{n-1}$ contains at least a point of that cycle.
Now the theorem follows from the 1/4-Theorem, Proposition~\ref{sectorsize}, and Equation~\eqref{equivalence}. 
\end{proof}

\begin{rem}
In \cite{AC12} we prove a stronger version of Proposition~\ref{sectorsize}, which is based on an infinitesimal 
estimate on the Fatou coordinates established in \cite{Ch10-II}. 
It is proved that given any neighborhood of the Siegel disk (or zero), as $n\to \infty$, the density of the number of sectors 
in $\Omega_0^n$ which are contained in that neighborhood tends to one. 
Although not all sectors are necessarily contained in such neighborhoods, surprisingly, it is also proved in \cite{AC12} that every neighborhood of the Siegel disk contains the orbit of infinitely many periodic points. 
\end{rem}

There is a large class of analytic maps of $\cc$ or $\RS$ that have a restriction which belong to the Inou-Shishikura class. 
Thus the above results apply to these maps as well. 
Here is a simple example. 
Recall the domain $U$ in \eqref{U}. 
Let $h$ be a rational map of the Riemann sphere that $h(0)=0$, $h'(0)=1$, and $h$ is univalent on the connected component of 
$h^{-1}(U)$ containing $0$.
Then the map $h\cdot(1+h)^2$ belongs to $\ff_0$.
Note that such maps may have arbitrarily large degrees.
Pre-composing these maps with rotations of angle $\alf\in \irr$, one has the bound on the conformal radius of their 
Siegel disk in the theorem, and in particular, the optimality of the Bruno condition for their linearizability.  
\section{Measure and topology of the attractor}\label{sec:measure}
Let $\alf \in \irr$ and $f_0\in \QIS_\alpha$. 
By Theorem~\ref{Ino-Shi2} in Section~\ref{SS:renormalization-def} the sequence of renormalizations 
$f_n= \rr\co{n}(f_0)$, $n\geq 0$, are defined. 
One forms the domains $\Omega_0^n$ and $\Omega_n^0$ for the map $f_0$, defined in Section~\ref{sec:neighborhood}. 
In this section we prove Theorem~\ref{pc-area}. 
The plan is to show that $\cap_{n=0}^\infty\Omega^n_0$, which contains the post-critical set by Proposition~\ref{neighbor}, 
does not contain any Lebesgue density point. 
As the proof spans over several pages, we briefly outline the argument in the next paragraph. 

In Subsection~\ref{S:Going-Down} we show that any point $z_0$ in $\cap_n \Omega_0^n$ can be 
mapped to arbitrarily deep levels of the renormalization planes using the changes of coordinates. 
Let $z_n$, for $n\geq 1$, denote the point obtained on level $n$ in this process. 
In Proposition~\ref{P:good-height-lem} we show that there are infinitely many levels $n$ with $|z_n|\geq \alf_n$.
In Proposition~\ref{P:r-ball}, we state that if at some level we have $|z_n|\geq \alf_n$, 
then there exists a ball of size comparable to its distance to $z_n$ in the complement of $\pc(f_n)$. 
In Subsection~\ref{S:Going-Up} we define holomorphic maps $g_i$ from an appropriate subset $V_i$ of the 
$i$-th renormalization level to a domain $V_{i-1}$ on level $i-1$.  
The maps $g_i$, for $i=n,n-1,\dots,1$ belong to a compact class of maps and $z_i \in V_i$ is mapped to $z_{i-1}\in V_{i-1}$ 
under $g_i$. 
In Lemma~\ref{contract} we show that each $g_i$ is uniformly contracting in the respective hyperbolic metrics, and in 
Lemma~\ref{safe-steps} we show that each $g_i$ is univalent on a ball of definite hyperbolic size (independent of 
$i$ and $n$) about $z_i$. 
The composition of these maps (from level $n$ to level $0$) sends the complimentary ball obtained in Proposition~\ref{P:r-ball} to 
the dynamic plane of $f_0$. 
By uniform contraction of the maps $g_i$, after first few iterates the image of the ball shrinks and falls in the 
neighborhood of some $z_j$ where $g_j$ is univalent, and stays in the balls where the further maps $g_i$, 
for $i=j-1,j-2,\dots,1$, are univalent.
Then, we use compactness of the class of maps containing $g_i$ and the distortion theorem to show that this 
composition provides us with a ball in the complement of $\pc(f_0)$ at a small scale near $z_0$.

\begin{figure}[ht]
\begin{center}
\includegraphics[scale=.7]{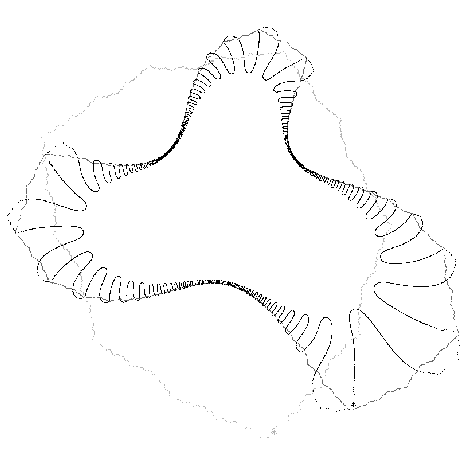}
\caption{The three curves in different colors approximate the orbit of the critical points for different
values of $\alf$. 
The light gray one is for $\alf=[0;3,1,1,1,\dots]$, the gray one for $[0;3,50,1,1,1,\dots]$, and the dark gray one 
for $[0;3,50,10^5,1,1,1,\dots]$.}
\label{post}
\end{center}
\end{figure}

\subsection{Balls in the complement at deep levels}\label{balls-in-complement}
Given $X\subseteq \cc$, let $B_\delta(X):=\cup_{x\in X} B(x, \delta)$.  

\begin{propo}\label{P:r-ball}
For all $E\in \mathbb{R}$ there are positive constants $\delta_1$, $\delta_2$, and $r^*$ satisfying the following. 
For every $\alf\in (0, r_2]$, every $f\in \QIS_\alf$, and every $\zeta\in \cc$ with 
$\Im\zeta\leq\frac{1}{2\pi}\log \alf^{-1}+E$ and $\ex(\zeta)\in \Omega^0_0(f)$, 
there exists a curve $\gamma:[0,1]\ra~\cc$, with $\gamma(0)=\zeta$, such that

\smallskip

\begin{itemize}
  \setlength{\itemsep}{.4em}
\item[(1)] $\ex\big(B_{\delta_1}\big(B(\gamma(1), r^*)\cup \gamma[0,1]\big)\big) \subseteq \Dom f \setminus\{0\}$,
\item[(2)]$\ex\big(B(\gamma(1), r^*)\big) \cap \Omega^0_0(f)=\emptyset$, 
$f\big(\ex(B(\gamma(1), r^*))\big) \cap \Omega^0_0(f)=\emptyset,$
\item[(3)] $\diam \Re \big(B_{\delta_1}\big(B(\gamma(1), r^*)\cup \gamma[0,1]\big)\Big)\leq 1-\delta_1$,
\item[(4)] $\bmod\; B_{\delta_1}\big(B(\gamma(1), r^*)\cup 
\gamma[0,1]\big)\setminus\big( B(\gamma(1), r^*)\cup \gamma[0,1]\big) \geq \delta_2$. \footnote[1]{mod denotes 
the conformal modulus of an annulus.}
\end{itemize}
\end{propo}

The proof of the above proposition appears in Section~\ref{sec:geometry-petals}. 
See Figure~\ref{complementary-Balls}. 

Recall the sets $\C_h$ (and $\C_n$ for $f_n$) introduced for the definition of renormalization. 

\begin{lem}\label{the-balls}
There exists a real constant $\delta_3< \min\{\delta_1, 1/8\}$ such that 
\begin{itemize}
\item $\forall j \in \mathbb{Z}$, $\forall n\in \mathbb{N}$, $\ex(B(j,\delta_3))\subset \daron (\C_n) \subset \Omega^0_n$,
\item $\forall n\in \mathbb{N}$, $\forall \xi\in\cc$ with $\ex(\xi)\in \Omega^0_n$, we have 
$\ex (B(\xi,\delta_3))\subset \Dom f_n$.
\end{itemize}
\end{lem}

\begin{proof} 
As each set $\C^{-i}\cup (\Csh)^{-i}$, for $i=0,1,2,\dots, k_n$, is compactly contained in  $\Dom f_n$, $\Omega_n$ is compactly 
contained in $\Dom f_n$. 
Therefore, it follows from continuous dependence of the Fatou coordinate on the map, the pre-compactness of $\ff_0$, and 
the uniform bound in Proposition~\ref{P:turning} that there exists a real constant $\delta>0$ such that
\begin{gather}\label{well-contained}
\forall n\geq 1, B(-4/27,\delta) \subset \C_n \text{ and } B_{\delta}(\Omega^0_n)\subset \Dom f_n.
\end{gather}
The first inclusion implies the first part of the lemma and the second one implies the second part of the lemma. 
\end{proof}

\subsection{Going down the renormalization tower}\label{S:Going-Down}
For every $n\geq 1$, let $\textrm{Fil}(\Omega^0_n)$ denote the set obtained from adding the bounded components of 
$\cc\setminus \Omega_n^0$ to $\Omega_n^0$, if there is any. 
For $n\geq 1$ and $j=0, 1, \dots, \lfloor \alf_n^{-1}\rfloor-\Bk-1$, let $I_{n,j}$ denote the closure of the connected component of 
\[\daron\big ( \textrm{Fil}(\Omega_n^0)\big ) \cap \Phi_n^{-1}\{j+\tfrac{1}{2}+t\Bi : t\in \mathbb{R}\}\]
landing at $0$. 
Each $I_{n,j}$ is a smooth curve in $\textrm{Fil}(\Omega^0_n)$ that connects the boundary of $\Omega^0_n$ to $0$. 
For every such $n$ and $j$, every closed loop (i.e.~homeomorphic image of a circle) contained in 
$\Omega^0_n\setminus I_{n,j}$ is contractible in $\cc\setminus\{0\}$. 
This implies that there is a continuous inverse branch of $\ex$ defined on every $\Omega^0_n\setminus I_{n,j}$.

By Proposition~\ref{P:uniformly-bounded-width-spiral}, Proposition~\ref{P:turning}, and the pre-compactness of 
$\cup_{\alf\in (0, r_3]}\ff_\alf$, there exists a 
positive integer $\Bk'$ such that 
\begin{equation}\label{k'}
\forall n\geq 1 \text{ and } \forall j \text{ with } 0\leq j < \tfrac{1}{\alf_n}-\Bk-1, 
\quad \sup_{z, z' \in \Omega^0_n \setminus I_{n,j}} |\arg(z)- \arg (z')| \leq 2\pi \Bk',
\end{equation}
for every continuous branch of argument defined on $\Omega^0_n \setminus I_{n,j}$.  
To simplify the technical details of the proof in this section, we assume the following condition on the rotations
\begin{equation}\label{k'-condition}
   N\geq 2\Bk'+\Bk+\Bk''+1.
\end{equation}

Fix an arbitrary point $z_0 \in \cap^{\infty}_{n=0}\Omega^n_0\setminus \{0\}$. 
We associate a sequence of quadruples 
\begin{align}\label{quadruples}
\langle (z_i,w_i,\zeta_i,\sigma(i))\rangle_{i=0}^\infty
\end{align}
to $z_0$, where $z_i, w_i \in \Dom f_i$, $\zeta_i \in \Phi_i(\p_i)$, and $\sigma(i)$ 
is a non-negative integer. 
This sequence shall be the trace of $z_0$ while going down the renormalization tower, and will be used to transport the 
complementary balls on level $n$, introduced in Proposition~\ref{P:r-ball}, back to the dynamic plane of $f_0$.
It is inductively defined as follows.  

Define the sets 
\begin{gather*}
\mathscr{A}_n:= \{z\in \p_n\mid \Re \Phi_n(z)\in [\Bk'+1/2,\lfloor 1/\alf_n\rfloor -\Bk], \text{ or } 
\Phi_n(z)\in \cup_{j=1}^{\Bk'}B(j,\delta_3)\}\\
 \mathscr{B}_n:= \Omega_n^0\setminus \mathscr{A}_n.
\end{gather*}

\begin{figure}
\begin{center}
\begin{pspicture}(0,1.6)(10,7)
\epsfxsize=6cm
\rput(3,4){\epsfbox{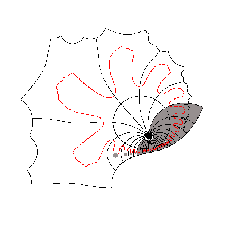}}
\psdot[dotsize=.9pt](1.8,5)
\rput(1.8,5.2){$z_0$}
\psdot[dotsize=.9pt](4.6,4.2)
\rput(4.8,4.2){$w_0$}

\psline[linearc=.04, linewidth=.3pt]{->}(4.6,4.3)(4.4,4.6)
\psline[linearc=.04, linewidth=.3pt]{->}(4.35,4.65)(4.,5.15)
\psline[linearc=.04, linewidth=.3pt]{->}(3.9,5.2)(3.2,5.4)
\psline[linearc=.04, linewidth=.3pt]{->}(3.1,5.4)(1.9,5)
\psline[linewidth=.7pt]{->}(5.,3.8)(6.5,3.8)
\rput(6,4.1){$\phi_0$}
\pspolygon[fillstyle=solid,fillcolor=LLLgray,linecolor=LLLgray](8.25,3.2)(9.6,3.2)(9.6,5.3)(8.25,5.3)
\psline[linewidth=.7pt]{->}(6.6,3.4)(9.8,3.4)
\psline[linewidth=.7pt]{->}(7.2,3.)(7.2,5.5)
\rput(8,5){$\zeta_0$}
\psdot[dotsize=1pt](8,4.7)
\psline[linewidth=.7pt]{->}(7,3.2)(5.5,1.7)
\rput(5.7,2.7){$e^{2\pi i w}$}
\psset{linecolor=LLLgray}
\qdisk(7.35,3.4){2.5pt}
\qdisk(7.6,3.4){2.5pt}
\qdisk(7.85,3.4){2.5pt}
\qdisk(8.1,3.4){2.5pt}
\rput(9.6,2.9){{\small $\frac{1}{\alf_0}-\Bk$}}
\end{pspicture}
\label{GoDown}
\caption{The two different colors correspond to the two different ways of going down the renormalization tower.
The gray part corresponds to $\mathscr{A}$ and the rest to $\mathscr{B}$.}
\end{center}
\end{figure}
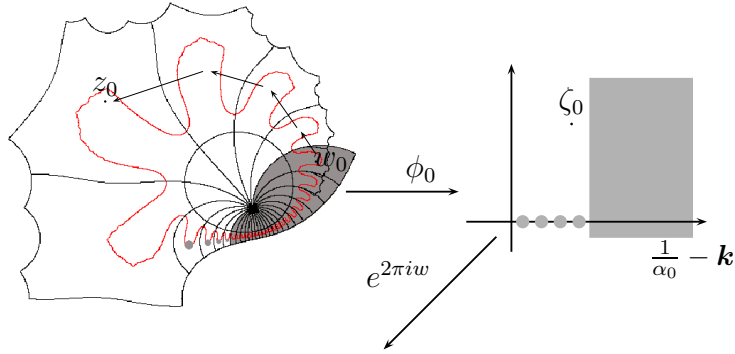
For $z_0\in \mathscr{A}_0$, let $w_0:=z_0$, $\sigma(0):=0$. 
For $z_0\in \mathscr{B}_0$, let $w_0\in S_0^0\bigcap  \cap_{n\geq 1} \Omega_0^n$ and positive integer 
$\sigma(0)< k_0+\Bk'$ be such that $f_0\co{\sigma(0)}(w_0)=z_0$. 
In both cases, let $\zeta_0:=\Phi_0(w_0)$. 

Define $z_1:=\ex(\zeta_0)$. 
Since $z_0 \in \Omega^1_0$, one can see that $z_1 \in \Omega^0_1$. 
Thus, we can repeat the above process to define the quadruple $(z_1,w_1,\zeta_1,\sigma(1))$, and so on. 

In general, for every $l\geq 0$, we have 
\begin{equation}\label{sign-property}
\begin{gathered} 
z_l=\ex(\zeta_{l-1}),\;  z_l\in \Omega_l^0, 
f_l\co{\sigma(l)}(w_l)=z_l,\; \Phi_l(w_l):=\zeta_l,
0 \leq \sigma(l) < k_l+\Bk'.
\end{gathered}
\end{equation}
Note that by the definition of this sequence and condition \eqref{k'-condition}, for all $l\geq 0$ we have 
\begin{equation}\label{zetas}
\Bk'+1/2 \leq \Re \zeta_l \leq \lfloor 1/\alf_l\rfloor -\Bk, \text{ or } \zeta_l\in \cup_{j=1}^{\Bk'} B(j,\delta_3).    
\end{equation}

\begin{propo}\label{P:good-height-lem}
Assume that $z_0\in \cap_{n=0}^\infty \Omega^n_0\setminus \{0\}$ and $\alf$ is a non-Brjuno number in $\irr$. 
If $\langle\zeta_j\rangle_{j=0}^\infty$ is the sequence associated to $z_0$, there are arbitrarily large $m$ with 
\begin{equation}\label{good-height-eq}
\Im \zeta_m\leq \frac{1}{2\pi}\log \frac{1}{\alf_{m+1}}.
\end{equation} 
\end{propo}

To prove Proposition~\ref{P:good-height-lem}, we need the estimate in the following Proposition, which 
will be prove in Section~\ref{sec:metric-properties}.

\begin{propo}\label{P:height-control-lem}
There exist  positive constants $D_1$ and $D_2$ such that for all $n\geq 1$,
\begin{multline}\label{height-control}
\textrm{if } \Im \zeta_{n+1} \geq \frac{D_1}{\alf_{n+1}}, \textrm{ then }
\Im \zeta_{n+1} \leq \frac{1}{\alf_{n+1}}\Im \zeta_n-\frac {1}{2\pi \alf_{n+1}}\log \frac{1}{\alf_{n+1}}+\frac{D_2}%
{\alf_{n+1}}.
\end{multline}
\end{propo} 

\begin{proof}[Proof of Proposition~\ref{P:good-height-lem}] 
Fix an arbitrary integer $\ell\geq 1$. 
We need to show that there exists $m\geq \ell$ satisfying \eqref{good-height-eq}. 
First note that one of the following two occurs. 
\begin{itemize}
\item[$(\ast)$] There exists an integer $n_0\geq \ell$ such that for every $j\geq n_0$, we have 
$\Im \zeta_{j}\geq \frac{D_1}{\alf_{j}}$. 
\item[$(\ast\ast)$] There are infinitely many integers $j$ greater than or equal to $\ell$  with $\Im \zeta_j< \frac{D_1}{\alf_j}$. 
\end{itemize}

If $(\ast)$ holds, we can use Proposition~\ref{P:height-control-lem} at every level $j\geq n_0$. 
So, recursively using Estimate \eqref{height-control}, we obtain the following inequality for every $n>n_0$, 
\begin{multline*}\label{long-inequality}
\Im \zeta_{n} \leq \frac{1}{\alf_n\alf_{n-1}\cdots\alf_{n_0}}\Im \zeta_{n_0-1} 
-\frac{1}{2\pi\alf_n\alf_{n-1}\cdots\alf_{n_0}}\log\frac{1}{\alf_{n_0}}\\
\qquad-\frac{1}{2\pi\alf_n\alf_{n-1}\cdots\alf_{n_0+1}}\log\frac{1}{\alf_{n_0+1}}-\cdots-\frac{1}{2\pi \alf_n}\log\frac{1}{\alf_n} \\ 
+D_2\left( \frac{1}{\alf_n\alf_{n-1}\cdots \alf_{n_0}}+\frac{1}{\alf_n\alf_{n-1}\cdots\alf_{n_0-1}}+\cdots+\frac{1}{\alf_n} \right).   
\end{multline*}
Let $\beta_{-1}:=1$, and $\beta_j:=\Pi_{l=0}^j \alf_l$, for $j\geq 0$. 
If \eqref{good-height-eq} does not hold for any $m\geq \ell$, replacing $\Im \zeta_n$ by $\frac{1}{2\pi}\log\frac{1}{\alf_{n+1}}$
in the above inequality and then multiplying both sides of it by $2\pi \beta_n$ we obtain 
\begin{align*}
\sum_{j=n_0-1}^{n}\beta_j  & \log\frac{1}{\alf_{j+1}}  \\
&  \leq 2\pi\beta_{n_0-1} \Im \zeta_{n_0-1}  
+2 \pi D_2\left(\beta_{n_0-1}+\beta_{n_0}+\cdots+\beta_{n-1}\right)\\
& \leq 2\pi\beta_{n_0-1} \Im \zeta_{n_0-1}  
+2 \pi D_2\left(\beta_{n_0-1}+\beta_{n_0})( 1+ 1/2 +1/4+1/8+\cdots \right)
\end{align*}
As $n$ was an arbitrary integer, the above inequality implies $\sum_{j=n_0-1}^{\infty} \beta_j  \log\frac{1}{\alf_{j+1}}$ 
is finite. 
By Inequality~\eqref{equivalence}, this contradicts $\alf$ being non-Brjuno.  

Now assume $(\ast\ast)$ holds \footnote{We present an alternative proof, in the proof of 
Proposition~\ref{P:borrowed-lem}, for a slightly weaker conclusion that is enough for our purpose. 
More precisely, we show that if $(\ast \ast)$ holds for some $z$, then there exists a constant $E$ and 
infinitely many $m$ with $\Im \zeta_m\leq \tfrac{1}{2\pi}\log \tfrac{1}{\alf_{m+1}}+E$.}. 
Let $n_1< m_2\leq n_2<m_3\leq n_3<\cdots$ be an increasing sequence of 
positive integers with the following properties (the case that some $n_i=\infty$ is easier and follows from the following argument)
\begin{itemize}
\item for every integer $j$ with $m_i\leq j \leq n_i$, we have $\Im \zeta_{j}< \frac{D_1}{\alf_{j}}$ 
\item for every integer $j$ with $n_i<j<m_{i+1}$, we have $\Im \zeta_j\geq \frac{D_1}{\alf_j}$.   
\end{itemize}
Assuming \eqref{good-height-eq} does not hold for any $m\geq \ell$, and recursively using \eqref{height-control} for  
$n$ in $\{n_i,n_i+1,\dots, m_{i+1}-2\}$, we obtain the following inequality for every $i\geq 2$.
\begin{eqnarray*}
\sum_{j=n_i}^{m_{i+1}-1}\beta_j\log\frac{1}{\alf_{j+1}} \leq 2\pi\beta_{n_i} \Im \zeta_{n_i}+%
2\pi D_2\left(\beta_{n_i}+\beta_{n_i+1}+\cdots+\beta_{m_{i+1}-2}\right). 
\end{eqnarray*}
Hence,
\begin{align*}
\sum_{j=m_2}^{\infty} \beta_j\log\frac{1}{\alf_{j+1}}&=\sum_{i=2}^{\infty}\Big(\sum_{j=m_i}^{n_i-1}\beta_j\log%
\frac{1}{\alf_{j+1}}+\sum_{j=n_i}^{m_{i+1}-1} \beta_j\log\frac{1}{\alf_{j+1}}\Big)\\
&\leq \sum_{i=2}^{\infty}\Big(\sum_{j=m_i}^{n_i-1} \beta_j\log\frac{1}{\alf_{j+1}}+2\pi \beta_{n_i}\Im \zeta_{n_i}+%
2\pi D_2 \sum_{j=n_i}^{m_{i+1}-2} \beta_j\Big). \\
&\leq \sum_{i=2}^{\infty}\Big( 2\pi D_1 (\beta_{m_i-1}+\beta_{m_i})+ 2\pi D_1 \beta_{n_i-1}+2\pi D_2 (\beta_{n_i}
+\beta_{n_i+1}) \Big)
\end{align*}
In the second inequality above, we used the assumption 
$\frac{1}{2\pi}\log\frac{1}{\alf_{j+1}}< \Im \zeta_j< \frac{D_1}{\alf_j}$, for $j\in\{m_i, m_i+1,\dots, n_i\}$.
On the other hand, since $\alf_i \alf_{i+1} \leq 1/2$, for $i\geq 0$, we have $\sum_{i=0}^\infty \beta_i \leq 3$.
Thus, the last line of the above inequalities is uniformly bounded from above, contradicting $\alf$ 
being a non-Brjuno number.
\end{proof}

\subsection{Going up the renormalization tower}\label{S:Going-Up}
Assume that $\Im \zeta_n\leq \frac{1}{2\pi}\log \frac{1}{\alf_{n+1}}$ holds for some positive integer $n$. 
We may use Proposition~\ref{P:r-ball} with $E=0$ and the map $f_{n+1}$, to obtain a curve $\gamma_n$ and a 
ball $B(\gamma_n(1), r^*)$ that projects under $\ex$ into $\mathbb{C} \setminus \Omega_0^0(f_{n+1})$. 
Let us define the set 
\[V_{n+1}:= B_{\delta_1}\big(B(\gamma_n(1), r^*)\cup \gamma_n[0,1]\big).\]
We shall define domains $V_n, V_{n-1}, \ldots, V_1$, a holomorphic map $g_{n+1}$, and anti-holomorphic maps 
$g_{n}, g_{n-1},\ldots, g_1$ as in diagram  
\begin{equation}\label{chain}
\xymatrix@1{ 
V_{n+1}\ar[r]^-{g_{n+1}} & V_n \ar[r]^-{g_n} & V_{n-1} \ar[r]^-{g_{n-1}} & \cdots \ar[r]^-{g_2} & V_1 \ar[r]^-{g_1} &
V_0:=B_1 (\Omega_0^0),                                            
}
\end{equation}
satisfying
\begin{itemize}
\item[--]for all $i=1,2,\dots,n$, $V_i=\Omega^0_i\setminus I_{i,j(i)}$ for some $j(i)\in\{0,1,\dots,\lfloor 1/\alf_i\rfloor-\Bk-1\}$;
\item[--]for all $i=1,2,\dots, n+1$, $g_i:V_i\ra V_{i-1}$; for all $i=0, 1, \dots, n$, $z_i \in V_i$;   
\item[--] $g_{n+1}(\zeta_n)=z_n$; and for all $i=1,2,\ldots ,n$, $g_i(z_i)=z_{i-1}$. 
\end{itemize}

We use an inverse inductive process to define the pairs $(g_{i+1},V_i)$, starting with $i=n$ and ending with $i=0$.

\smallskip

{\em Base step $i=n$:} 
Recall that $\zeta_n\in V_{n+1}$ satisfies \eqref{zetas}. 
As $\diam(\Re V_{n+1})\leq 1-\delta_1$, and $\delta_3< \delta_1$, there exists an integer $j\in \{0,1\}$ such that 
\[\Re(V_{n+1}-j) \subset (0,\alf_n^{-1}-\Bk).\]
With this choice of $j$, we define $g_{n+1}: V_{n+1} \ra \cc$ as  
\begin{equation*}
g_{n+1}(\zeta):=f_n\co{(j+\sigma(n))}(\Phi_n^{-1}(\zeta-j)).
\end{equation*}
By Proposition~\ref{P:r-ball}-1, $\ex(V_{n+1}-j)$ is contained in $\Dom f_{n+1}$. 
So, Lemma~\ref{renorm}, combined with Equations~\eqref{k'-condition} and \eqref{sign-property}, 
imply that $f_n\co{(j+\sigma(n))}$ is defined on $\Phi_n^{-1}(V_{n+1}-j)$. 
Indeed, we have $g_{n+1}(V_{n+1}) \subset \Omega^0_n$.  

Since $g_{n+1}(V_{n+1})$ intersects at most $\sigma(n)+1\leq \Bk'+k_n$ of the curves $I_{n,j}$, 
there exists $j(n)\in \{0,1,\dots ,\lfloor 1/\alf_n\rfloor-\Bk-1\}$ with $g_{n+1}(V_{n+1})\cap I_{n,j(n)}=\emptyset$. 
We define $V_n:=\Omega^0_n\setminus I_{n,j(n)}$. 

Finally, by the equivariance property of $\Phi_n$,  
\[g_{n+1}(\zeta_n)=f_n\co{(j+\sigma(n))}(\Phi_n^{-1}(\zeta_n-j))=f_n\co{\sigma(n)}(w_n)=z_n.\]

{\em Induction step:} Assume that $(g_{i+1},V_i)$ is defined and we want to define $(g_i,V_{i-1})$. 
As every closed loop in $V_{i}$  is contractible in $\cc\setminus\{0\}$, there exists an inverse branch $\eta_i$ of $\ex$ 
defined on $V_i$ with $\eta_i(z_i)=\zeta_{i-1}$. 
Now we consider two separate cases.
\begin{description}
\item[$\mathscr{R}$] $\Re (\eta_i(V_i)) \subset [1/2,\infty)$,
\item[$\mathscr{L}$] $\Re (\eta_i(V_i))\cap (-\infty, 1/2) \neq \emptyset$.
\end{description}

\smallskip

{\em If  $\mathscr{R}$ occurs:} 
Since $\zeta_{i-1}\in \eta_i(V_i)$ satisfies \eqref{zetas}, and $\diam (\Re B_{\delta_3}(\eta_i(V_i)))\leq \Bk'+1/4$ 
by Equation \eqref{k'} and $\delta_3<1/8$, there exists an integer $j\in \{0,1, \dots, \Bk'+1\}$ with 
\begin{equation}\label{well-inside}
B_{\delta_3}(\eta_i(V_i))-j \subset \{\xi\in \cc :  3/8 \leq \Re \xi \leq \lfloor 1/\alf_{i-1}\rfloor -\Bk \}.
\end{equation} 
By Lemma~\ref{the-balls}, $\ex (B_{\delta_3}(\eta_i(V_i))) \subset \Dom f_i$. 
Thus, Lemma~\ref{renorm} and Inequality \eqref{k'-condition} imply that $f_{i-1}^{j+\sigma(i-1)}$ is defined 
on $\Phi_{i-1}^{-1}(B_{\delta_3}(\eta_i(V_i))-j)$. 
Define $\widetilde{g}_{i}$ on $B_{\delta_3}(\eta_i(V_i))$ as  
\begin{align}\label{psi-tilde}
\widetilde{g}_{i}(\zeta):=f_{i-1}\co{(j+\sigma(i-1))}(\Phi_{i-1}^{-1}(\zeta-j)),
\end{align}
and let 
\[g_i:= \widetilde{g}_i \circ \eta_i.\]

One can see that $\widetilde{g}_i(B_{\delta_3}(\eta_i(V_i)))$ intersects at most $\Bk'+k_{i-1}+1$ of the curves $I_{i-1, j}$, for 
$j=0,1,\dots , \lfloor 1/\alf_{i-1}\rfloor -\Bk-1$.
Hence, by Equation \eqref{k'-condition}, there is $j(i-1)$ in that set with 
$\widetilde{g}_i(B_{\delta_3}(\eta_i(V_i))\cap I_{i-1,j(i-1)}=\emptyset$.
Now, we define $V_{i-1}:=\Omega^0_{i-1}\setminus I_{i-1,j(i-1)}$ so that 
\begin{equation}\label{well-inside-1}
\widetilde{g}_i(B_{\delta_3}(\eta_i(V_i)))\subset V_{i-1}.
\end{equation}

Finally, by the equivariance property of $\Phi_{i-1}$, we have   
\[g_i(z_i)=f_{i-1}\co{(j+\sigma(i-1))}(\Phi_{i-1}^{-1}(\eta_i(z_i)-j))=f_{i-1}\co{\sigma(i-1)}(w_{i-1})=z_{i-1}.\]

\smallskip

{\em If  $\mathscr{L}$ occurs:} 
Here, because $\diam(\Re \eta_i(V_i))\leq \Bk'$ and $\zeta_{i-1} \in \eta_i(V_i)$ satisfies \eqref{zetas}, 
we must have $\zeta_{i-1}\in B(j,\delta_3)$ for some $j$ in $\{1,2,\dots,\Bk'\}$.
Therefore, by Lemma~\ref{the-balls} (and $\C_i\subset V_i$, $\delta_3<1/8$), $\eta_i(V_i)\supseteq B(j, \delta_3).$ 
This implies that 
\[\eta_i(V_i)\cap B_{\delta_3}(\{0,-1,-2,\dots,-\Bk'\})=\emptyset,\]
or equivalently 
\begin{equation}\label{away-from-danger}
B_{\delta_3}(\eta_i(V_i)) \cap \{0,-1,-2,\dots,-\Bk' \}=\emptyset. 
\end{equation} 
Now, we extend $\Phi_{i-1}: \p_{i-1}\to \cc$ over a larger domain, using the dynamics of $f_{i-1}$, so that a unique branch of 
$\Phi_{i-1}^{-1}$ is defined on $B_{\delta_3}(\eta_i(V_i))$.

Recall the sectors $\C_{i-1}^{-j}\cup(\Csh_{i-1})^{-j}$, for $1\leq j\leq k_{i-1}$, and 
$S^0_{i-1}=\C_{i-1}^{-k_{i-1}}\cup(\Csh_{i-1})^{-k_{i-1}}$ 
introduced  for the definition of renormalization (of $f_{i-1}$). 
If $k_{i-1}<\Bk'+1$, using \eqref{k'-condition}, one can consider further pre-images for 
$j=k_{i-1}+1,\dots,\Bk'+1$ as  
\begin{gather*}
\C_{i-1}^{-j}:=\Phi_{i-1}^{-1}(\Phi_{i-1}(\C_{i-1}^{-k_{i-1}})-(j-k_{i-1})), \\
(\Csh_{i-1})^{-j}:=\Phi_{i-1}^{-1}(\Phi_{i-1}((\Csh_{i-1})^{-k_{i-1}})-(j-k_{i-1})). 
\end{gather*}
Let $\D_{i-1}:=\C_{i-1}^{-\Bk'-1}\cup(\Csh_{i-1})^{-\Bk'-1}$, and observe that 
$f_{i-1}\co{(\Bk'+1)}:\D_{i-1} \ra \C_{i-1}\cup \Csh_{i-1}$.
For $i\geq 1$, define the set
\[\p_{i-1}^\natural:=\bigcup_{j=0}^{\Bk'} f_{i-1}\co{j}(\D_{i-1}).\]   
Define $\Phi_{i-1}^\natural:\p_{i-1}^\natural\ra \cc$ as follows.
For $z\in \p_{i-1}^\natural$, there is an integer $j$ with $ 0\leq j\leq \Bk'+1$ and $f_{i-1}\co{j}(z)\in \p_{i-1}$. 
Let 
\[\Phi_{i-1}^\natural(z):=\Phi_{i-1}(f_{i-1}^j(z))-j.\] 
As $\Phi_{i-1}$ satisfies the Abel functional equation on $\p_{i-1}$, one can see that $\Phi_{i-1}^\natural$ is independent of 
the choice of $j$ and hence, defines a holomorphic map on $\p_{i-1}^\natural$. 
The map $\Phi_{i-1}^\natural$ is not univalent. 
However, it still satisfies the Abel Functional equation on $\p_{i-1}^\natural$. 
Indeed, it has critical points at the critical point of $f_{i-1}$ and its $\Bk'$ pre-images within $\p_{i-1}^\natural$. 
The $\Bk'+1$ critical points of $\Phi_{i-1}^\natural$ are mapped to $0,-1,-2, \ldots,-\Bk'$.  

The map $\Phi_{i-1}^\natural$ is a natural extension of $\Phi_{i-1}$ to a multi-valued holomorphic map on 
$\p_{i-1}^\natural\cup \p_{i-1}$. 
However, the two maps 
\begin{align*}
&\Phi_{i-1}^\natural:\bigcup_{j=0}^{\Bk'} f_{i-1}\co{j}(\D_{i-1})\ra\cc,\; \Phi_{i-1}:\bigcup_{j=\Bk'+1}^%
{\lfloor1/\alf_{i-1}\rfloor +\Bk'-\Bk-1} f_{i-1}\co{j}(\D_{i-1})\ra\cc 
\end{align*}
provide a well defined holomorphic map on every $\Bk'+1$ consecutive sectors of the form $f_{i-1}^j(\D_{i-1})$. 
More precisely, for every $l$ with $0\leq l < \lfloor 1/\alf_{i-1}\rfloor-\Bk$,  
\[\Phi_{i-1}^\natural \amalg_l \Phi_{i-1}: \bigcup_{j=0}^{\Bk'} f_{i-1}\co{(l+j)}(\D_{i-1})\ra \cc\] 
is defined as
\[
\Phi_{i-1}^\natural \amalg_l \Phi_{i-1}(z):=
\begin{cases}
\Phi_{i-1}^\natural(z),   & \text{if } z\in f_{i-1}\co{j}(\D_{i-1}) \text{ with } j< \Bk'+1 ; \\
\Phi_{i-1}(z),            & \text{if } z\in f_{i-1}\co{j}(\D_{i-1}) \text{ with } j\geq \Bk'+1.
\end{cases}
\]

The set $B_{\delta_3}(\eta_i(V_i))$ has diameter strictly less than $\Bk'+1$. 
Therefore, it can intersect at most $\Bk'+1$ vertical strips of width one. 
In other words, $B_{\delta_3}(\eta_i(V_i))$ is contained in $\Bk'+1$ consecutive sets in the list 
\begin{multline*}
\Phi_{i-1}^\natural (                                          \D_{i-1}),
\Phi_{i-1}^\natural (f_{i-1}                                  (\D_{i-1})), \dots,
\Phi_{i-1}^\natural (f_{i-1}\co{\Bk'}                           (\D_{i-1})),\\
\Phi_{i-1}      (f_{i-1}\co{(\Bk'+1)}                             (\D_{i-1})),\dots, 
\Phi_{i-1}       (f_{i-1}\co{(2\Bk'+1)}(\D_{i-1})).
\end{multline*}
Thus, by the above argument about $\Phi_{i-1}^\natural \amalg_l \Phi_{i-1}$, and that every closed loop in 
$B_{\delta_3}(\eta_i(V_i))$ is contractible in the complement of the critical values of 
$\Phi_{i-1}^\natural$ by \eqref{away-from-danger}, there exists an inverse branch of this map defined on 
$B_{\delta_3}(\eta_i(V_i))$. We denote this map by $\widetilde{g}_i$ and let 
\[g_i:=\widetilde{g}_i \circ \eta_i: V_{i}\ra \Omega^0_{i-1}.\] 

One can similarly verify that $\widetilde{g}_i(B_{\delta_3}(\eta_i(V_i)))$ does not intersect some curve 
$I_{i-1,j(i-1)}$. 
We define $V_{i-1}:=\Omega^0_{i-1}\setminus I_{i-1,j(i-1)}$ and have $g_i:V_i \ra V_{i-1}$. 
Indeed, we have shown 
\begin{equation}\label{well-inside-2}
\widetilde{g}_i(B_{\delta_3}(\eta_i(V_i)))\subset V_{i-1}. 
\end{equation}
Here, $\sigma(i-1)=0$, $\Phi_{i-1}(w_{i-1})=\zeta_{i-1}$, and $w_{i-1}=z_{i-1}$. 
Hence $g_{i}(z_i)=z_{i-1}$.
This finishes the definition of the domains and maps when $\mathscr{L}$ occurs.\footnote{The chain of the domains 
and maps \eqref{chain} defined here depends on the value of $n$. 
It is likely that the last parts of the chains defined for two different values of $n$ are not identical.}
\footnote{For an alternative approach to going down and up the tower see \cite[Section 3]{Ch10-II}.}

\subsection{Safe lifts}\label{sec:safe-transport}
Each domain $V_{j}$, for $j=n+1, n, \dots, 0$, is a hyperbolic Riemann surface. 
Let $\rho_i(z)|dz|$ be the complete metric of constant curvature $-1$ on $V_i$. 
The next two lemmas are natural consequences of the definition of the chain \eqref{chain}.    

\begin{lem}\label{contract}
For every $i \in \{1, 2, \dots n\}$,  the map $g_i:(V_i, \rho_i) \ra (V_{i-1}, \rho_{i-1})$ is uniformly contracting. 
More precisely, for every $z \in V_i$, we have  
\[\rho_{i-1}(g_i(z))\cdot |g_i'(z)| \leq \delta_4 \cdot\rho_i(z),\]  
with $\delta_4:=(2\Bk'+1)/(2\Bk'+1+\delta_3)$.
\end{lem}

\begin{proof} 
Let $\tilde{\rho}_i(z)|d z|$ and $\hat{\rho}_i(z)|dz|$ denote the Poincar\'e metric on the domains $\eta_i(V_i)$ 
and $B_{\delta_3}(\eta_i(V_i))$, respectively. 
By the definition of $g_i$ and properties \eqref{well-inside-1} and \eqref{well-inside-2} we can decompose the map 
$g_i:(V_i, \rho_i) \ra (V_{i-1}, \rho_{i-1})$ as follows:
\begin{displaymath}
\xymatrix{
(V_i,\rho_i) \ar[r]^-{\eta_i} &\; (\eta_i(V_i), \tilde{\rho}_i) \ar@{^{(}->}[r]^-{inc.} &\;  
(B_{\delta_3}(\eta_i(V_i)),\hat{\rho}_i)\; \ar[r]^-{\tilde{g}_i} & \;(V_{i-1},\rho_{i-1}).}
\end{displaymath}

By Schwartz-Pick Lemma, the first map and the last map in the above chain are non-expanding. 

To show that the inclusion map is uniformly contracting in the respective metrics, fix an arbitrary point 
$\xi_0$ in $\eta_i(V_i)$, and define 
\[H(\xi):=\xi+(\xi-\xi_0)\frac{\delta_3}{(\xi-\xi_0+2\Bk'+1)}: \eta_i(V_i)\ra \cc.\]
Since $\diam \Re(\eta_i(V_i))\leq \Bk'$, we have $|\Re(\xi-\xi_0))|\leq \Bk'$ for every $\xi \in \eta_i(V_i)$. 
This implies that $|\xi-\xi_0|<|\xi-\xi_0+2\Bk'+1|$, and hence  
\begin{align*}
|H(\xi)-\xi|&= \delta_3 |\frac{\xi-\xi_0}{\xi-\xi_0+2\Bk'+1}| < \delta_3.
\end{align*}
So, $H(\xi)$ is a holomorphic map from $\eta_i(V_i)$ into  $B_{\delta_3}(\eta_i(V_i))$. 
By Schwartz-Pick Lemma, $H$ is non-expanding. 
In particular, at $H(\xi_0)=\xi_0$ we obtain  
\[\tilde{\rho}_i(\xi_0) |H'(\xi_0)|= \tilde{\rho}_i(\xi_0)(1+\frac{\delta_3}{2\Bk'+1}) \leq \hat{\rho}_i(\xi_0).\]
That is, $\tilde{\rho}_i(\xi_0)\leq \delta_4\cdot \hat{\rho}_i(\xi_0)$
with $\delta_4=(2\Bk'+1)/(2\Bk'+1+\delta_3)$.
\end{proof}

\begin{lem}\label{safe-steps}
There exists a constant $\delta_5>0$ independent of $n$ such that   
\begin{itemize}
\item[(1)]each $g_i:V_i\ra V_{i-1}$, for $i=1, 2, \dots, n+1$, is either one-to-one or has only one simple critical point;
\item[(2)]each $g_i:V_i\ra V_{i-1}$, for $i=1, 2, \dots, n$, is one-to-one on the hyperbolic ball 
\begin{equation*}
B_{\rho_i}(z_i,\delta_5):=\{z\in V_i\mid d_{\rho_i}(z,z_i)< \delta_5\}.
\end{equation*}
\end{itemize}
\end{lem}

\begin{proof}
{\em Part (1):} 
Each map $g_i$ is a composition of at most four maps; $\eta_i$ (this does not appear for $g_{n+1}$), 
a translation by an integer $j$,  $\Phi_{i-1}^{-1}$, and $f_{i-1}\co{(j+\sigma(i-1))}$. 
The first three maps are one-to-one. 
The map $f_{i-1}\co{(j+\sigma(i-1))}$ on $\Phi_{i-1}^{-1}(\eta_i(V_i)-j)$ is either one-to-one or has at most one 
simple critical point. 
To see this, first note that the relevant critical points of $f_{i-1}\co{(j+\sigma(i-1))}$ within $\Omega_{i-1}^0$ are contained in    
$\cup_{l=0}^{j+\sigma(i-1)} \{f_{i-1}^{-l}(\cp_{f_{i-1}})\}$, which are all non-degenerate. 
If $\Phi_{i-1}^{-1}(\eta_i(V_i)-j)$ contains more than one element in the above list, by the equivariance property of 
$\Phi_{i-1}$, there must be a pair of points $\xi$, $\xi+m$ (for some integer $m\neq0$) in $\eta_i(V_i)-j$. 
As this set is contained in the lift of a simply connected domain under $\ex$, that is not possible.  

\medskip

{\em Part (2):} The proof  is broken into four small steps.

\smallskip

{\em Step 1.} If $g_i$ has a critical value, then it belongs to $\cup_{l=0}^{j+\sigma(i-1)-1} \{f_{i-1}\co{l}(-4/27)\}$,
where $j$ is the  integer defined in case $\mathscr{R}$ of the inductive construction\footnote{When $\sigma(i-1)+j=0$,  
we define the set to be empty.}.

\smallskip

Looking back at the inductive process, the map $g_i$ introduced in $\mathscr{L}$ is univalent, 
hence, we only need to look at maps considered in $\mathscr{R}$. 
By the definition \eqref{psi-tilde}, if $g_{i-1}$ has a critical value, it must belong to the above set.

\smallskip

Let $\cv_{g_{i-1}}$ denote the critical value of $g_{i-1}$, if it exists. 

\smallskip

{\em Step 2.} If $\cv_{g_{i-1}}=f_{i-1}\co{l}(-4/27)= \Phi_{i-1}^{-1}(l+1)$, for some $l$ with $0\leq l \leq \sigma(i-1)+j-1$, then 
$z_{i-1}\notin \Phi_{i-1}^{-1}(B(l+1, \delta_3))$.
 
To see this, we refer to the definition of quadruples \eqref{quadruples}. 
If $z_{i-1}\in \mathscr{A}_{i-1}$, recall that $\sigma(i-1)=0$ and $z_{i-1}=\Phi_{i-1}^{-1}(\zeta_{i-1})$.
By the above step, and since $j\leq \Bk'+1$, we have    
$\cv_{g_{i-1}}\in \cup_{l=0}^{\Bk'} \{f_{i-1}\co{l}(-4/27)\}$.  
Now, if $\Re \zeta_{i-1}\geq \Bk'+1/2$, then 
$z_{i-1}$ can not belong to $\cup_{l=1}^{\Bk'} \Phi_{i-1}^{-1}(B(l,\delta_3))$. 
And if $\zeta_{i-1}\in \cup_{l=1}^{\Bk'} B(l,\delta_3)$, we have $j=0$. 
By Step 1, $g_{i-1}$ has no critical value.   

When $z_{i-1}\in \mathscr{B}_{i-1}$, by definition, $z_{i-1} \notin \cup_{l=0}^{\Bk'}\Phi_{i-1}^{-1}(B(l,\delta_3))$.

\smallskip

{\em Step 3.} There exists a real constant $\delta>0$ such that $B_{\rho_{i-1}}(z_{i-1}, \delta)$
is simply connected and does not contain $\cv_{g_i}$.

By steps 1 and 2, it is enough to show that there exists a $\delta>0$ such that for every $l\in \{1,2,\dots, 2\Bk'+1\}$ 
and every $i\leq n$, $B_{\rho_{i-1}}(\Phi_{i-1}^{-1}(l),\delta)$ is simply connected and is contained in 
$\Phi_{i-1}^{-1}(B(l,\delta_3))$. 
Recall that $\Phi_{i-1}$ is univalent on $\{\xi\in \cc ; \Re \xi\in (0, \alf_{i-1}^{-1}-\Bk)\}$. 
As the balls $B(l,\delta_3)$ and the segments $\{s\cdot l+(1-s)(1/8-2\Bi), s\in (0,1)\}$, for $l=0, 1, \Bk'$, 
are well contained in this strip, it follows from the distortion theorem that there are constants $M_1$ and $M_2$ 
such that 
\begin{gather*}
\Phi_{i-1}^{-1}(B(l,\delta_3))\supset B(\Phi_{i-1}^{-1}(l), M_1\cdot (\Phi_{i-1}^{-1})'(l)),\\
\forall z\in \Phi_{i-1}^{-1}(B(l,\delta_3)), d(z, \partial V_{i-1})\leq M_2\cdot  (\Phi_{i-1}^{-1})'(l).
\end{gather*}
For the second line of the above equations we have used the Koebe distortion theorem on the segment 
$\{s\cdot l+(1-s)(1/8-2\Bi), s\in (0,1)\}$. 
As $\rho_{i-1}(\cdot)$ is comparable to $1/d(\cdot, \partial V_{i-1})$, one infers that 
$\Phi_{i-1}^{-1}(B(l,\delta_3))$ contains a round hyperbolic ball of radius uniformly bounded from below. 

\smallskip

{\em Step 4.} part (2) of the lemma holds for $\delta_5:=\delta$. 

By the contraction of $g_i$ (Lemma~\ref{contract}), $g_i (B_{\rho_i}(z_i,\delta_5))$ is contained 
in the ball $B_{\rho_{i-1}}(z_{i-1}, \delta)$. 
As $B_{\rho_{i-1}}(z_{i-1},\delta)$ is simply connected and does not contain any critical value of $g_i$, one can find 
a univalent inverse branch of $g_i$ defined on this ball. 
Therefore, $g_i$ is one-to-one on the ball $B_{\rho_i}(z_i,\delta_5)$.   
\end{proof}

To avoid unnecessary details in Section~\ref{S:Going-Up} we assumed that the constant $E=0$. 
Clearly, one may repeat the construction line by line for a nonzero value of the constant $E$. 

Let $\G_n$ denote the map 
\[\G_n:=g_1\circ g_2 \dots  \circ g_{n+1}:V_{n+1}\ra \Omega^0_0.\] 
Recall the curve $\gamma_n$ obtained in Proposition~\ref{P:r-ball} for the map $f_{n+1}$. 
We have $\gamma_n(0)=\zeta_n$. 
By the properties of the chain~\eqref{chain}, $\G_n(\gamma_n(0))=z_0$. 

\begin{lem}\label{safe-trip} 
For all $E\in \mathbb{R}$, there exists a constant $D_3>0$ such that for every $n\geq 1$ 
satisfying $\Im \zeta_n\leq (2\pi)^{-1} \log (1/\alf_{n+1})+E$, there exists $r_n\in (0, +\infty)$ such that   
\begin{enumerate}\label{transfer}
  \setlength{\itemsep}{.3em}
\item[(1)] $\G_n(B(\gamma_n(1),r^*)) \cap \Omega^{n+1}_0=\emptyset$; 
\item[(2)] $B(\G_n(\gamma_n(1)),r_n)\subset\G_n(B(\gamma_n(1),r^*))$, and $|\G_n(\gamma_n(1))-z_0|\leq D_3\cdot r_n$;
\item[(3)] $r_n\leq D_3\cdot (\delta_4)^n$.
\end{enumerate}
\end{lem}

\begin{proof}
{\em Part (1):} By Proposition~\ref{P:r-ball}-2, for every $\zeta\in B(\gamma_n(1),r^*)$ we have 
\[\ex(\zeta)\notin\Omega^0_{n+1}, \text{ and } f_{n+1}(\ex(\zeta))\notin \Omega^0_{n+1}.\]
We claim that this implies 
\[g_{n+1}(\zeta)\notin \Omega^1_n, \text{ and } f_n(g_{n+1}(\zeta))\notin \Omega^1_n.\]
It follows from the definition of the renormalization (see proof of Lemma~\ref{renorm}) that since 
$\ex(\zeta)\notin \Omega^0_{n+1}$, then $\Phi_n^{-1}(\zeta)\notin \Omega^1_n$. 
Also from $f_{n+1}(\ex(\zeta))\notin \Omega_{n+1}^0$, it follows that $f_n^{j}(\Phi_n^{-1}(\zeta))\notin \Omega^1_n$, 
for $j=0,1,2,\dots, b_n+1$. 
In particular, by \eqref{k'-condition}, \eqref{sign-property}, and $j\leq \Bk'+1$,   $g_{n+1}(\zeta)$ and $f_n(g_{n+1}(\zeta))$ 
are not in $\Omega^1_n$. 

The same argument implies the following statement for every $i=n,n-1,\dots,1$.
\begin{multline*}
\text{For all $w\in V_i$, if } w \notin\Omega^{n-i+1}_i \text{ and } f_i(w)\notin \Omega^{n-i+1}_i,\\
\text{then } g_i(w)\notin\Omega^{n-i+2}_{i-1} \text{ and } f_{i-1}(g_i(w))\notin \Omega^{n-i+2}_{i-1}
\end{multline*} 
By an inductive argument, one infers from these that $\G_n(z)\notin \Omega^0_{n+1}$.

\medskip

{\em Part (2):} It follows from Proposition~\ref{P:r-ball}-4 that there exists a constant  $C$ such that 
$B\,(\gamma_n(1),r^*)\cup \gamma_n[0,1]$ has hyperbolic diameter (with respect to $\rho_{n+1}$ in $V_{n+1}$) less than $C$.
Let $m$ denote the smallest non-negative integer with 
\[C \cdot (\delta_4)^{m}\leq \delta_5/2.\] 
Note that $m$ is uniformly bounded from above independent of $n$.
We decompose the map $\G_n$ into two maps as follows
\[\G_n^l:=g_{n-m+1} \circ g_{n-m+2}\circ\cdots\circ g_{n+1}\; \text{ and }\; \G_n^u:=g_1\circ g_{2}\circ\cdots
\circ g_{n-m}.\]
By Lemma~\ref{contract} and our choice of $m$, we have
\[\G_n^l(B(\gamma_n(1),r^*)\cup \gamma_n[0,1])\subseteq B_{\rho_{n-m}}(z_{n-m},\delta_5/2).\]
Since by Lemmas~\ref{contract} and \ref{safe-steps} each $g_i$, for $i=n-m, n-m-1,\dots, 1$, is univalent and uniformly 
contracting on $B_{\rho_i}(z_i,\delta_5)$, we conclude that $\G_n^u$ is univalent on $B_{\rho_{n-m}}(z_{n-m},\delta_5)$. 
Thus, by the distortion theorem, $\G_n^u$ has bounded distortion on 
$\G_n^l(B(\gamma_n(1),r^*)\cup \gamma_n[0,1])$. 

We claim that $\G_n^l$ belongs to a pre-compact class of maps. 
That is because it is a composition of $m$ maps $g_i$, for $i=n+1,\dots, n-m+1$, where each of these 
maps is a composition of two maps as $g_i=\widetilde{g}_{i}\circ \eta_i$.  
The map $\eta_i$ is univalent on $V_i$ and, by the distortion theorem, has uniformly bounded distortion on 
sets of bounded hyperbolic diameter. 
The map $\widetilde{g}_{i}$ extends over the larger set $B_{\delta_3}(\eta_i(V_i))$, by \eqref{well-inside-1} 
and \eqref{well-inside-2}. 
So, it belongs to a compact class. 
(Indeed, $f_i\co{(\sigma(i)+j)}$ is a uniformly bounded number (by Proposition~\ref{P:turning}, \eqref{sign-property}, and 
$j\leq \Bk'+1$) of iterates of a map in the pre-compact class $\cup_{(0, r_3]} \ff_\alf$). 

Putting all these together, one infers that there exists a constant $C'$ such that
\begin{align*}
 |\G_n(\gamma_n(1))- z_0|&=  |\G_n (\gamma_n(1))- \G_n(\gamma_n(0))|\\
                                                       & \leq C'\cdot \diam (\G_n(B(\gamma_n(1), r^*)). 
\end{align*}
Also, $\G_n(B(\gamma_n(1),r^*))$ contains a round ball of Euclidean radius comparable to the diameter of 
$\G_n(B(\gamma_n(1),r^*))$. 

\medskip

{\em Part (3):} 
The domain $\G_n(B(\gamma_n(1),r^*))$ is contained in $\Omega^0_0$ which is compactly contained in $V_0$.
Thus, the Euclidean and the hyperbolic (with respect to $\rho_0$) metrics are comparable on $\Omega_0^0$. 
Now, the uniform contraction with respect to the hyperbolic metric  in Lemma~\ref{contract} implies the claim.
\end{proof}

\subsection{Semi-continuity of the post-critical set}

\begin{propo}\label{P:nearby-holes}
Let $\alf  \in  \irr$ and $f_0  \in \QIS_{\alf}$. 
Assume that for some non-zero $z$ in $\cap_{n=0}^\infty \Omega_0^n(f_0)$ there are infinitely many 
distinct positive integers $m$ for which Inequality \eqref{good-height-eq} holds. 
Then $\cap_{n=0}^\infty \Omega_0^n(f_0)$  is non-uniformly porous at $z$. 
\end{propo}

\begin{proof}
By the assumption on $\alf$, $f_0$ is infinitely near-parabolic renormalizable and hence we may define 
$f_n=\rr \co{n} f_0$, $n\geq 0$. 
Let $n_i$, $i\geq 0$, be an increasing sequence of positive integers for which Inequality~\ref{good-height-eq} holds.  
Applying Proposition~\ref{P:r-ball} to the maps $f_{n_i+1}$ there are curves $\gamma_{n_i}$ and balls $B(\gamma_{n_i}(1),r^*)$ 
enjoying the properties in that lemma. 
The maps $\G_{n_i}$, by Lemma~\ref{transfer}, provide us with a sequence of balls 
$B(\G_{n_i}(\gamma_{n_i}(1)),r_{n_i})$ satisfying 
\[B(\G_{n_i}(\gamma_{n_i}(1)),r_{n_i})\cap \Omega^0_{n_i+1}=\emptyset,\; |\G_{n_i}(\gamma_{n_i}(1))-z_0|\leq D_3\cdot r_{n_i}, 
\; r_{n_i}\ra 0. \]   
This finishes the proof of the proposition. 
\end{proof}

\begin{proof}[Proof  of Theorem~\ref{pc-area}] 
Let $z_0 \in \pc(f_0)\setminus \{0\}$. 
By Proposition~\ref{neighbor}, $z_0 \in \cap_{n=0}^\infty\Omega^n_0\setminus \{0\}$. 
Thus, we can define the sequence of quadruples~\eqref{quadruples}. 
Since $\alf$ is a non-Brjuno number, Proposition~\ref{P:good-height-lem} provides us with an 
strictly increasing sequence of integers $n_i$ for which we have 
inequality~\eqref{good-height-eq} with $m=n_i$. 
By Proposition~\ref{P:nearby-holes}, $\pc(f_0)$ is non-uniformly porous at $z_0$. 
This implies that $z_0$ is not a Lebesgue density point of $\pc(f_0)$, and hence, 
by the Lebesgue density theorem, $\pc(f_0)$ must have zero area.
Indeed, continuing the notations in the proof of Proposition~\ref{P:nearby-holes}, for $s_i:=r_{n_i}+D_3\cdot r_{n_i}$, we have 

\begin{equation*}
\frac{\area\, (B(z_0,s_i))\cap \pc(f_0))}{\area\,(B(z_0,s_i))}\leq
\frac{\pi (s_i)^2 -\pi(r_{n_i})^2} {\pi (s_i)^2} 
\leq \frac{(D_3)^2+2D_3} {(D_3)^2+2D_3+1} <1.  
\footnote{The proof does not imply that $\pc(f_0)\setminus \{0\}$ is porous (shallow), i.e.\ at every scale around 
a point in $\pc(f_0)\setminus\{0\}$ there is a disk of comparable radius in the complement of $\pc(f_0)$.     
Indeed, it seems that in Proposition~\ref{P:good-height-lem}, given  any increasing sequence of positive integers 
$\langle n_i\rangle$ one can find a non-Brjuno $\alf$ and $z\in \pc(P_\alf)$ such that Inequality~\eqref{good-height-eq} 
holds only at levels $n_i$.
Hence, the scales obtained in the above proof may shrink to zero arbitrarily fast.}
\end{equation*}
\end{proof}

\begin{proof}[Proof of Corollary~\ref{non-recurrent}]
By the argument before Proposition~\ref{visit}, the orbit of almost every point in the Julia set eventually stays in $\Omega_0^n$, 
for $n\geq 0$. 
This implies that almost every point in the complement of $\Omega^n_0$, for $n\geq 0$, is non-recurrent. 
As $\area\, \Omega^n_0$ shrinks to zero, almost every point in the Julia set must be non-recurrent.    
The second part follows from the first part and Poincar\'e recurrence Theorem.
\end{proof}

\begin{propo} \label{like-rotation-propo}
There exist $M>0$ and $\mu<1$ such that for all $\alf\in \irr$, all $f\in \QIS_\alf $, all $n\geq 1$, and 
all $z\in\Omega^{n+1}_0$ 
we have 
\[ |f\co{q_n}(z)-z| \leq M\cdot  \mu^n.\]  
In particular this holds on the post-critical set. 
\end{propo}

\begin{proof}
The result basically follows from the uniform contraction in Lemma~\ref{contract}, but, since we are not concerned 
with the distortions of the maps here, one may go down the tower in a simpler fashion. 
We briefly outline the procedure here and leave further details to the reader.  

Let $f_0:=f$, and $f_i:=\rr\co{i}(f)$, for $i\geq 1$. 
Given $z_0\in \Omega_0^{n+1}\setminus \{0\}$, inductively define the sequence of points 
$w_i,\zeta_i, z_{i+1}$ and non-negative integers $\sigma_i$, for $i=0,1,\dots, n$, according to the following rules. 
When $z_i\in \mathscr{A}_i':=\cup_{j=k_i+\Bk'}^{b_i}f_i\co{j}(S_i^0)$, then 
$w_i:=z_i$, and $\sigma_i:=0$. 
When $z_i \in (\Omega_i^0\setminus \mathscr{A}_i')$, one may choose $w_i\in (S_i^0 \cap \Omega_i^{n+1-i})$ and 
$\sigma_i\in \{1,2, \dots, k_i+\Bk'-1\}$ so that $f_i\co{\sigma_i}(w_i)=z_i$. 
The existence of such $w_i$ follows from the choice of the inverse branch $\psi_{i+1}$ in \eqref{WidthofLift} (that is, 
$\sup \Re \Phi_i\circ \psi_{i+1}(\p_{i+1})\leq  \inf \Re \Phi_i (S_i^0)$). 
However, $w_i$ is not necessarily unique. 
In both cases, $\zeta_i:=\Phi_i(w_i)$, $z_{i+1}:=\ex(\zeta_i)$. 

The last point $z_{n+1}\in \Omega_{n+1}^0$. 
By Lemma~\ref{the-balls}, we may choose a curve $\gamma_n:[0,1]\to \cc$ with $\gamma_n(0)=\zeta_n$, 
$\gamma_n(1)=\zeta_n+1$, and $\ex (B_{\delta_3}(\gamma_n)) \subset \Dom f_{n+1}$. 
By the distortion theorem, $\gamma_n$ may be chosen to have uniformly bounded Euclidean length, 
independent of $n$. 
Define $V_{n+1}:= B_{\delta_3} (\gamma_n\cup (\gamma_n-1))$, and note that the hyperbolic 
distance between $\zeta_n$ and $\zeta_n+1$ within $V_{n+1}$ is uniformly bounded from above, 
independent of $n$. 

We have $(V_{n+1}-1)\subset\Phi_n(\Dom \Phi_n)$, and by Lemma~\ref{renorm} and Equation~\eqref{k'-condition},
$f_n$ may be iterated $k_n+\Bk'$ times on $\Phi_n^{-1} (V_{n+1}-1)$.  
Define $g_{n+1}(\zeta):=f_n\co {\sigma_n+1} \circ \Phi_n^{-1}(\zeta-1)$ on $V_{n+1}$, and as in the 
previous argument, choose $j(n)$ in $\{0, 1, \dots, \lfloor 1/\alf_n\rfloor -\Bk-1\}$ such that 
$g_{n+1}(V_{n+1})\subset \Omega_n^0\setminus I_{n, j(n)}$. 
Let $V_{n}:= \Omega_n^0\setminus I_{n, j(n)}$. 

We have $\Re \zeta_i\in [\Bk'+1/2,\lfloor 1/\alf_i\rfloor -\Bk-1/2]$, for $i=0,1,\dots, n$.
Repeating only case $\mathscr{R}$ of the construction in Section~\ref{S:Going-Up}, 
one inductively defines the pairs $(g_{i+1}, V_i)$, for $i=n-1,n-2, \dots, 1$, such that 
$g_{i+1}:=f_{i}\co{\sigma_{i}+j}\circ \Phi_{i}^{-1} \circ (\eta_{i+1}-j)$,
for some $j\in \{0, 1,\dots, \Bk'\}$ with $\Re (\eta_i(V_{i+1})-j)\subset (0, \lfloor 1/ \alf_i\rfloor -\Bk)$. 
Moreover, $j\in \{0,1,\dots ,\lfloor 1/\alf_i\rfloor-\Bk-1\}$ is chosen so that $V_i=\Omega_i^0\setminus I_{i, j(i)}$
contains $f_i\co{\sigma_i+j}\circ \Phi_i^{-1}(B_{\delta_3}(\eta_i(V_{i+1}))-j)$. 

The composition of these maps, denoted by $\G_n$, satisfies $\G_n(\zeta_n)=z_0$. 
We claim that $\G_n(\zeta_n+1)= f_0\co{q_n}(z_0)$. 
To see this, first note that $\Phi_n^{-1}(B(\zeta_n-1,\delta_3)) \subset \p_n'$ and hence by Lemma~\ref{conjugacy}, 
$\Psi_n \circ f_n \circ \Phi_n^{-1}=f_0\co{q_n}\circ \Psi_n \circ \Phi_n^{-1}$ on $B(\zeta_n-1,\delta_3)$.   
On the other hand, by the definition of renormalization, one can see that 
$\G_n=f_0\co{s} \circ \Psi_n \circ \Phi_n^{-1}$ on $V_{n+1}$, for some non-negative integer $s$. 
The integer $s$ is non-negative because of the choices of the branches of $\psi_i$ in \eqref{WidthofLift} (that is,    
$\Re \Phi_i(\psi_{i+1} (\p_{i+1}'))\leq \Bk'+1$). 
Then, at every point $\xi\in B(\zeta_n-1,\delta_3)$ we have  
\begin{multline*}
f_0\co{q_n}(\G_n(\xi))
= f_0\co{q_n} \circ f_0\co{s} \circ \Psi_n \circ \Phi_n^{-1} (\xi) 
= f_0\co{s} \circ f_0\co{q_n} \circ \Psi_n \circ \Phi_n^{-1}(\xi)\\
=f_0\co{s} \circ \Psi_n \circ f_n\circ  \Phi_n^{-1}(\xi)
=f_0\co{s} \circ \Psi_n \circ \Phi_n^{-1}(\xi+1)
=\G_n(\xi+1). o
\end{multline*}
Since $V_{n+1}$ is connected and the above equation holds on $B(\zeta_n-1, \delta_3)\subset V_{n+1}$, 
it must hold on $V_{n+1}$. 
In particular, $\G_n(\zeta_n+1)=f_0\co{q_n}(\G_n(\zeta_n))$.
Now, by the uniform contraction of the maps $g_i$, one concludes the result.   
\end{proof}

Recall that $\Delta(f)$ denotes the Siegel disk of $f\in \QIS_\alf$ centered at $0$, provided it exists.
Theorem~\ref{thm:connectivity-Siegel-boundary} is a special case of the following proposition. 

\begin{propo} \label{equal-to-PC}
For all $\alf\in \irr$ and all $f\in  \QIS_\alf$, the following properties hold.
\begin{itemize}
\item[(1)]If $\alf$ is a non-Brjuno number, then $\pc(f)=\cap_{n=0}^\infty \Omega_0^n$.
\item[(2)]If $\alf$ is a Brjuno number, then 
\begin{itemize}
\item[a)] $\daron (\cap_{n=0}^\infty \Omega_0^n)=\Delta(f)$,
\item[b)] $\pc(f)=\cap_{n=0}^\infty \Omega_0^n\setminus \Delta(f)$, 
in particular, $\partial \Delta(f)\subseteq \pc(f)$.
\end{itemize}
\item[(3)]$\pc(f)$ is a connected set.
\end{itemize}
\end{propo}

The connectivity of $\pc(P_\alpha)$ for irrational values of $\alpha$ follows from the main result of \cite{Chi08}, although  
it is not directly stated in that paper. 
Indeed, it is proved that $\pc(P_\alpha)$ is equal to the closure of a class of hedgehogs of $P_\alpha$. 
By definition, each hedgehog of $P_\alpha$ is a connected sets. 
See also \cite{Blo10} . 
However, proving the connectivity of the post-critical set in the greater generality of the class $\ff$ is particularly important 
when studying the applications of near-parabolic renormalization scheme. 
See for instance Lemma 4.5 in \cite{CC13}. 

To prove the above proposition we need the next two Propositions, which shall be proved in 
Section~\ref{sec:metric-properties}. 

\begin{propo}\label{P:borrowed-lem}
There exists $E\in \mathbb{R}$ such that for all Brjuno $\alf\in \irr$, all $f$ in $\QIS_\alf$, and all $z$ in 
$\cap_{n=0}^\infty \Omega_0^n \setminus \overline{\Delta(f)}$, there are infinitely many positive integers 
$m$ with $\Im \zeta_m \leq \tfrac{1}{2\pi} \log \alf_{m+1}^{-1}+E$. 
\end{propo}

The set $\cap_{n=0}^\infty \Omega_0^n \setminus \overline{\Delta(f)}$ may be empty for some values of $\alf$, 
in which case the statement of the above lemma is void.

\begin{propo}\label{P:lift-of-integers}
For all $E \in \mathbb{R}$ there is $\delta_6>0$, such that 
for all $n \geq 1$, and all $\zeta\in \ex^{-1}(\Omega_{n+1}^0)$ with $\Im\zeta\leq\frac{1}{2\pi}\log \alf_{n+1}^{-1}+E$, 
there is $\zeta'\in  \ex^{-1}\circ \Phi_{n+1}^{-1}(\{1,2,\dots, \lfloor 1/(2\alf_{n+1})\rfloor\})$ with 
\[|\Re (\zeta'-\zeta)|\leq 1/2, \; d(\zeta', \zeta) \leq \delta_6.\]
\end{propo}

Combining Proposition~\ref{P:nearby-holes} and Proposition~\ref{P:borrowed-lem}, we obtain the following corollary. 

\begin{cor}\label{C:thin-hairs}
Let $\alf$ be a Brjuno number in $\irr$ and $f\in \QIS_\alf$. 
Then, the set $\cap_{n=0}^\infty \Omega_0^n \setminus \overline{\Delta(f_\alf)}$ is non-uniformly porous. 
In particular,  $\cap_{n=0}^\infty \Omega_0^n \setminus \Delta(f_\alf)$ has empty interior. 
\end{cor}

\begin{proof}[Proof of Proposition~\ref{equal-to-PC}] 
Let $f_0:=f$ and $f_n:= \rr^n(f_0)$, for $n\geq 1$. 
Also, $\alf_n$ denotes the asymptotic rotation of $f_n$ at $0$, for $n\geq 0$. 

\medskip

{\em Part (1):}
Fix a point $z_0 \in \cap_{n=0}^\infty \Omega_0^n\setminus \{0\}$, and recall the sequence of quadruples 
$\langle (z_i, w_i, \zeta_i, \sigma_i)\rangle_{i=0}^\infty$ introduced in the proof of 
Proposition~\ref{like-rotation-propo}.
Proposition~\ref{P:good-height-lem} applies to the sequence $\zeta_i$ as well and provides an increasing 
sequence of positive integers $n_i$ satisfying $\Im \zeta_{n_i}\leq \frac{1}{2\pi}\log \alf_{n_i+1}^{-1}$. 
One uses Proposition~\ref{P:lift-of-integers} (with $E=0$) to find 
\[\zeta'_{n_i} \in \ex^{-1}(\Phi_{n_i+1}^{-1}(\{1,2,\dots, \lfloor 1/(2\alf_{n_i+1}) \rfloor\}))\] 
enjoying the properties in the lemma. 
The two points $\zeta_{n_i}$ and $\zeta'_{n_i}$ are mapped to the dynamic plane of $f$ under the map 
$\G_{n_i}$ built in the proof of Proposition~\ref{like-rotation-propo}. 
Moreover, by the uniform contraction of the changes of coordinates, see Lemma~\ref{contract}, 
$\G_{n_i}(\zeta'_{n_i})$ converges to $z_0=\G_{n_i}(\zeta_{n_i})$, as $n_i$ tends to infinity. 

Elements of $\Phi_{n_i+1}^{-1}(\{1,2,\dots, \lfloor 1/(2\alf_{n_i+1})\rfloor\})$ 
belong to the orbit of the critical value of $f_{n_i+1}$, by the normalization of the coordinates $\Phi_{n_i}$. 
Then, $\Psi_{n_i}$ maps these elements into the orbit of the critical value of $f_0$, by the definition of renormalization.  
Moreover, $\G_{n_i}=f_0\co{s_{n_i}} \circ \Psi_{n_i}$, for some non-negative integer $s_{n_i}$ (see the proof of 
Proposition~\ref{like-rotation-propo}).
In particular, $\G_{n_i}(\zeta'_{n_i})$ belongs to the orbit of the critical value of $f_0$. 
Thus, $z_0\in \pc(f)$.

\medskip

{\em Part (2)-a:} 
Recall that by Proposition~\ref{neighbor}, the intersection $\cap_{n=0}^\infty \Omega_0^n$ is forward invariant 
under $f_0$, is compact, and is connected. 
Moreover, it contains $0$ and the point $\cp_{f_0}$ outside of $\Delta(f_0)$. 
It follows that $\overline{\Delta(f_0)} \subseteq \cap_{n=0}^\infty \Omega_0^n$, and 
$\Delta(f_0) \subseteq \daron(\cap_{n=0}^\infty \Omega_0^n)$.

On the other hand, by Corollary~\ref{C:thin-hairs} a point in $\cap_{n=0}^\infty \Omega_0^n \setminus \overline{\Delta(f)}$ 
may not be in the interior of the set $\cap_{n=0}^\infty \Omega_0^n$.
In other words, a point $z\in \daron(\cap_{n=0}^\infty \Omega_0^n)$ either belongs to $\Delta(f_0)$ or 
belongs to $\partial \Delta(f_0)$. 
We claim that the latter may not occur, and hence  
$\daron(\cap_{n=0}^\infty \Omega_0^n)\subseteq \Delta(f_0)$. 
Let $U_0$ be the connected component of $\daron(\cap_{n=0}^\infty \Omega_0^n)$ containing $0$. 
By the previous paragraph, $U_0$ contains $\Delta(f_0)$. 
If some $z$ in $\daron(\cap_{n=0}^\infty \Omega_0^n)$ belongs to $\partial \Delta(f_0)$, $U_0$ is strictly 
larger than $\Delta(f_0)$. 
Let $\hat{U}_0$ denote the filled-in set of $U_0$, and note that $f_0: \hat{U}_0\to \hat{U}_0$. 
Let $\psi:\hat{U}_0\to \mathbb{D}$ denote the uniformization of $\hat{U}_0$ by the unit disk mapping 
$0$ to $0$. 
By the Schwarz lemma, $\psi\circ f_0 \circ \psi^{-1}$ is a rotation of $\mathbb{D}$. 
That is, $f_0$ is conjugate to a ration on an strictly larger set than $\Delta(f_0)$, which contradicts the 
maximality of $\Delta(f_0)$. 

\medskip

{\em Part (2)-b:}  
As the orbit of $\cp_{f_0}$ is recurrent by Proposition~\ref{like-rotation-propo}, 
$\pc(f_0)\cap \Delta(f_0)$ is empty. 
By Proposition~\ref{neighbor}, we only need to show that 
$(\cap_{n=0}^\infty \Omega_0^n \setminus \Delta(f)) \subseteq \pc(f_0)$. 
Let $z \in \cap_{n=0}^\infty \Omega_0^n \setminus \Delta(f)$. 
By the previous part, $z \in \cap_{n=0}^\infty \Omega_0^n \setminus \daron(\cap_{n=0}^\infty \Omega_0^n)$, 
and hence, there exists a sequence $z_i\in \partial \Omega_0^i$, for $i=0,1,\dots$, converging to $z$. 
We shall show that there exists a sequence $w_i$, for $i=0,1,\dots$, in the orbit of the critical point of $f_0$ with 
$d(z_i,w_i)\to 0$, as $i$ tends to infinity. 
This proves that $z\in \pc(f_0)$.  

Recall the sets $\C_n:=\C_{f_n}$, $\C_n^{-1}, \dots, \C_n^{-k_n}$, and $\Csh_n:=\Csh_{f_n}$, $(\Csh_n)^{-1}, \dots, 
(\Csh_n)^{-k_n}$ introduced in Section~\ref{sec:sectors-introduced}.
To prove the above claim it is enough to show that for all $n\geq 1$ we have,  
\begin{itemize}
\item[(a')]$\partial\Omega_0^n\subset\bigcup_{j=0}^{q_n b_n+q_{n-1}} \overline{f_0\co{j}(\Psi_n(\C_n^{-k_n}))}$;
\item[(b')]$\forall j=0,1,\dots, q_n b_n+q_{n-1}$, $\exists l_j\geq 0$, with 
$\orb(\cv_{f_0}) \cap \overline{f_0\co{j}(\Psi_n(\C_n^{-k_n}))}\neq \emptyset$; 
\item[(c')] 
\[\lim_{n\to \infty}\sup \big \{\diam \overline{f_0\co{j}(\Psi_n(\C_n^{-k_n}))} \; \big | \;
0\leq j \leq q_ n b_n+q_{n-1}\big \}=0.\]
\end{itemize} 

{\em Proof of (a'):} Recall that by Theorem~\ref{Ino-Shi2}, $f_{n+1}(z)=P\circ \psi^{-1}(e^{2\pi \alf_{n+1}\Bi}\cdot z)$, 
for some univalent map $\psi:U\to e^{-2\pi \alf_{n+1}\Bi}\cdot \Dom (f_{n+1})$ with $\psi'(0)=1$.
The ellipse $E$ defined in  Section~\ref{Inou-Shishikura-class} is contained in $B(0,2)$, and hence, 
$U$ contains $B(0,8/9)$. 
The $1/4$-Theorem implies that $\psi (U)$ and $\Dom f_{n+1}$ contain $B(0, 2/9)$. 
Thus, 
\begin{equation}\label{E:inclusion-1}
\{w, \Im w\geq 2\}\subset \ex^{-1}(\Dom f_{n+1}).
\end{equation}
The definition of renormalization implies that 
\[\ex(\Phi_n((\Csh_n)^{-k_n}))=f_{n+1}^{-1}(B(0,\tfrac{4}{27}e^{-4\pi})).\] 
By an explicit estimate on the polynomial $P$ and the distortion theorem applied to the map 
$z\mapsto \tfrac{8}{9}\psi(\tfrac{9}{8}\cdot z)$ one concludes that 
$f_{n+1}^{-1}(B(0,\tfrac{4}{27}e^{-4\pi})) \subset B(0, \tfrac{4}{27}e^{4\pi})$. 
Therefore, $\ex (\Phi_n((\Csh_n)^{-k_n}))\subset B(0, \tfrac{4}{27}e^{4\pi})$, or in other words,  
\begin{equation}\label{E:inclusion-2}
\Phi_n((\Csh_n)^{-k_n}) \subset \{w\in \cc\mid \Im w > -2\}.
\end{equation}
Combining the inclusions \eqref{E:inclusion-1} and \eqref{E:inclusion-2}, one infers that 
$\partial \Omega_n^0 \cap \cup_{j=0}^{b_n} f_n\co{j} ((\Csh_n)^{-k_n})=\emptyset$. 
By Lemma~\ref{conjugacy}, this implies that $\partial \Omega^n_0 \cap \cup_{j=0}^{b_nq_n+q_{n-1}} f_0\co{j}
(\Psi_n(\Csh_n)^{-k_n})=\emptyset$.
By the definition of $\Omega_0^n$,  this finishes the proof of part (a).  

\medskip

{\em Proof of (b'):}
By the definition of renormalization, 
\[\ex(\Phi_n(\C_n^{-k_n}))=f_{n+1}^{-1}(B(0,\tfrac{4}{27}e^{4\pi})\setminus B(0,\tfrac{4}{27}e^{-4\pi}).\] 
We have $(P^{-1}(B(0, \tfrac{4}{27}e^{-4\pi}))\cap U)\subset B(0,\tfrac{8}{27}e^{-4\pi})$, and 
the distortion theorem applied to the map $z\mapsto \tfrac{8}{9}\psi(\tfrac{9}{8}\cdot z)$ 
implies that $\psi(B(0,\tfrac{8}{27}e^{-4\pi})) \subset B(0,1/10)$.   
Hence, $f_{n+1}^{-1}(B(0,\tfrac{4}{27}e^{-4\pi}) \subset B(0,1/10)$.  
On the other hand, by the previous part, we have $\psi(U)\supset B(0, 2/9)$. 
Combining these, we have $-4/27\in \ex (\Phi_n(\C_n^{-k_n}))$, or equivalently, 
\[\Phi_n(\C_n^{-k_n}) \cap \{1,2,\dots , \lfloor 1/\alf_n\rfloor-\Bk-1\}\neq \emptyset.\]    
This means that for $j=0,1,\dots, b_n$, $f_n\co{j}(\C_n^{-k_n}) \cap \orb (\cv_{f_n})\neq \emptyset$. 
By the definition of renormalization, $\Psi_n(\orb(\cv_{f_n})\cap \p_n)\subset \orb(\cv_{f_0})$. 
This finishes the proof of part (b).  

\medskip

{\em Proof of (c'):} 
Since $\psi$ has univalent extension onto $V$, the distortion theorem implies that there exists a constant $C$, 
independent of $n$, such that $\Dom f_{n+1}$ is contained in $B(0,C)$. 
On the other hand, $P(B(0,\tfrac{4}{54} e^{-4\pi})) \subset B(0,\tfrac{4}{27} e^{-4\pi})$ and by
$1/4$-Theorem, $\psi(B(0,\tfrac{4}{54} e^{-4\pi}))\supseteq B(0,\tfrac{1}{54} e^{-4\pi})$.  
Combining these inclusions with the first equation in the proof of part (b'), we obtain  
$\ex (\Phi_n(\C_n^{-k_n}))\subset B(0, C) \setminus B(0,\tfrac{1}{54}e^{-4\pi})$.
Thus,  
\[\Im \Phi_n(\C_n^{-k_n})\subseteq [\tfrac{1}{2\pi}\log(\tfrac{4}{27C}), 2+\tfrac{1}{2\pi}\log 8)].\]
We also have 
\[\Re \Phi_n(\C_n^{-k_n}) \subseteq [1/2,\lfloor 1/\alf_n\rfloor -\Bk-1/2], \text{ and }\diam\Re (\Phi_n(\C_n^{-k_n}))\leq\Bk'',\]
by the choice of $k_n$, Proposition~\ref{P:turning}, and condition~\eqref{k'-condition}. 

Let $V_{n+1}:= B_{\delta_3}(\Phi_n(\C_n^{-k_n}))$, where $\delta_3$ is the constant in Lemma~\ref{the-balls}. 
The above equations imply that $\Phi_n(\C_n^{-k_n})$ has uniformly bounded hyperbolic diameter in $V_{n+1}$, 
independent of $n$. 

For every $n\geq 1$ and every $j=0,1,\dots, b_nq_n+q_{n-1}$, there is a chain of maps as in \eqref{chain} that 
maps the closure of $\C_n^{-k_n}$ to the closure of $f_0\co{j}(\Psi_n(\C_n^{-k_n}))$. 
That is, given $j$, there are non-negative integers $\sigma_i\in \{0,1,\dots, b_i\}$, for $i=0,1,\dots, n$, 
such that $j$ times iterating $f_0$ on the closure of $\Psi_n(\C_n^{-k_n})$ breaks down to 
$\sigma_i$ times iterating $f_i$ on level $i$, for $i=0,1,\dots, n$, using the changes of coordinates. 
Then, one defines a chain of maps as in \eqref{chain} so that each $g_i$ is the composition of three maps 
$\eta_i$, $\Phi_i^{-1}$, and $f_i\co{\sigma_i}$, where $\eta_i$ is an appropriate inverse branch of $\ex$ and 
$\sigma_i\in \{0, 1, \dots, b_i\}$.
As we have used this argument several times before, here we leave further details to the reader. 
The uniform contraction in Lemma~\ref{contract} implies that the supremum exponentially tends to $0$. 
 
\medskip

{\em Part (3):}
Each $\Omega^0_n$, for $n\geq 0$ is a connected set. 
It is a finite union of connected sets (sectors) all containing $0$.
If $\alf$ is not a Brjuno number, by the first part, $\pc(f_0)$ is the intersection of this nest of connected sets, 
and hence is connected. 
If $\alf$ is a Brjuno number, each $\Omega_0^n$, for $n\geq 0$, contains the full set $\Delta(f_0)$ in its interior. 
Thus each $\Omega_0^n \setminus \Delta(f_0)$, for $n\geq 0$, is a connected set. 
Therefore, their intersection, which is the post-critical set by Part 2-a, is a connected set.  
\end{proof}

\begin{proof}[Proof of Theorem~\ref{thm:semicontinuous}]
Let $f_\alf$ be a continuous family of maps in $\QIS_\alf$. 
For values of $\alf$ in $\irr$, we may use the construction in Section~\ref{neighbor} to define the nest of domains 
$\Omega_0^n(f_\alf)$ for the map $f_\alf$.  
By Proposition~\ref{equal-to-PC}, $\Lambda (f_\alf) =\cap _{n=0}^\infty \Omega_0^n(f_\alf)$. 

Let us fix an arbitrary $\eps>0$. 
By the above paragraph, there is a positive integer $n$ such that $\Omega_0^n(f_\alf) \subseteq B(\Lambda(f_\alf), \eps/2)$. 

Let $I_n$ denote the set of rotations $\alf' \in (0, r_3]$ such that the first $n$ entries of the continued fraction expansion of 
$\alf'$ are equal to the corresponding entries of $\alf$. 
For all $\alf'$ in  $I_n$, $f_{\alf'}$ is $n$ times near-parabolic renormalizable. 
In particular, the sets $\Omega_0^n(f_{\alf'})$ are defined for all those $\alf'$. 

Recall that the Fatou coordinate of a map, defined in Proposition~\ref{Ino-Shi1}, depends continuously on the map. 
Then it follows from the definitions that for $\alf \in (0, r_3]$, the closures of the sets $\Omega_0^0(f_\alf)$ depend 
continuously on $\alf$ with respect to the Hausdorff distance on compact subsets of $\mathbb{C}$.  
Moreover, for $\alf'\in I_n$, the maps $\rr\co{i} (f_\alf)$, for $0\leq i \leq n$ are defined and depend continuously on $\alf'$. 
This implies that there is a neighborhood $J_n \subseteq I_n$ of $\alf$ such that for all $\alf' \in J_n$ we have 
$\Omega_0^n (f_{\alf'}) \subseteq B(\Omega_0^n (f_\alf), \eps/2)$. 

By the above paragraphs, for all $\alf' \in J_n$, we have $\Omega_0^n (f_{\alf'}) \subseteq B(\Lambda (f_\alf), \eps)$. 
For $\alf'$ in $J_n \cap \irr$, $\Lambda (f_{\alf'}) = \cap_{n=0}^\infty \Omega_0^n(f_{\alf'}) \subset \Omega_0^n(f_{\alf'})$. 
In particular, we must have $\Lambda(f_{\alf'}) \subseteq B(\Lambda (f_\alf), \eps)$. 
As $\eps $ was chosen arbitrarily, this finishes the proof of the theorem. 
\end{proof}
\section{Perturbed Fatou coordinates}\label{sec:perturbed}
In this section we analyze the perturbed Fatou coordinates. 
Our approach incorporates ideas from quasi-conformal mappings and the generalized Cauchy integral formula, although 
quasi-conformal mappings do not directly appear here. 
After writing this paper, the methods presented here have been further developed in \cite{Ch10-II} to prove an 
optimal infinitesimal estimate on the perturbed Fatou coordinates. 
This method is also employed in \cite{CC13} to prove some sharp estimates on the dependence of the Fatou coordinates 
on the linearity and non-linearity of the map. 

We shall work with the maps in the class
\[\QIS_\alf:=\ff_\alf\cup \{Q_\alf\},\]
where $\alf\in \mathbb{R}$, and $\ff_\alf$ as well as $\{Q_\alf\}$ are the sets of maps defined in 
Section~\ref{Inou-Shishikura-class}. 
The unique critical point of a map in $\QIS$ is used to normalized the Fatou coordinate, i.e.\ make it unique. 
However, most of the arguments presented here may be applied to more general maps, provided there is a special point 
that can be used to make the Fatou coordinate unique. 

\subsection{Unwrapping the map, pre-Fatou coordinate}
Recall that an element of $\ff_\alf$ is of the form  
\[h(z)=P \circ \vfi^{-1}(\ea \cdot z)\colon e^{-2\pi \alf \Bi}\cdot \vfi(U)\ra \cc,\] 
where $\alf\in \mathbb{R}$, as well as $\vfi: U\to \cc$ is a univalent map with $\vfi(0)=0$ and $\vfi'(0)=1$. 

\begin{lem}\label{L:very-basic-estimates-on-IS}
The domain $U$ contains $B(0, 8/9)$. 
Every $h\in \QIS_\alf$, with $\alf\in \mathbb{R}$, is univalent on $B(0, 4/27)$, and $|\cp_h|\in [4/27,  4/3]$. 
\end{lem}

\begin{proof}
Recall from Section~\ref{Inou-Shishikura-class} that the ellipse $E$ is contained in $B(0, 2)$. 
By a simple calculation, this implies that $U$ contains $B(0,8/9)$. 

The polynomial $P$ is univalent on $B(0,1/3)$. 
That is because, 
\[P(w_1)-P(w_2)= (w_1-w_2)((1+w_1+w_2)^2-w_1w_2),\]
while for $w_1, w_2\in B(0, 1/3)$, $\Re (1+w_1+w_2)^2> 1/9$ and $ -1/9<  \Re (w_1w_2)< 1/9$. 
Hence, for $w_1$ and $w_2$ in $B(0,1/3)$, $P(w_1)=P(w_2)$ only if $w_1=w_2$. 

The map $\psi(z)=  \tfrac{3}{2}\vfi(\tfrac{2}{3}\cdot z)$ is defined and univalent on the disk $|z|<1$. 
By the distortion Theorem~\ref{T:Distortions}, $\psi(B(0, 1/2))$ contains the ball of radius $(1/2)/(1+1/2)^2= 2/9$. 
Moreover, $2/9 \leq |\psi(-1/2)|\leq 2$.
This implies that $\vfi(B(0, 1/3))\supset B(0, 4/27)$ and $|\cp_h|= |\vfi(-1/3)|\in [4/27,  4/3]$. 
In particular, $h$ must be univalent on $B(0,4/27)$. 

The quadratic polynomial $Q_\alf$ has a critical point at $-8 e^{-2\pi \alf\Bi}/27$. 
By a calculation similar to the above one for the polynomial $P$, $Q_\alf$ is univalent on the ball $B(0, 8/27)$. 
\end{proof}

Recall from Theorem~\ref{Ino-Shi1} that every $h\in \QIS_\alf$, with $\alf\in (0, r_1]$, 
has a non-zero fixed point $\sigma_h\in \partial \p_h$. 
We will see in Lemma~\ref{cyl-cond} that on some fixed neighborhood of $0$ independent of $h$, $0$ 
and $\sigma_h$ are the only fixed points of $h$. 
Following \cite{Sh00}, we write
\begin{equation*} 
h(z)=z+z(z-\sigma_h)u_h(z),
\end{equation*}
with $u_h(z)$ a holomorphic function defined on $\Dom h$.  
If $0$ is a  simple fixed point of $h$, i.e.\ $h'(0)\neq 1$, $u_h(0)\neq 0$. 
Differentiating this equation at $0$, one obtains 
\begin{equation}\label{sigma}
\sigma_h=\frac{1-\ea}{u_h(0)}. 
\end{equation}

Let us define the constant 
\begin{equation}\label{E:C_1}
C_0:= \inf \{|u_h(0)| : h\in \cup_{\alf \in (0, r_1]} \ff_\alf\}.
\end{equation}

\begin{lem}\label{L:bound-on-u_h(0)}
We have $C_0>0$. 
\end{lem}

\begin{proof}
For every $\alf \in (0,r_1]$, and every $h$ in the closure of the class of maps $\ff_\alf$ 
(with respect to the uniform convergence on compact sets), $h$ has two preferred fixed points at $0$ and $\sigma_h$. 
Hence, for every such $h$, $u_h(z)$ is defined on $\Dom h$. 
Moreover, since for all those $h$, $h'(0)= e^{2\pi \Bi \alf}\neq 1$, $0$ is a simple fixed point of $h$, and so $u_h(0)\neq 0$.
This implies that for every $\alf \in (0,r_1]$, $\inf \{|u_h(0)|  ; h\in \ff_\alf\}$ is non-zero. 
On the other hand, when $\alf=0$, $\sigma_h=0$ is a double fixed point of $h\in \ff_0$, and we may define 
$u_h(z)=(h(z)-z)/z^2$. 
In particular, $u_h(0)= h''(0)/2$.
By Equation~\eqref{E:bound-on-second-derivative}, $|h''(0)|\in [2, 7]$. 
The map $h \mapsto u_h(0)$ is defined on the closure of the set of maps 
$\cup_{\alf \in [0,r_1]} \ff_\alf$, and depends continuously on $\alf$ and $h$. 
This finishes the proof of the lemma. 
\end{proof}

For the map $h=Q_\alf$, $\sigma_h=(1-\ea) \tfrac{16}{27} e^{-4\pi \alf\Bi}$. 
Define the constant 
\begin{equation}\label{E:constant-c_1}
C_1= \min \{\frac{2\pi}{C_0} , \frac{32\pi}{27}\}. 
\end{equation}
By Lemma~\ref{L:bound-on-u_h(0)} and Equation~\eqref{sigma}, for all $\alf\in (0, r_1]$ and all $h \in \QIS_\alf$, 
\begin{equation}\label{E:bound-on-sigma}
|\sigma_h| \leq C_1 \alf. 
\end{equation}

Define the map $\tau_h:\mathbb{C} \to \hat{\mathbb{C}} \setminus \{0, \sigma_h\}$, where $\hat{\mathbb{C}}$ 
denotes the Riemann sphere, by the formula
\begin{equation}\label{E:covering-formula}
\tau_h(w):=\frac{\sigma_h}{1-e^{-2\pi \Bi \alf w}}.
\end{equation} 
This is a covering transformation from $\mathbb{C}$ onto $\hat{\mathbb{C}} \setminus \{0, \sigma_h\}$ with the 
deck transformation group generated by the map $w \mapsto w+1/\alf$, that is, $\tau_h(w+1/\alf)=\tau_h(w)$. 
Moreover, we have $\tau_h (w)\to 0$ as $\Im w \to +\infty$, and $\tau_h (w) \to \sigma_h$ as $\Im w \ra -\infty$.

For $\alf \in (0, r_1]$ and $R\in (0, +\infty)$, define the set 
\[\Theta_\alf(R):=\cc\setminus\cup_{n \in \mathbb{Z}} B(n/\alf,R).\]

\begin{lem}\label{strip-image} We have 
\begin{itemize}
\item[(1)] for all $\delta>0$, all $\alf\in (0, \min\{r_1, \delta/(2C_1)\}]$, and all $h\in\QIS_\alf$,  
\[\tau_h(\Theta_\alf(C_1/(\pi \delta))) \subset B(0, \delta).\] 
\item[(2)]for all $\alf\in (0, r_1]$, all $h\in\QIS_\alf$, and all $r\in (0,1/2]$, we have 
\[\sup_{w\in \Theta_\alf (r/\alf)}  |\tau_h(w)|\leq e^{2\pi} C_1 \frac{\alf}{r} e^{-2\pi\alf \Im w}.\] 
\end{itemize}
\end{lem}

\begin{proof}
{\em Part (1)} Fix $\delta>0$ and let $\eps= \min\{r_1, \delta/(2C_1)$ and $R= C_1/(\pi \delta)$. 
For all $\alf\in (0,\eps]$ and all $w \in \partial B(0, R)$, 
$|1-e^{-2\pi \Bi \alf w}|\geq 1-e^{-2\pi \alf R}$. 
Also, by elementary calculations, for all $\alf\in [0,\eps]$, 
$1-e^{-2\pi \alf R} -\pi \alf R\geq 0$ 
(This holds at the end points $0$ and $\eps$ as well as at the critical point $\delta \log 2 /(2C_1)$).  
Thus, by Equation~\eqref{E:bound-on-sigma}, and the periodicity of $\tau_h$, for $w \in \partial \Theta_\alf(0, R)$, 
\[|\tau_h(w)|\leq |\frac{\sigma_h}{1-e^{-2\pi\alf R}}|\leq \frac{C_1 \alf}{\pi \alf R}=\delta.\]  

On the other hand, for every $\alf\in (0, \eps]$, and every $h\in \QIS_\alf$, as $w\in \Theta_\alf(R)$ 
tends to $+\Bi \infty$, $\tau_h(w)\to 0$, and as $w\to -\Bi \infty$, $\tau_h(w)\to \sigma_h$, 
where $|\sigma_h|\leq C_1 \alf\leq \delta/2$. 
Therefore, by the maximum principle, the inequality in part 1 holds on $\Theta_\alf (R)$. 

\medskip 
 
{\em Part (2)} 
For $w\in \partial \Theta_\alf(r/\alf)$, we have 
\[|1-e^{-2\pi \Bi\alf w}|\geq  1-e^{-2 \pi r}  \geq  e^{-2\pi} r e^{2\pi r} \geq e^{-2\pi} r e^{2\pi \alf \Im w}.\] 
Also, as $\Im w\to \pm\infty$, $ |1-e^{-2\pi\Bi \alf w}|\geq e^{-2\pi} r e^{2\pi \alf \Im w}$ holds.
By the maximum principle, for $w \in \Theta_\alf(r/\alf)$ we must have $ |1-e^{-2\pi\Bi \alf w}|\geq e^{-2\pi} r e^{2\pi \alf \Im w}$. 
Then, by Equation~\eqref{E:bound-on-sigma}, 
\[|\tau_h(w)|=\frac{|\sigma_h|}{|1-e^{-2\pi\Bi \alf w}|}\leq C_1 e^{2\pi} \frac{\alf}{r} e^{-2\pi\alf \Im w}.\] 
\end{proof}

One may lift $h$ under $\tau_h$ to obtain a map $F_h$ defined near $+\Bi \infty$ and $-\Bi \infty$. 
Any such lift satisfies
\begin{equation}\label{lift-relation}
h \circ \tau_h(w)=\tau_h \circ F_h(w), \;F_h(w)+\alf^{-1}= F_h (w + \alf^{-1}),  
\end{equation}
wherever they are defined. 
In the next lemma we analyze the domain of definition of $F_h$ and its asymptotic behaviors near 
$+\Bi \infty$ and $-\Bi \infty$. 
The plan is to study how the perturbed Fatou coordinate of $h$, $\Phi_h$, compares with an appropriate inverse branch 
of $\tau_h$.


\subsection{Estimates on the lift $F_h$}

\begin{lem}\label{cyl-cond}
There are constants $r_2'>0$, $C_2$, and $C_3$ such that for all $\alf\in (0, r_2']$, and all $h \in \QIS_\alf$, 
there exists a lift of $h$ under $\tau_h$, denoted by $F_h$ such that 
\begin{enumerate}
\item[(1)] $F_h$ is defined and univalent on $\Theta_\alf (C_2)$; 
\item[(2)] for all $w \in \Theta_\alf (C_2)$ we have 
\begin{equation*}\label{I}
|F_h(w)-(w+1)|<1/4, \quad |F_h'(w)-1|<1/4;
\end{equation*} 
\item[(3)] for all $r\in (0,1/2]$ and all $w\in \Theta_\alf(r/\alf)\cap \Theta_\alf(C_2)$ we have  
\begin{equation*}\label{II}
|F_h(w)-(w+1)|< C_3 \frac{\alf}{r} e^{-2\pi\alf \Im w}, \quad |F'_h(w)-1|< C_3 \frac{\alf}{r} e^{-2\pi\alf \Im w}.
\end{equation*}
\end{enumerate}
\end{lem}

\begin{proof}  
{\em Part (1):}
By Lemma~\ref{L:very-basic-estimates-on-IS} every $h\in \QIS_\alf$, with $\alf\in{\mathbb{R}}$, 
is defined and is univalent on $B(0,4/27)$. 
We apply Lemma~\ref{strip-image}-1 with $\delta=4/27$ to obtain $\eps_1 = \min\{r_1, 2/(27 C_1)\}$ 
such that for all $\alf \in (0,\eps_1]$ and $h\in \QIS_\alf$ we have 
$\tau_h(\Theta_\alf(27 C_1/(4\pi)))\subset B(0,4/27)$. 
In particular, $\sigma_h= \lim_{\Im w \to -\infty} \tau_h(w)$ must be in $B(0,4/27)$. 
On the other hand, as $h$ is univalent on $B(0, 4/27)$, $0$ and $\sigma_h$ are the only pre-images of $0$ and $\sigma_h$ in 
$B(0, 4/27)$. 
This implies that any lift of $h$ under $\tau_h$ is a well-defined, finite, and univalent function on $\Theta_\alf (27 C_1/(4\pi))$. 
There are many choices for this lift, but we may choose one with 
\begin{equation}\label{E:choice-of-lift}
\lim_{\Im w\to +\infty} (F_h(w)-w)=1.
\end{equation}
Since $\eps_1 \leq  2 \pi / (27C_1)$, $\Theta_\alf(27 C_1/(4\pi))$ is connected and the normalization 
of $F_h$ near $+\Bi \infty$ uniquely determines the lift $F_h$ on $\Theta_\alf (27 C_1/(4\pi))$.  
Although $F_h$ may be defined beyond $\Theta_\alf(27 C_1/(4\pi))$, it may have singularities outside 
of $\Theta_\alf(27 C_1/(4\pi))$ when $h^{-1}(\sigma_h)\setminus \{\sigma_h\}$ is not empty.  

\medskip

{\em Part (2):}
Using the first formula in Equation~\eqref{lift-relation}, one can see that $F_h$ is given by the formal expression
\begin{align*}
F_h(w)&=w+\frac{1}{2\pi\alf \Bi}\log\big(1-\frac{\sigma_hu_h(z)}{1+zu_h(z))}\big) \\
&=w+1+ \frac{1}{2\pi\alf \Bi}\log\big( e^{-2\pi \alf \Bi} (1-\frac{\sigma_hu_h(z)}{1+zu_h(z)})\big), \text{ with } z=\tau_h(w).
\end{align*}
The branch of $\log$ on the second line of the above equations is determined by $-\pi < \Im \log(\cdot) < \pi$. 
This is consistent with our choice of the branch in Equation~\eqref{E:choice-of-lift}.  

Recall that by Lemma~\ref{L:very-basic-estimates-on-IS} every $h\in \QIS_\alf$ is defined on $B(0,4/27)$. 
Moreover, by the distortion Theorem~\ref{T:Distortions}, $\cup_{\alf\in [0, r_1]} \QIS_\alf$ forms a pre-compact
class of maps. 
By the continuous dependence of $u_h$ on $h$, the restrictions of $u_h$ to $B(0,4/27)$ must also form a pre-compact 
class. 
This implies that there exists a positive constant $\delta_1\leq  4/27$ such that 
\begin{equation}\label{upper-bound-on-expression}
\forall z\in B(0,\delta_1), \quad \Big|1-\frac{u_h(z)}{(1+zu_h(z))u_h(0)}\Big|< \frac{1}{4\pi}.
\end{equation}
Using Lemma~\ref{strip-image}-1 with $\delta=\delta_1$ we obtain 
\[\eps_2=\min \{r_1, \delta_1/(2C_1)\}\] such that for all $w\in \Theta_\alf (C_1/(\pi \delta_1))$, all $\alf\in (0, \eps_2]$ 
and all $h\in \QIS_\alf$, one has $|z|=|\tau_h(w)|<\delta_1$. 
Replacing $\sigma_h$ by the expression in \eqref{sigma} and using $|1-\ea|<2\pi\alf$, we obtain  
\begin{align*}\label{expression}
\Big| e^{-2\pi \alf\Bi} (1-\frac{\sigma_hu_h(z)}{1+zu_h(z)})-1\Big| 
&=\Big |1-\frac{\sigma_hu_h(z)}{1+zu_h(z)}-\ea \Big| \\
&=\Big|(1-\ea)(1-\frac{u_h(z)}{(1+zu_h(z))u_h(0)})\Big| \\
&< 2\pi \alf \cdot \frac{1}{4\pi} < \frac{1}{2}.
\end{align*}
In particular, for all $z \in \tau_h(\Theta_\alf (C_1/(\pi \delta_1)))$, in the expression for $F_h$ the branch of $\log$ 
is well defined.
Note that for all $x\in B(1, 1/2)$ we have $|\log x | \leq 2|x-1|$. 
Then, for all $w\in \Theta_\alf(C_1/(\pi \delta_1))$, using the above inequality we get
\begin{align*}
|F_h(w)-(w+1)|
&=\Big |\frac{1}{2\pi\alf \Bi} \log\Big(e^{-2\pi \alf\Bi}(1-\frac{\sigma_hu_h(z)}{1+zu_h(z)})\Big)\Big| \\
&\leq \frac{1}{2\pi \alf}\cdot  2 \cdot (2\pi \alf \cdot \frac{1}{4\pi}) < \frac{1}{4}.
\end{align*}  
This finishes the proof of the first inequality in part 2. 
On the other hand, for $w_0$ in $\Theta_\alf(C_1/(\pi \delta_1)+1)$, the map $F_h(w)-w-1$ is defined on 
$B(w_0, 1)$ and its absolute value is bounded by $1/4$. 
By the Schwarz-Pick lemma, we must have $|F_h'(w_0)-1|< 1/4$. 
We may define 
\[C_2:=C_1/(\pi \delta_1)+1.\]

\medskip

{\em Part (3):}
By Equation~\eqref{upper-bound-on-expression}, for all $z\in \partial B(0,\delta_1)$  we have,     
\[\Big | 1-\frac{u_h(z)}{(1+zu_h(z))u_h(0)} \Big |< \frac{1}{4\pi} =  \frac{1}{4 \pi \delta_1} |z|.\] 
As $z$ tends to $0$, the expression on the left hand of the above equation tends to $0$. 
By the maximum principle, the above inequality must hold on $B(0, \delta_1)$. 
Therefore, for $w$ in $\Theta_\alf(r/\alf) \cap \Theta_\alf(C_1/(\pi \delta_1))$ we have 
\begin{align*}
|F_h(w)-(w+1)|
&=\Big |\frac{1}{2\pi\alf \Bi} \log\Big(e^{-2\pi \alf\Bi}(1-\frac{\sigma_hu_h(z)}{1+zu_h(z)})\Big)\Big| \\
&\leq \frac{1}{2\pi \alf}\cdot  2 \cdot  \Big| e^{-2\pi \alf\Bi} (1-\frac{\sigma_hu_h(z)}{1+zu_h(z)})-1\Big| \\
&\leq \frac{1}{2\pi \alf}\cdot  2 \cdot 2\pi \alf \cdot \Big | 1-\frac{u_h(z)}{(1+zu_h(z))u_h(0)} \Big | 
\leq \frac{|z|}{2 \pi \delta_1}. 
\end{align*}
By the estimate in Lemma~\ref{strip-image}-2, we obtain the first inequality in part 3 with 
\[C_3=\frac{e^{2\pi} C_1}{2 \pi \delta_1}.\]
The second inequality is similarly proved using the Schwarz-Pick lemma.

The polynomial $P(z)= z(1+z)^2$ belongs to $\ff$. 
This implies that $C_0\leq u_P(0)=2$. 
Then, 
\[\eps_2 \leq \frac{\delta_1}{2 C_1} \leq \frac{4}{27} \cdot \frac{1}{2} \cdot \max \{ \frac{C_0}{2 \pi} , \frac{27}{32 \pi} \} 
\leq \frac{2}{27 \pi}.\]
This guarantees that 
\begin{equation}\label{E:alpha_2-first-condition} 
2 C_2+ \frac{3}{2}= 2 (\frac{C_1}{\pi \delta_1}+1)+ \frac{3}{2} \leq \frac{1}{\pi \eps_2}+\frac{7}{2} 
\leq \frac{1}{\eps_2} \leq \frac{1}{\alf}.
\end{equation}
In particular, for $\alf \in (0, \eps_2]$, the set $\Theta_\alf (C_2)$ is connected.  
\end{proof}


\subsection{Fatou coordinate of $F_h$, $L_h$}\label{SS:Fatou-coordinate-of-F_h}
Let $h\in \QIS_\alf$, $\alf\in (0, r_1]$, with the perturbed Fatou coordinate 
$\Phi_h: \p_h\to \cc$ introduced in Theorem~\ref{Ino-Shi1}.
The set $\tau_h^{-1}(\p_h)$ is periodic of period $1/\alf$ and is contained in $\mathbb{C} \setminus (\mathbb{Z}/\alf)$. 
Each connected component of $\tau_h^{-1}(\p_h)$  is simply connected and is bounded by piece-wise analytic curves going 
from $-\Bi \infty$ to $+\Bi\infty$. 
Let $\widetilde{\p}_h$ denote the connected component of $\tau_h^{-1}(\p_h)$ separating $0$ from $1/\alf$. 
We have a univalent map
\begin{equation}\label{linearize}
L_h:=\Phi_h\circ \tau_h: \widetilde{\p}_h \to \cc.
\end{equation}

The map $h: \p_h \to \mathbb{C} \setminus \{0, \sigma_h\}$ lifts to a univalent map $F_h: \widetilde{\p}_h \to \mathbb{C}$ 
that agrees with the lift we chose in Lemma~\ref{cyl-cond}. 
There is a unique point $\cp_{F_h} \in \partial \widetilde{\p}_h$ with $\tau_h(\cp_{F_h})=\cp_h$.
As $F_h$ is univalent on $\Theta_\alf(C_2)$, $\cp_{F_h}$ must lie in $B(0,C_2) \cup B(1/\alf ,C_2)$. 
Because of orientation, $\cp_{F_h}$ must lie in $B(0,C_2)$. 

By Theorem~\ref{Ino-Shi1}, $L_h$ satisfies the following properties:
\begin{itemize}
\item[1)] $L_h(\cp_{F_h})=0$; 
\item[2)] as $\Im (w)\to +\infty$ in $\widetilde{\p}_h$, $\Im L_h(w)\to +\infty$, and 
as $\Im (w)\to -\infty$ in $\widetilde{\p}_h$, $\Im L_h(w) \to -\infty$; 
\item[3)] $L_h(\widetilde{\p}_h) \supset \{w\in\cc : 0< \Re(w) \leq 1\}$;
\item[4)] If $w$ and $F_h(w)$ belong to $\widetilde{\p}_h$, then    
\begin{equation}\label{equivariance}
L_h(F_h(w))=L_h(w)+1.
\end{equation} 
\end{itemize}
We wish to estimate the map $L_h$ on $\widetilde{\p}_h$. 
To this end, we need to extend $L_h$ onto a larger domain containing $\widetilde{\p}_h$. 
This can be done using the above functional equation, carried out in the remaining of this section. 

By Lemma~\ref{cyl-cond}-2, for every $w\in \Theta_\alf(C_2)$, 
\begin{equation}\label{E:slopes-of-X}
|\arg (F_h(w)-w)|<\arcsin (1/4) < \arcsin ((\sqrt{6}-\sqrt{2})/4)=\pi/12.
\end{equation}

Let $a=C_2 e^{5\pi \Bi/12}$, and $\overline{a}$ denote the complex conjugate of $a$. 
Define the piece-wise analytic curve 
\begin{multline*}
\ell=  \{w\in \mathbb{C} : \arg (w- a)=11\pi/12\}\\
\cup \{w\in \mathbb{C} : \arg (w- \overline{a})=-11\pi/12\} \\
\cup \{C_2 e^{2\pi \Bi \theta} : \theta\in [-5 \pi /12, 5 \pi /12]\}.
\end{multline*}
Let $\ell'$ denote the curve $-\ell+ 1/\alf$.  
The union of $\ell$ and $\ell'$ decomposes $\cc$ into three connected components. 
Let
\[A_1\]
denote the connected component of $\cc \setminus (\ell \cup \ell')$ that contains $1/ (2\alf)$.  
The set $A_1$ is connected and simply connected. 

The line $\ell$ intersects the vertical line $\Re w =0$ at points denoted by $b$ and $\overline{b}$, with 
$\Im b > \Im \overline{b}$. 
Let $\theta_0:= \arcsin(1/4)/2+ \pi/24 \in (\arcsin (1/4), \pi/12)$, and define 
\begin{equation*}
A:= \{w\in \cc \mid \arg (w-b)\in [\pi-\theta_0, 11\pi/12]\}\}. 
\end{equation*}
Set
\[A_2:= A_1 \cup A \cup s(A) \cup (-A+1/\alf) \cup (-s(A)+1/\alf),\]
where $s(w)$ denotes the complex conjugate of $w\in \cc$. 

Each of the curves $L_h^{-1}(\Bi \mathbb{R})$ and $L_h^{-1}(1+\Bi \mathbb{R})$ divide $\cc$ into two 
connected components; say $R$ for the right hand connected component of $\cc\setminus L_h^{-1}(\Bi \mathbb{R})$ 
and $L$ for the left hand connected component of $\cc\setminus L_h^{-1}(1+\Bi \mathbb{R})$. 
Also, let $L'$ and $R'$ denote the corresponding components of $\cc\setminus A_2$. 
These are open subsets of $\cc$. 
Define 
\[A_3:=(R\cap L')\cup (L\cap R'), \footnote{The set $L\cap R'$ is likely to be empty. 
We show in Lemma~\ref{L:domain-of-E-precoordinate} that this is the case when $\alf$ is small enough.}.\] 
The real analytic curves bounding $A_2$ intersect the curves $L_h^{-1}(\Bi \mathbb{R})$ and 
$L_h^{-1}(1+\Bi \mathbb{R})$  at most in a finite number of places. 
Thus, $A_3$ is a union of a finite number of simply connected domains

Set
\[ X:=\daron(A_2 \cup A_3).\]
The set $X$ is connected, $\cp_{F_h}\in \partial X$, $\cv_{F_h}:=F_h(\cp_{F_h})\in X$, and 
$L_h^{-1}((0,1)+\Bi \mathbb{R}) \subset X$.  

\begin{lem}\label{L:extending-F_h}
For every $\alf\in (0, r_2']$ and every $h\in \QIS_\alf$ we have  
\begin{itemize}
\item[(1)] $X$ is simply connected and $F_h$ is defined and is univalent on $X$; 
\item[(2)] for all $w\in X$, there are integers $m_w< n_w$ such that for all integers $j$ with $m_w< j< n_w$
we have $F_h\co{j}(w)\in X$, but for $j \in \{m_w, n_w\}$ we have $F_h\co{j}(w)\notin X$.  
\item[(3)] for all $w \in X$, there exists a unique integer $j_w \in [m_w+1, n_w-1]$ with 
$F_h\co{j_w}(w)\in L_h^{-1}((0,1]+\textnormal{\Bi}\mathbb{R})$;
\item[(4)] for all $w\in B(0, C_2)\cap X$, there is a positive integer $l_w$ with $\Re F_h\co{l_w}(w)\geq C_2$.
Moreover, $l_w$ is uniformly bounded from above independent of $\alf$, $h$, and $w$. 
\end{itemize}
\end{lem}

\begin{proof}
{\em Part (1):} 
As $A_2 \subset \Theta_\alf(C_2)$, by Lemma~\ref{cyl-cond}, $F_h$ is defined on $A_2$. 
The main issue we face here is to show that $0$ and $1/\alf$ do not belong to $A_3$. 

Fix $h_0\in \ff_0\cup \{Q_0\}$ and for $\alf\in (0, r_2']$ let $h_\alf(z):=h_0(\ea z)$ be an element of $\QIS_\alf$. 
When $\alf\to 0$, $\tau_{h_\alf}$ converges uniformly on compact sets to the map $\tau_{h_0}(w):=-2/(h_0''(0) w)$, 
and $h_0$ may be lifted under $\tau_{h_0}$ to a unique map $F_{h_0}$.  
Recall the petal $\p_{h_0}$, the Fatou coordinate $\Phi_{h_0}$, and the relation 
$\Phi_{h_0}(\p_{h_0})=(0, +\infty) +\Bi \mathbb{R}$ from Theorem~\ref{Ino-Shi0}. 
Under $\tau_{h_0}$, $\p_{h_0}$ lifts to a set $\widetilde{\p}_{h_0}$ and $\Phi_{h_0}$ lifts to the univalent 
map $L_{h_0}: \widetilde{\p}_{h_0} \to \mathbb{C}$. 
Then, we have $L_{h_0}(\widetilde{\p}_{h_0})=(0, +\infty)+\Bi \mathbb{R}$ and $L_{h_0}(F_{h_0}(w))= L_{h_0}(w)+1$, 
for all $w\in \widetilde{P}_{h_0}$.
The map $L_{h_0}^{-1}$ extends onto $[0,+\infty)+\Bi \mathbb{R}$ and its image covers the right hand 
connected component of $\cc\setminus L_{h_0}^{-1}(\Bi \mathbb{R})$. 
On the other hand, as $\p_{h_0}$ is compactly contained in $\Dom h_0$, $0$ does not belong to the 
closure of $\widetilde{\p}_{h_0}$. 
Thus, $0$ must be in the left hand connected component of $L_{h_0}^{-1}(\Bi \mathbb{R})$. 
For all $\alf\in (0, r_2']$, $0 \notin \tau_{h_\alf}^{-1}(\Dom h \setminus \{0,\sigma_{h_\alf}\})$. 
Hence, by the continuous dependence of $L_{h_\alf}$ on $\alf$, for all $\alf\in (0, r_2']$,  
$0$ must belong to the left hand connected component of $L_{h_\alf}^{-1}(\Bi \mathbb{R})$.
Similarly, this implies that for all $h\in \cup_{\alf\in (0, r_2']}\QIS_\alf$, $1/\alf$ belongs to the right-hand
side connected component of $\cc\setminus L_h^{-1}(1+\Bi \mathbb{R})$. 

Recall the sets $A_3= (R\cap L')\cup (L\cap R')$. 
By the above paragraph, $0\in L$ and $1/\alf\in R$. 
Using the periodicity of $F_h$ and the above argument, one may prove that $-n/\alf \in L-n/\alf$, 
for $n\in \mathbb{N}$, and $n/\alf \in R+n/\alf$, for $n\in \mathbb{N}$.
Therefore, $-n/\alf\in L$, for $n\in \mathbb{N}\cup \{0\}$ and $n/\alf\in R$, for $n\in \mathbb{N}$. 
Clearly, $-n/\alf\notin R'$, for $n\in \mathbb{N}\cup\{0\}$, and $n/\alf\notin L'$, for $n\in\mathbb{N}$.
Putting these together, we deduce that for all $n\in \mathbb{Z}$, $n/\alf \notin A_3$. 

Let $U_h= \Dom h \subset \mathbb{C}$. 
The set $\hat{C} \setminus U_h$ lifts under $\tau_h$ to countably many simply connected 
components each containing a unique $n/\alf$ for some $n\in \mathbb{Z}$. 
Since every component of $A_3$ is a simply connected region whose boundary is contained in $\tau_h^{-1}(U_h)$,
and $A_3$ avoids $\mathbb{Z}/\alf$, $A_3 \subset \tau_h^{-1}(U_h)$.  
This implies that $F_h$ is defined on $A_3$, except possibly on a discrete set of singularities. 
Such singularities might arise at $\tau_h^{-1}$ of $h^{-1}(\{0,\sigma_h\}) \setminus \{0, \sigma_h\}$, if such points 
lie in $\tau_h (A_3)$.  
Thus, so far we know that such singularities may not occur on $\partial A_3$, and hence may not occur on $\partial X$. 

Recall the polynomial $P(z)=z (1+z)^2$ in the definition of the class $\ff$. 
There is $\eps>0$ such that $B(0, \eps)$ is only covered once by the map $P: U \to \mathbb{C}$, that is, 
$P: P^{-1}(B(0, \eps)) \cap U \to B(0,\eps)$ is a homeomorphism.
When $\alf$ is small enough, $\sigma_h$ belongs to $B(0, \eps)$. 
Thus, $0$ and $\sigma_h$ must be the only pre-images of $0$ and $\sigma_h$ within $U_h$. 
Combining with the above paragraph, for small values of $\alf$, there is no singularity of $F_h$ within $X$. 
One the other hand, $\partial X$ has continuous dependence on $\alf$ and $h$. 
The singularities also depend continuously on $\alf$ and $h$, when they exist, and do not hit $\partial X$ by the above 
paragraph. 
Therefore, for all $\alf \in (0, r_2']$ and $h \in \QIS_\alf$, $F_h$ has no singularity in $X$. 

The set $A_2$ is simply connected and is bounded by a finite number of analytic curves. 
The set $X$ is formed of attaching to $A_2$ a finite number of simply connected domains that share a 
connected and analytic boundary curve with $A_2$.  
This implies that $X$ is simply connected. 
By the same argument $F_h(X)$ is also simply connected.

Let $u: X \to B(0,1)$ and $u': F_h(X) \to B(0,1)$ be some uniformizations. 
The map $u' \circ F_h \circ u^{-1} : B(0,1) \to B(0,1)$ is a proper mapping. 
So, it must be a finite Blaschke product.
But, $F_h^{-1}$ is uniquely defined near $\pm \Bi \infty$.
Thus, the blaschke product must be of degree one, or in other words, $F_h : X \to F_h(X)$ is a homeomorphism. 

\medskip

{\em Part (2):}
Let $\beta_l, \beta_r: \mathbb{R}\to \partial X$ be the piece-wise analytic curves bounding $X$, with 
$\beta_l$ on the left side of $\beta_r$ as well as $\Im \beta_l(t)$ and $\Im \beta_r(t)$ tending  
to $+\infty$ as $t\to +\infty$.  
There are four points $t_i\in \mathbb{R}$, for $i=1,2,3,4$, such that 
\begin{gather*}
\forall t\in (t_1, t_2), \Im \beta_l(t_1) < \Im \beta_l(t) < \Im \beta_l(t_2), \Im \beta_l(t_1)\leq -C_2  
, \Im \beta_l(t_2)\geq C_2,\\
\forall t\in (t_3, t_4), \Im \beta_r(t_3) < \Im \beta_r(t) < \Im \beta_r(t_4), 
\Im \beta_r(t_3)\leq -C_2, \Im \beta_r(t_4)\geq  C_2.
\end{gather*}

Moreover, one may choose $t_2$ and $t_4$ arbitrarily large, as well as $t_1$ and $t_3$ arbitrarily small. 
Then adjust the curves $\beta_r$ and $\beta_l$ into simple curves $\hat{\beta}_r$ and $\hat{\beta}_l$ as follows. 

\begin{gather*}
\hat{\beta}_l(t):= 
\begin{cases}
\beta_l(t_2)+ (t-t_2)\Bi & \text{ if } t\geq t_2 \\
\beta_l(t)  & \text{ if } t\in (t_1, t_2) \\
\beta_l(t_1)+ (t-t_1)\Bi  & \text{ if } t\leq t_1
\end{cases} 
,
\hat{\beta}_r:= 
\begin{cases}
\beta_r(t_4)+ (t-t_4)\Bi & \text{ if } t\geq t_4 \\
\beta_r(t)  & \text{ if } t\in (t_3, t_4) \\
\beta_r(t_3)+ (t-t_3)\Bi  & \text{ if } t\leq t_3
\end{cases} 
\end{gather*}
Let $B$ denote the region bounded by the two curves $\hat{\beta}_l$ and $\hat{\beta}_r$. 
As $t_1, t_3 \to -\infty$ and $t_2, t_4\to +\infty$, the corresponding sets $B$ exhaust $X$. 
Note that $F_h$ maps $\hat{\beta}_l$ into $B$ and $F_h^{-1}$ maps $\hat{\beta}_r$ into $B$.
The domain $B':= B \cap F_h^{-1} (B)$ is simply connected and bounded by $\hat{\beta}_l$ and 
$F_h^{-1}(\hat{\beta}_r)$. 

There is a harmonic function $\mathfrak{u}: B\to (0,1)$ such that $\mathfrak{u}(w)\to 0$ as $w\to \hat{\beta}_l$  
and $\mathfrak{u}(w) \to 1$ as $w\to \hat{\beta}_r$. 
Near the upper end of $B$, $\mathfrak{u}(w)$ tends to a linear function of $\Re w$, 
that is, as $\Im w\to +\infty$, $
\mathfrak{u}(w)$ tends to the function $(\Re w- \Re \hat{\beta}_l(t_2))/(\Re \hat{\beta}_r(t_4)-\Re \hat{\beta}_l(t_2))$.  
That is because, as $\Im w\to +\infty$, the probability of a Brownian motion in $B$ that starts at height $\Im w$ to hit the height 
$\max \{\Im \hat{\beta}_r(t_4), \Im \hat{\beta}_l(t_2)\}$ tends to zero. 
Similarly, near the lower end of $B$, $\mathfrak{u}$ tends to a linear function of $\Re w$. 

Consider the harmonic function $\mathfrak{u}_1: B'\to \mathbb{R}$ defined as 
$\mathfrak{u}_1(w):= \mathfrak{u}(F_h(w))-\mathfrak{u}(w)$. 
We claim that the infimum of $\mathfrak{u}_1$ on $B'$ is strictly positive. 
By the maximum principle, we only need to show this on the boundary of $B'$. 
At $w\in \hat{\beta}_l$, $\mathfrak{u}_1(w)=\mathfrak{u}(F_h(w))>0$, and at $w\in F_h^{-1}(\hat{\beta}_r)$,  
$\mathfrak{u}_1(w)=1- \mathfrak{u}(w)>0$. 
By the above paragraph, near the two ends of $B$, $\mathfrak{u}(w)$ tend to some linear functions of $\Re w$ 
and we have $|F_h(w)-w-1|\leq 1/4$. 
This implies that $\mathfrak{u}_1(w)$ is uniformly bounded away from $0$ when $|\Im w|$ is large enough. 
This finishes the proof of the claim. 

By the above paragraph, the forward orbit of every point in $B$ eventually leaves $B$ on the right hand of 
$\hat{\beta_r}$ and the backward orbit of every point in $B$ eventually leaves $B$ on the left hand 
of $\hat{\beta}_l$.
As the sets $B$ exhaust $X$, we conclude the same statement about every orbit in $X$.  
In particular,  any orbit in $X$ must cross the closure of the region bounded 
by the two curves $L_h^{-1}(\Bi \mathbb{R})$ and $F_h(L_h^{-1}(\Bi \mathbb{R}))$. 

\medskip

{\em Part (3):}
By the previous part, for every $w\in X$ both forward and backward orbit of $w$ under $F_h$ leave $X$. 
In particular, every such orbit must hit $L_h^{-1}((0,1]+\textnormal{\Bi}\mathbb{R})$. 

To see the uniqueness, assume that for some $j_1, j_2\in \mathbb{Z}$ and $w\in X$, $F_h\co{j_1}(w)$ 
and $F_h\co{j_2}(w)$ belong to $L_h^{-1}((0,1]+\Bi \mathbb{R})$. 
Then $F_h\co{j_2-j_1}$ maps the point $F_h\co{j_1}(w)$ in $L_h^{-1}((0,1]+\Bi \mathbb{R})$ into 
$L_h^{-1}((0,1]+\Bi \mathbb{R})$.
However, by Equation~\eqref{equivariance}, $F_h$  maps every point to the right of $L_h^{-1}(\Bi \mathbb{R})$ 
to a point to the right of $L_h^{-1}(1+ \Bi \mathbb{R})$, and $F_h^{-1}$  maps every point to the left of 
$L_h^{-1}(1+\Bi \mathbb{R})$ to a point to the left of $L_h^{-1}(\Bi \mathbb{R})$. 
Thus, we must have $j_1=j_2$, which proves the uniqueness of $j_w$.

\medskip

{\em Part (4):}
Since the forward orbit of every point in $X$ under $F_h$ leaves the domain $X$ on the right hand, 
the existence of $l_w$ follows. 
The uniform bound on $l_w$ is a result of the pre-compactness of the class $\cup_{\alf\in [0, r_2']}\QIS_\alf$.
Indeed, as a sequence $h_i$ tends to some map $h$ (not necessarily in the class $\cup_{\alf\in [0, r_2']}\QIS_\alf$), 
the sequence of Fatou coordinates $\Phi_{h_i}$ converges uniformly on compact sets to some univalent map $\Phi_h$
that enjoys the same equivariance property on its domain of definition.
Further details are left to the reader. 
\end{proof}

\begin{lem}\label{L:extending-L_h}
For all $\alf\in (0, r_2']$ and all $h\in \QIS_\alf$, $L_h$ has a unique univalent extension onto $X$. 
In particular, when both $w$ and $F_h(w)$ belong to $X$, $L_h(F_h(w))= L_h(w)+1$.
\end{lem}

\begin{proof}
By Lemma~\ref{L:extending-F_h}, for all $w\in X$ there is a unique integer $j_w$ with 
$F_h\co{j_w}(w)$ belongs to $L_h^{-1}((0, 1]+\Bi \mathbb{R})$. 
Define $L_h (w):=L_h(F_h\co{j_w}(w))-j_w$.   
Although $j_w$ cannot be continuous in $w$, thanks to Equation \eqref{equivariance} on $\widetilde{\p}_h$, 
this provides us with a well-defined holomorphic map on $X$. 

To prove that $L_h$ is one-to-one, assume that for some $w_1$ and $w_2$ in $X$, $L_h(w_1)=L_h(w_2)$. 
Choose $j\in \mathbb{Z}$ with $\Re L_h(w_1)=\Re L_h(w_2)\in (j,j+1]$. 
We must have $j_{w_1}=j_{w_2}=-j$ and the equation $L_h(w_i)=L_h(F_h\co{(-j)}(w_i))+j$, for $i=1,2$. 
As $L_h$ is univalent on $\widetilde{\p}_h\supset L_h^{-1}((0,1]+\Bi \mathbb{R})$ and 
$F_h$ is univalent on $X$, then $w_1=w_2$.

Since the holomorphic function  $L_h(F_h(w))-L_h(w)-1$ is identically zero on $\widetilde{\p}_h$, 
by the uniqueness of the analytic continuation, it must be $0$ on all of $X$.  
\end{proof}

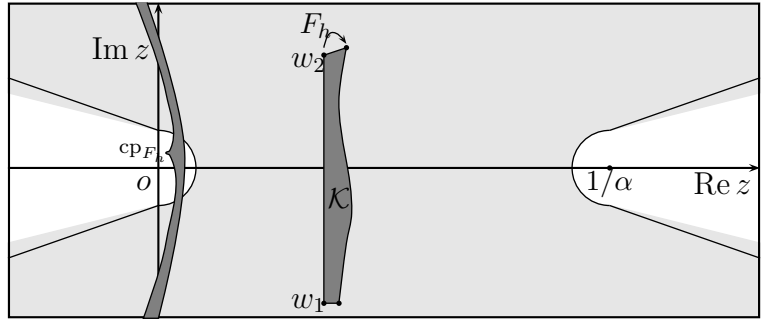
\begin{figure}[ht]
\begin{center}
\begin{pspicture}(0,0)(10,4.2) 
 \newgray{Lgray}{.90}
 
\pscustom[linewidth=.5pt,fillstyle=solid,fillcolor=Lgray,linecolor=Lgray]{
 \psline[liftpen=1](0,0)(0,1)
 \psline[liftpen=1](0,1)(2,1.5)
 \psarc[liftpen=1](2,2){0.5}{270}{90}
\psline[liftpen=1](2,2.5)(0,3)
 \psline[liftpen=1](0,3)(0,4.2)
 \psline[liftpen=1](10,4.2)(10,3)
\psline[liftpen=1](10,3)(8,2.5)
\psarc[liftpen=1](8,2){.5}{90}{270}
\psline[liftpen=1](8,1.5)(10,1)
\psline[liftpen=1](10,1)(10,0) }

\psframe(0,0)(10,4.2) 
\psline{->}(0,2)(10,2)
\psline{->}(2,0)(2,4.2)
\rput(9.5,1.8){$\Re z$}
\rput(1.5,3.6){$\Im z$}
\rput(1.8,1.8){$o$}
\psdot[dotsize=2pt](8,2)
\rput(8,1.8){\small $1/\alf$}

\psline[linewidth=.5pt](2,2.5)(0,3.2)
\psline[linewidth=.5pt](2,1.5)(0,.8)      
\psarc[linewidth=.5pt](2,2){0.5}{270}{90}
\psarc[linewidth=.5pt](8,2){.5}{90}{270}

\psline[linewidth=.5pt](8,2.5)(10,3.2)
\psline[linewidth=.5pt](8,1.5)(10,.8)

\pscustom[linewidth=.5pt,fillstyle=solid,fillcolor=gray]{
  \psline[liftpen=1](4.2,.2)(4.2,3.5)    
  \pscurve[liftpen=1](4.5,3.6)(4.4,2.9)(4.57,1.4)(4.5,1.)(4.4,.2)}
\psline[linewidth=.5pt](4.2,.2)(4.4,.2)

 \psdots[dotsize=2pt](4.2,3.5)(4.5,3.6)(4.2,.2)(4.4,.2)
  \pscurve[linewidth=.5pt]{->}(4.2,3.6)(4.3,3.8)(4.5,3.65)
  \rput(4.1,3.9){\small $F_h$}
  \rput(4,3.4){$w_2$}
  \rput(4,.2){$w_1$}
  \rput(4.4,1.6){\small $\mathcal{K}$}
     
\pscustom[linewidth=.5pt,fillstyle=solid,fillcolor=gray]{
\pscurve[linearc=1,liftpen=1,showpoints=true](1.7,4.2)(2.1,3)(2.15,2.23)(2.1,2.2)(2.15,2.17)(2.1,1)(1.8,0)    
\pscurve[liftpen=1,showpoints=true](2,0)(2.35,2.2)(1.85,4.2)}  

\rput(1.8,2.2){\tiny $\cp_{F_h}$}     
\end{pspicture}
\caption{The light gray region shows the domain $A_2$, and the dark gray regions show the sets $\mathcal{K}$ 
and $L_h^{-1}([0,1]+\Bi \mathbb{R})$.}
\label{Sigma-fig}
\label{Sigma}
\end{center}
\end{figure}

\subsection{Estimates on $L_h$}\label{sec:estimates-L_h}
The univalent map $L_h: X \to \mathbb{C}$ provides us with two foliations on $X$ that are the pre-images of the 
horizontal and vertical line segments in $L_h(X)$. 
In other words, the horizontal leaves are the solutions of the vector field $1/L_h'$, while the vertical ones 
are the integral curves of the vector field $\Bi/ L_h'$.
By Lemma~\ref{L:extending-F_h}, the horizontal leaves are invariant under $F_h$ while the vertical leaves are 
mapped to one another under $F_h$. 
In this section we prove some geometric features of these foliations. 

In the next lemma, the sets $A_1$ and $A_2$ refer to the sets defined in Section~\ref{SS:Fatou-coordinate-of-F_h}. 
The constants $r_2'$ and $C_3$ are introduced in Lemma~\ref{cyl-cond}. 

\begin{lem}\label{L:derivative-of-L_h}
There exists $C\in \mathbb{R}$ such that for all $\alf\in (0, r_2']$ and all $h\in \QIS_\alf$, the following hold.
\begin{itemize}
\item[(1)] For all $w$ with $B(w,5)\subseteq A_2$,  we have $|\arg (L_h'(w))|< \pi/3$ and 
$2/5\leq  |L_h'(w)|\leq 8/3$.
\item[(2)] For all $R \in [3.25, 1/(2\alf)]$ and all $w$ with $B(w, R)\subseteq A_2$, 
\[|\arg (L_h'(w))|\leq \frac{8 C_3}{3R}e^{-2\pi\alf\Im w} + \frac{40}{3R},\]
\[ (1-\frac{8C_3}{3R}e^{-2\pi \alf \Im w}) (1-  \frac{5}{R})
\leq |L_h'(w)|
\leq (1+\frac{8 C_3}{3R}e^{-2\pi \alf \Im w}) (1+ \frac{40}{3R}).\]
\item[(3)] For all $w \in A_1$, $C^{-1}\leq |L_h'(w)|\leq C$. 
\item[(4)] As $\Im w\to +\infty$ in $A_1$, $|L_h'(w)-1|= O (1/\Im w+ \alf e^{-2\pi \alf\Im w})$, 
with the constant in $O$ independent of $\alf$ and $h$. 
\end{itemize} 
\end{lem}

\begin{proof}
{\em Part (1):} 
Given $w_0$ with $B(w_0,5)\subseteq A_2$, let $w_1:=4(F_h(w_0)-w_0)/15$. 
By Lemma~\ref{cyl-cond}-2, $|F_h(w_0)-w_0|<1+1/4$,  $|w_1|< 1/3$, and $B(F_h(w_0), 15/4) \subset B(w_0, 5)$. 
Thus, $L_h$ is defined and univalent on $B(F_h(w_0),15/4)$. 
The function 
\[\psi(w)=\Big( L_h(F_h(w_0))- L_h(F_h(w_0)-\frac{15w}{4}) \Big) \frac{4}{15 L_h'(F_h(w_0))}, \] 
is defined and univalent on $|w| <1$ with $\psi(0)=0$ and $\psi'(0)=1$. 
Applying Theorem~\ref{T:Distortions}-4 to the above map at $w_1$ we obtain 
\begin{multline*}
|\arg (w_1L_h'(w_0))| =  |\arg (w_1L_h'(w_0)\frac{15}{4})|  \leq |\arg (w_1 \psi'(w_1)/\psi(w_1))| \\
 \leq \log \frac{1+|w_1|}{1-|w_1|} \leq \log \frac{1+ 1/3}{1-1/3} < \frac{3}{4} < \frac{\pi}{4}.
\end{multline*}
Above, we have used $\log 2=\int_{1}^{2} 1/x \, dx < 1/8(8/8+8/9+8/10+\cdots + 8/15) < 3/4$. 

By Equation~\eqref{E:slopes-of-X}, $|\arg (w_1)|< \pi/12$. 
Therefore, $|\arg (L_h'(w_0))|< \pi/12+\pi/4= \pi/3$.
  
Applying Theorem~\ref{T:Distortions}-3 to the function $\psi$ we obtain 
\[ \frac{1}{2}\cdot \frac{4}{5} \leq  \frac{1-1/3}{1+1/3} \cdot |\frac{4}{15 w_1}|
\leq  |L_h'(w_0)|\leq 
|\frac{4}{15w_1}|\cdot \frac{1+1/3}{1-1/3} \leq \frac{4}{3} \cdot 2.\]
Above, we use $4/5 \leq |4/(15 w_1)|\leq 4/3$ which is obtained from $|15w_1/4-1|\leq 1/4$. 

\medskip

{\em Part (2):} 
Fix $w_0$ with $B(w_0, R)\subset A_2$. 
Let $R':= R-1.25$ and $w_1:=(F_h(w_0)-w_0)/R'$.  
By Lemma~\ref{cyl-cond}-2, and a simple calculation, 
\begin{equation}\label{E:bound-n-w1}
|w_1| \leq \frac{5}{4R'} \leq \frac{5}{8},  \quad
\frac{1}{R} \leq \frac{1}{R'} \leq \frac{2}{R}.
\end{equation} 
As $B(w_0, R) \subset A_2$, $d(w_0, \mathbb{Z}/\alf) \geq R > R'$. 
Hence $w_0\in \Theta_\alf(R')$. 
Using Lemma~\ref{cyl-cond}-3 with $r= R' \alf \in [0, 1/2]$, we have 
\begin{equation}\label{E:bound-on-R'w1}
|R'w_1-1|\leq \frac{C_3}{R'}e^{-2\pi \alf \Im w_0} \leq \frac{2C_3}{R}e^{-2\pi \alf \Im w_0}.
\end{equation}
Then, using Equation~\eqref{E:slopes-of-X}, and the inequality 
$\mathrm{d} \arcsin x/\mathrm{d}x=1/\sqrt{1-x^2}\leq 4/3$ on $[0, 1/4]$, we get 
\begin{multline*}
|\arg (w_1)| = |\arg (R' w_1)|< \arcsin (\min \big \{\frac{2C_3}{R} e^{-2\pi \alf\Im w_0}, 1/4 \big \} ) \\
\leq \min \{\frac{4}{3} \cdot \frac{2C_3}{R}e^{-2\pi\alf\Im w_0}, \pi/12\}.
\end{multline*}
By Lemma~\ref{cyl-cond}-2, $B(F_h(w_0), R') \subset B(w_0, R) \subset A_2$. 
Hence, $L_h$ is defined and univalent on $B(F_h(w_0), R')$.  
The function 
\begin{equation}\label{E:univalent-function}
\psi(w)= \Big ( L_h(F_h(w_0))-L_h(F_h(w_0)-R'w) \Big ) \frac{1}{R' L_h'(F_h(w_0))}
\end{equation} 
is defined and univalent on $|w|<1$ with $\psi(0)=0$ and $\psi'(0)=1$. 
Using the distortion theorem~\ref{T:Distortions}-4 at $w_1$, we get 
\begin{multline*}
|\arg (w_1 \psi'(w_1)/\psi(w_1))|=
|\arg (w_1 L_h'(w_0))|  \\
\leq  \log (\frac{1+|w_1|}{1-|w_1|})  
= \log (1+\frac{2|w_1|}{1-|w_1|}) \\
\leq  \frac{2|w_1|}{1-|w_1|} 
\leq \frac{16}{3} |w_1|
\leq \frac{16}{3} \cdot \frac{5}{4R'} 
\leq \frac{40}{3R}.
\end{multline*}
Therefore, 
\[|\arg (L_h'(w_0))| \leq |\arg (w_1)| +  |\arg (w_1 L_h'(w_0))|  \leq \frac{8C_3}{3R}e^{-2\pi\alf\Im w_0} + \frac{40}{3R}.\] 

Applying Theorem~\ref{T:Distortions}-3 to the function $\psi$ at $w_1$ we have  
\[ \frac{1 - |w_1|}{1 + |w_1|} \leq  \Big | \frac{w_1 \psi'(w_1)}{\psi(w_1)} \Big | 
= | L_h'(w_0) w_1 R' | \leq \frac{1+ |w_1|}{1- |w_1|} \]  
This implies 
\begin{align*}
|L_h'(w_0)|  
& \leq \big |\frac{1}{R'w_1} \big | \cdot \frac{1+|w_1|}{1-|w_1|}   \\
& = \big |\frac{R' w_1+1 - R'w_1}{R'w_1} \big |  \cdot \frac{1-|w_1|+ 2|w_1|}{1-|w_1|} \\
& \leq  (1+ \frac{|1-R'w_1|}{|R'w_1|}) (1+ \frac{2|w_1|}{1- |w_1|})  \\
& \leq  (1+\frac{4}{3} \frac{2C_3}{R}e^{-2\pi \alf \Im w_0}) (1+ \frac{40}{3R}).
\end{align*}
In the last inequality of the above estimates we have used the inequalities in Equations~\eqref{E:bound-on-R'w1} and 
\eqref{E:bound-n-w1}, as well as $|R' w_1| \geq 3/4$. 

Similarly,  
\begin{align*}
|L_h'(w_0)| 
& \geq \frac{1}{|R'w_1|} \cdot \frac{1-|w_1|}{1+|w_1|} \\
&  \geq (1- \frac{|1-R'w_1|}{|R'w_1|}) (1- \frac{2|w_1|}{1+|w_1|}) \\
& \geq (1- \frac{4}{3} \frac{2C_3}{R}e^{-2\pi \alf \Im w_0}) (1-\frac{5}{R}).
\end{align*}

\medskip

{\em Part (3):} 
The set $A_1$ is simply connected and is bounded by two piece-wise analytic curves, say $\ell_1$ and $\ell_2$ 
with $\ell_1$ on the left hand of $\ell_2$. 
By Lemma~\ref{cyl-cond}-2, as well as Equations~\eqref{E:slopes-of-X} and \eqref{E:alpha_2-first-condition}, 
$F_h (\ell_1)$ and $F_h^{-1}(\ell_2)$ are defined and belong to $A_1$. 
Indeed, there is $\delta>0$ such that the Euclidean distances $d(F_{h}(\ell_1), \partial A_1)$ and $d(F_h^{-1}(\ell_2), \partial A_1)$ 
is at least $\delta$. 

Let $B= \{w\in A_1 : B(w, 5)\subset A_1\}$, and let $\dhyp$ denote the poincar\'e metric on $A_1$ 
(i.e.\ the complete metric of constant curvature $-1$). 
The set $B$ is closed, but not necessarily connected.
By the previous paragraph, there is $\delta'>0$, depending only on $\delta$, such that for every 
$w \in A_1$ there is $j_w\in \{-1, 0, 1\}$ such that $\dhyp (F_h \co{j_w}(w), B)\leq \delta'$. 
Define $B'= \{w \in A_1 : \dhyp (w, B) \leq \delta' \}$.  

Let us fix an arbitrary $w$ in $A_1$.
To find uniform upper and lower bounds on $|L_h'(w)|$ we consider the three cases 
$w\in B$, $w\in B' \setminus B$, and $w \in A_1\setminus B'$.  

If $w\in B$, then by Part (1), $ 2/5 \leq |L_h'(w)| \leq 8/3$.  

Assume that $w \in B' \setminus B$ (here $j_w=0$). 
There is $w' \in B$ with $\dhyp(w,w') \leq \delta'$, and a Riemann mapping $u: B(0,1) \to A_1$ with $u(0)=w'$. 
It follows that there is $\delta''<1$ depending only on $\delta'$, such that $|u^{-1}(w)| \leq \delta''$. 
By the distortion Theorem~\ref{T:Distortions}-1, 
\[\frac{1-\delta''}{(1+\delta'')^3} \leq \frac{|(L_h \circ u)'(u^{-1}(w))|}{|(L_h\circ u)'(0))|} \leq \frac{1+\delta''}{(1-\delta'')^3}, \]
and 
\[\frac{1-\delta''}{(1+\delta'')^3} \leq \frac{|u'(u^{-1}(w))|}{|u'(0))|} \leq \frac{1+\delta''}{(1-\delta'')^3}.\] 
These imply that 
\begin{equation}\label{E:upper-bound-L'}
|L_h'(w)| \leq  \frac{1+\delta''}{(1-\delta'')^3} \cdot \frac{(1+\delta'')^3}{1-\delta''} \cdot |L_h'(w')| \leq  
\frac{(1+\delta'')^4}{(1-\delta'')^4} \cdot  \frac{8}{3},
\end{equation}
and  
\begin{equation}\label{E:lower-bound-L'}
|L_h'(w)| \geq  \frac{1-\delta''}{(1+\delta'')^3} \cdot  \frac{(1-\delta'')^3}{1+\delta''} \cdot |L_h'(w')| \geq  
\frac{(1-\delta'')^4}{(1+\delta'')^4} \cdot \frac{2}{5}.
\end{equation}

Assume that $w \in A_1 \setminus B'$.  
By the first paragraph, $w'= F_h \co {j_w}(w) \in B'$ for some $j_w \in \{+1, -1\}$.
If $j_w=+1$,  then 
\[L_h'(w)= (L_h \circ F_h-1)'(w) = L_h'(w') \cdot F_h'(w),\]
while the absolute value of the right hand of the above equation is bounded from above and away from $0$ 
by Equations~\eqref{E:upper-bound-L'} and \eqref{E:lower-bound-L'}, and Lemma~\ref{cyl-cond}.
If $j_w=-1$, then 
\[L_h'(w)= (L_h \circ F_h^{-1} +1)'(w) = L_h'(w') \cdot (F_h^{-1})'(w),\]
with the size of the right hand of the above equation uniformly bounded from above and away from $0$
by Equations~\eqref{E:upper-bound-L'} and \eqref{E:lower-bound-L'}, and Lemma~\ref{cyl-cond}.   
   
\medskip

{\em Part (4):} 
Since there is a positive angle between the rays bounding the set $A_1$ and $A_2$, there is 
 $\delta>0$, independent of $\alf$ and $h$, such that for $w_0 \in A_1$ with $\Im w_0 \geq 2 C_2$, 
 $B(w_0, \delta \Im w_0) \subset A_2$. 
The function 
\[\psi(w)= \big ( L_h(w_0)- L_h (w_0 - \delta \Im w_0 \cdot w ) \big ) \frac{1}{\delta \Im w_0 L_h'(w_0)}\]
is defined and univalent on $|w|< 1$, $\psi(0)=0$, and $\psi'(0)=1$. 
Let us define $w_1=(w_0-F_h(w_0))/ (\delta \Im w_0)$. 
For $\Im w_1$ large enough, $|w_1| \leq 1/2$. 

Applying Theorem~\ref{T:Distortions}-3 to $\psi$ at $w_1$ we get
\[\frac{1}{|w_0- F_h(w_0)|} \cdot \frac{1-|w_1|}{1+ |w_1|} \leq |L_h'(F_h(w_0))| 
\leq \frac{1+ |w_1|}{1- |w_1|} \cdot \frac{1}{|w_0- F_h(w_0)|}\]
which implies 
\[\big | |L_h'(F_h(w_0))| -1 \big | = O (1/ \Im w_0).\]
Applying Theorem~\ref{T:Distortions}-4 to $\psi$ at $w_1$ we get 
\begin{align*}
|\arg (L_h'(F_h(w_0)))| & \leq |\arg (L_h'(F_h(w_0)) (w_1 \delta \Im w_0)) - \arg (w_1 \delta \Im w_0)| \\
& \leq \log \frac{1+ |w_1| }{1- |w_1|} + | \arg (F_h(w_0) -w_0)| \\
& \leq O (1/\Im w_0) + O (\alf e^{-2\pi \alf\Im w_0}).
\end{align*}
In the last line of the above equations we have used $|F_h(w_0) - w_0 -1| = O (\alf e^{-2\pi \alf\Im w_0})$ which is 
provided by Lemma~\ref{cyl-cond} (with $r=1/2$). 
Combining the above inequalities we conclude 
\begin{align*}
|L_h'(F_h(w_0))-1| & \leq  |\arg (L_h'(F_h(w_0)))| \cdot  |L_h'(F_h(w_0))|+  \big | |L_h'(F_h(w_0))| -1 \big |  \\
& = \big (O(1/\Im w_0) + O(\alf e^{-2\pi \alf\Im w_0}) \big ) + O(1/\Im w_0). 
\end{align*}
Finally,  
\begin{align*}
|L_h'(w_0)-1| &= |L_h'(F_h(w_0)) \cdot F_h'(w_0)-1| \\
& =| L_h'(F_h(w_0)) (F_h'(w_0)-1)| + |L_h'(F_h(w_0)) -1 | \\
& \leq C \cdot O(\alf e^{-2\pi \alf\Im w_0}) + O(1/\Im w_0) + O(\alf e^{-2\pi \alf\Im w_0}).
\end{align*}
The constant in all of $O$ depend only on $\delta$. 
\end{proof}

We also need to control the geometry of the vertical leaves of  the foliation in $X$. 
But, integrating the vector field $\Bi/L_h'$, using the estimates in Lemma~\ref{L:derivative-of-L_h}, 
results in diverging integrals.
We present an alternative approach to deal with this issue in the next proposition. 

\begin{propo}\label{fine-tune} 
For all $M'\in \mathbb{R}$, there is $M>0$ such that for all $\alf$ in $(0, r_2']$, all $h\in \QIS_\alf$, and 
all $r\in (0,1/2]$ the following holds. 
Let $w_1,w_2 \in A_1$ with 
\begin{itemize}
\item[--]$\Re w_1=\Re w_2$, and $\Im w_2 \geq \Im w_1 \geq M'/\alf$,
\item[--]for all $t\in (0,1)$, $t w_1+(1-t)w_2 \in \Theta_\alf(r/\alf) \cap A_1$.
\end{itemize}
Then, 
\begin{itemize}
\item[(1)] $|\Re(L_h(w_1)-L_h(w_2))|\leq M/r$,
\item[(2)] $|\Im \left(L_h(w_1)-L_h(w_2)\right)-\Im(w_1-w_2)|\leq M/r$,
\item[(3)] As $M'$ tends to $+\infty$, $M$ tends to $0$.
\end{itemize}
\end{propo} 

\begin{proof} 
For the simplicity of the notations let $t_i:=\Im w_i$, $i=1,2$.
Define 
\[l:=\{t w_1+(1-t)w_2: t\in [0,1]\}.\] 

Let us assume that $F_h(l) \subset A_1$.  
By Lemma~\ref{cyl-cond}-2, the two curves $l$, $F_h(l)$, as well as the two line segments $w_1+t(F_h(w_1)-w_1)$, $t\in [0,1]$, 
and $w_2+t(F_h(w_2)-w_2)$, $t\in [0,1]$, cut $\cc$ into two connected components. 
Denote the closure of the bounded one by $\mathcal{K}$ (see Figure~\ref{Sigma-fig}). 
We have $\mathcal{K}\subset A_1$, and by Lemma~\ref{L:extending-L_h}, $L_h$ is defined on $\mathcal{K}$. 

Define the rectangle 
\[\mathcal{D}:=\{ s+\Bi t : 0 \leq s \leq 1,\; t_1 \leq t \leq t_2\},\] 
and the map $g: \mathcal{D} \to \mathcal{K}$ as 
\[g(s+\Bi t):=(1-s)(\Re(w_1)+\Bi t)+sF_h(\Re(w_1)+\Bi t).\] 
By the estimates in Lemma~\ref{cyl-cond}-2, $g$ is uniformly close to a translation on $\mathcal{D}$. 
To prove the desired estimates in the proposition, we compare $L_h$ to $g^{-1}$ by analyzing the map 
$G:=L_h\circ g: \mathcal{D} \mapsto \mathbb{C}$.
Using the notations $\zeta=s+\Bi t$, $d\zeta=ds+\Bi dt$ and $d\bar{\zeta}=ds-\Bi dt$, by Green's Theorem, we have 
\begin{align} \label{Green}
\ointclockwise_{\partial \mathcal{D}} G(\zeta)\, d\zeta=\iint_{\mathcal{D}} -\frac{\partial G}{\partial \bar{\zeta}}(\zeta) \,d\zeta%
\wedge d\bar{\zeta}.
\end{align}
With notation $w=g(\zeta)$ and the Cauchy-Riemann equation $\partial L_h/\partial \bar{w}=0$,  
the complex chain rule for $G$ can be written as 
\begin{equation*}
\frac{\partial G}{\partial\bar{\zeta}}=
(\frac{\partial L_h}{\partial w}\circ g)\frac{\partial g}{\partial \bar{\zeta}}+
(\frac{\partial L_h}{\partial \bar{w}}\circ g)\frac{\partial \bar{g}}{\partial \bar{\zeta}}
=(\frac{\partial L_h}{\partial w}\circ g)\frac{\partial g}{\partial \bar{\zeta}}.
\end{equation*} 
A simple differentiation gives
\begin{equation*} 
\begin{aligned}
\frac{\partial g}{\partial \bar{\zeta}}(s+\Bi t)&=\frac{1}{2}[\frac{\partial g}{\partial s}+\Bi \frac{\partial g}{\partial t}](s+\Bi t)\\
&=\frac{1}{2}[F_h(\Re w_1+\Bi t)-(\Re w_1+\Bi t)-1+s(1-F'_h(\Re w_1+\Bi t))].
\end{aligned}
\end{equation*}
Since $l\subset \Theta_\alf(r/\alf)$, by Lemma~\ref{cyl-cond}-3, 
\[|\frac{\partial g}{\partial \overline{\zeta}} (s+\Bi t)| \leq \frac{1}{2} |F_h(g(\Bi t))-g(\Bi t)-1| + \frac{1}{2} |F_h'(g(\Bi t))-1|
 \leq C_3 \frac{\alf}{r} e^{-2\pi \alf t}.\] 
Also recall that by Lemma~\ref{L:derivative-of-L_h}, $|L_h'|\leq C$ on $\mathcal{K}$. 
Then, the right hand of Equation~\eqref{Green} may be controlled as follows:  
\begin{equation}\label{E:righ-hand-green}
\begin{aligned}
\Big|\iint_\mathcal{D} \frac{\partial G}{\partial \bar{\zeta}}(\zeta) \, d\zeta\wedge d\bar{\zeta}\Big |
&\leq 2 \int_{t_1}^{t_2} \int_0^1 \big|\frac{\partial G}{\partial \bar{\zeta}}(s+\Bi t)\big|\, ds dt   \\
&\leq 2\sup_{w\in \mathcal{K}} |L_h'(w)| \int_{t_1}^{t_2} \int_0^1\!\! C_3 \frac{\alf}{r}
e^{-2\pi\alf t}\,ds dt \\
&\leq \frac{2 C C_3 \alf}{r}  \int_{M'/\alf}^{\infty} e^{-2\pi\alf t}\, dt  \\
&\leq  \frac{C C_3}{\pi r} e^{-2\pi  M'}. 
\end{aligned}
\end{equation}

The left hand of Equation \eqref{Green} may be written as 
\begin{multline*}
\int_{t_1}^{t_2}G(\Bi t) \Bi\,d t + \int_0^1 G(s +\Bi t_2)\,d s \\ 
+\int_{t_1}^{t_2} G(1+\Bi (t_1+t_2- t))(-\Bi)\,d t +\int_0^1 G(1- s +\Bi t_1)(-1)\,d s. 
\end{multline*}
By definition, for $t\in [t_1, t_2]$, we have $g(1+ \Bi t)= F_h(g(\Bi t))$. 
This implies that for all $t\in [t_1, t_2]$, $G(1+ \Bi t)= G(\Bi t)+1$. 
Using this relation in the third integral of the above equation and then making a change of coordinate, 
the above sum is equal to 
\begin{eqnarray*}
-\Bi(t_2-t_1)+\int_0^1 G(s+\Bi t_2)\,ds + \int_0^1 -G(1-s+\Bi t_1)\,ds.
\end{eqnarray*}

Set 
\[M_i:= \sup_{s\in [0,1]} \big |L_h'(g(s+ \Bi t_i)) - 1\big |, i=1,2.\]
By Lemma~\ref{L:derivative-of-L_h}-3, $M_1$ and $M_2$ are uniformly bounded from above independent of $\alf$ and $h$. 
Moreover, by Lemma~\ref{cyl-cond}, $|\partial g/\partial s| \leq 5/4$ on $\mathcal{K}$ and 
$|\partial g /\partial s -1|\leq C_3 \alf e^{-2\pi M'} /r$ on $\mathcal{K}$. 
Hence,  
\begin{equation}\label{E:the-1/4-error-1}
\begin{aligned}
\Big| \int_{0}^{1} G(s  + \Bi t_2)\,ds - L_h(w_2)-1/2 \Big | 
& = \Big | \int_0^1 G(s +\Bi t_2) - G(\Bi t_2)-s \,ds \Big | \\
&\leq \int_0^1 |G(s +\Bi t_2)-G(\Bi t_2)-s| \,ds  \\
&\leq \sup_{s \in [0,1]} |G(s +\Bi t_2)-G(\Bi t_2)-s|\\ 
&\leq \sup_{s \in [0,1]} \big |\frac{\partial G}{\partial s}(s+\Bi t_2)-1\big |\\
&\leq \sup_{s \in [0,1]}  \big | \frac{\partial g}{\partial s} (s+\Bi t_2)  \cdot L_h'(g(s+\Bi t_2)) -1 \big | \\
& \leq \frac{5}{4} \cdot M_2+\frac{C_3\alf}{r} e^{-2\pi M'}.
\end{aligned}
\end{equation}
To get the fifth inequality in the above equation we have used the formula $AB-1= A(B-1)+A-1$. 
On the other hand, 
\begin{align*}
\int_{0}^{1}-G(1-s+ \Bi t_1)\,ds+ L_h(w_1) +1/2 
 &=\int_{0}^{1}-G(1-s+ \Bi t_1) + L_h(w_1) +(1-s) \,ds \\
 & = -\int_{0}^{1} G(s + \Bi t_1) - L_h(w_1) - s \,ds.
\end{align*}
Thus, the inequalities in Equation~\eqref{E:the-1/4-error-1} may be repeated for the above integral to conclude 
\begin{equation}\label{E:the-1/4-error-2}
\Big|  \int_{0}^{1}-G(1-s+ \Bi t_1)\,ds+ L_h(w_1) +1/2  \Big |  \leq \frac{5}{4} \cdot M_1+\frac{C_3\alf}{r} e^{-2\pi M'}.
\end{equation}

One infers the inequalities in Parts (1) and (2) of the proposition by considering the real part and the imaginary part of  
Equation~\eqref{Green} as well as using the bounds in Equations~ \eqref{E:the-1/4-error-1} and  \eqref{E:the-1/4-error-2}.  
For instance, for Part (2), 
\begin{align*}
|\Im (w_2-w_1) &- \Im (L_h(w_2)-L_h(w_1)) | \\
& \leq \Big| (t_2 -t_1) - \Im \big (\int_0^1 G(s+\Bi t_2)\,ds \big ) - \Im \big ( \int_0^1 -G(1-s+\Bi t_1)\,ds\big) \Big  | \\
& \qquad + \Big | \Im \big( \int_0^1 G(s+\Bi t_2)\,ds \big) - \Im (L_h(w_2)) -1/2 \Big | \\
& \qquad \qquad + \Big | \Im \big(\int_0^1 - G(1-s+\Bi t_1)\,ds \big) + \Im (L_h(w_1)) +1/2\Big | \\
& \leq \big (\frac{C C_3e^{-2\pi M'}}{\pi} +  \frac{5M_2 r}{4} + \frac{5M_1r}{4} + 2 C_3 \alf e^{-2\pi M'}\big) \frac{1}{r}.
\end{align*}

This finishes the proof of the inequalities in Parts (1) and (2) under the assumption $F_h(l)\subset A_1$ that 
we made at the beginning of the proof.

If $F_h^{-1}(l)\subset A_1$, one considers the bounded region $\mathcal{K}'$ enclosed by the curves $l$, $F_h^{-1}(l)$, 
$t\mapsto w_1+ t(F_h^{-1}(w_1)-w_1)$, $t\in [0,1]$ and $t\mapsto w_2+ t(F_h^{-1}(w_2)-w_2)$, $t\in [0,1]$. 
One may repeat the above calculations and estimates for the map $g: \mathcal{D}\to \mathcal{K}'$ defined as 
$g(s+\Bi t):=s(\Re(w_1)+\Bi t)+(1-s)F_h^{-1}(\Re(w_1)+\Bi t)$.  
This leads to the same estimates in Parts (1) and (2), under this condition. 

There is also the possibility that for some $w_1$ and $w_2$ neither $F_h(l) \subset A_1$ nor $F_h^{-1}(l) \subset A_1$ holds. 
(This is when the two balls $B(0,C_2)$ and $B(1/\alf, C_2)$ are close.)  
By the assumption on $\alf$ in Equation~\eqref{E:alpha_2-first-condition}, the balls $B(0,C_2)$ and $B(1/\alf, C_2)$ 
are at least $5/4$ apart. 
Thus, there are $w_1' \in B(w_1,1)$ and $w_2' \in B(w_2,1)$ with $\Re w_1'= \Re w_2'$, the line 
segment $l'$ connecting them remains in $\Theta_\alf(r/(2\alf)) \cap A_1$, and either 
$F_h(l') \subset A_1$ or $F_h^{-1}(l')\subset A_1$. 
By the above argument we obtain the inequalities in the proposition for $w_1'$ and $w_2'$.  
From that, one infers the inequalities in Parts (1) and (2) for $w_i$ using the uniform bound 
on $L_h'$ in Lemma~\ref{L:derivative-of-L_h}-3, by making $M$ large enough.   
  
\medskip

{\em Part (3):} 
For $M' \geq 1/2$, we may fix $r=1/2$ in the above arguments. 
Also, by Lemma~\ref{L:derivative-of-L_h}-4, the constants $M_1$ and $M_2$ are $O (\alf/M' + e^{-2\pi M'})$. 
In particular, $M_1$ and $M_2$ tend to $0$ as $M'$ tends to $+\infty$. 
This implies that
\[M=(\frac{C C_3e^{-2\pi M'}}{\pi} +  \frac{5M_2 r}{4} + \frac{5M_1r}{4} + 2 C_3 \alf e^{-2\pi M'}\big)\]
tends to $0$ as $M'$ tends to $+\infty$. 
\end{proof}

\begin{lem}\label{E:asymptote-of-L_h} 
The limit 
\[\ell_h:=\lim _{\substack{\Im w\to +\infty \\ w\in A_1}} L_h(w)-w \]
exists and is finite. 
\end{lem}

\begin{proof}
For $w_1$ and $w_2$ on the vertical line $\Re w= 1/(2\alf)$ we may use Proposition~\ref{fine-tune} with 
$r=1/2$ and large values of $M'$to conclude that $|(L_h(w_2)-L_h(w_1)) - (w_2-w_1)|$ tends to $0$ as $\Im w_1$ and 
$\Im w_2$ tend to $+\infty$.
This implies that the function $L_h(w)-w$ satisfies the Cauchy criterion and hence its limit 
exists as $\Im w\to +\infty$ along $\Re w= 1/(2\alf)$. 

By Lemma~\ref{L:derivative-of-L_h}-4, $|L_h'(w)-1|$ tends to $0$ as $\Im w \to +\infty$ in $A_1$. 
This implies that as $\Im w \to +\infty$ in the strip $\Re w \in [1/(2\alf) -1, 1/(2\alf)+1]$, the limit of $L_h(w)-w$ exists. 

Using Lemma~\ref{cyl-cond}-2, for $w\in A_1$ there is an integer $j$ with $|j| \leq O(\Im w)$ 
such that $\Re F_h\co{j}(w) \in [1/(2\alf) -1, 1/(2\alf)+1]$. 
Moreover, by Lemma~\ref{cyl-cond}-3 with $r=1/2$, 
\[|F_h\co{j}(w) -j -w| \leq j \cdot 2 C_3 \alf e^{-2\pi \alf \Im w} \leq O((\Im w) \cdot  e^{-2\pi \alf \Im w}).\]  
Therefore,  
\[L_h(w) -w = L_h (F_h\co{j}(w)) -j -w = \big (L_h (F_h\co{j}(w)) - F_h\co{j}(w)\big) + \big( F_h\co{j}(w) -j -w \big).\]
Thus, $L_h(w) -w$ has a limit as $w\in A_1$ tends to $+\Bi \infty$. 
\end{proof}
We shall give an upper bound on the size of $\ell_h$ in Corollary~\ref{C:size-of-translation}. 

\subsection{The image of $L_h$}\label{sec:width-of-L_h(X)}
Recall that $L_h(X)$ is an open subset of $\cc$ containing $(0,1]+\Bi \mathbb{R}$. 
In this section we prove a lower bound on 
\begin{equation}\label{E:x_h}
x_h:=\sup \{ t\in (0, +\infty) \mid (0,t)+\Bi \mathbb{R} \subseteq L_h(X)\}.
\end{equation}
Recall the constants $r_2'$ and $C_2$ introduced in Lemma~\ref{cyl-cond}.    
Define the constant 
\begin{equation}\label{E:r_2}
r_2:=\min \{r_2', 1/(4C_2+20)\}.
\end{equation}
For $\alpha\leq r_2$, $2C_2+10$ and $\alpha^{-1}-2C_2 -10$ belong to $X$, and hence $L_h$ is defined 
at these points. 
Let 
\begin{equation}\label{E:s_1-s_2}
s_1= \Re (L_h(2C_2+10)), \quad s_2= \Re (L_h(\alpha^{-1}-2C_2 -10)).
\end{equation}

\begin{lem}\label{L:domain-of-E-precoordinate} 
For all $\alf \in (0, r_2]$ and $h\in \QIS_\alf$, we have
\begin{itemize}
\item[(1)] $1< s_1 \leq s_2 \leq x_h$; 
\item[(2)] $\forall w\in L_h^{-1}((s_1, s_2)+\textnormal{\Bi} \mathbb{R})$, $B(w,5) \subset A_1$;
\item[(3)] $\forall t \in \mathbb{R}, |\arg (L_h^{-1}(1+\textnormal{\Bi} t) - 2C_2-10)| \geq \pi/6$;
\item[(4)] $\forall t \in \mathbb{R}, |\arg(L_h^{-1}(x_h+ \textnormal{\Bi} t) - \alf^{-1}+ 2C_2+10)|\leq 5\pi/6$;
\item[(5)] $\limsup_{|t|\to +\infty} |\arg (L_h^{-1}(\textnormal{\Bi}t)-2C_2-10)| \leq 5\pi/6$;
\item[(6)] $\limsup_{|t|\to +\infty} |\arg (L_h^{-1}(x_h+\textnormal{\Bi}t)-\alf^{-1}+2C_2+10)| \geq \pi/6$.
\end{itemize}
\end{lem}

\begin{proof}
Let $X'$ be the set of points $w\in \cc$ such that 
\[|\arg (w-2C_2-10)| \leq 5\pi/6, \text{ and } |\arg(w-1/\alf+2C_2+10)|\geq \pi/6.\]
The condition $\alpha\leq r_2$ implies that $X'$ is a connected set containing the points $2C_2+10$ and $\alpha^{-1}-2C_2 -10$. 
Moreover, for $w\in X'$, $|w-i/\alpha| \geq C_2+5$, for all $i\in \mathbb{Z}$.  
In particular, for $w \in X'$, $B(w, 5)\subset A_1 \subset A_2$. 
Then, by Lemma~\ref{L:derivative-of-L_h}-1, for $w\in X'$, $\arg (L_h'(w))$ belongs to $(-\pi/3, \pi/3)$. 

By the above paragraph, for every $w \in X'$, the argument of the tangent line to the vertical foliation passing through $w$ 
at $w$ belongs to $\pi/2+(-\pi/3, \pi/3)=(\pi/6, 5\pi/6)$, modulo $\pi$.    
In particular, $L_h^{-1}(s_1+\Bi \mathbb{R})$ and $L_h^{-1}(s_2+\Bi \mathbb{R})$ stay within $X' \subset X$, and spread from 
$-\Bi \infty$ to $+ \Bi \infty$.  
This implies parts 1 and 2 of the lemma. 

By the above paragraph, $L_h^{-1}(s_1+\Bi \mathbb{R})$ lies to the left of the curve $|\arg (w-2C_2-10)|= \pi/6$.
Also, since vertical leaves are disjoint, the vertical leaf passing through $\cv_{F_h}$, $L_h^{-1}(1+\Bi \mathbb{R})$, 
must lie to the left of  $L_h^{-1}(s_1+\Bi \mathbb{R})$. 
This implies part 3 of the lemma. 
Similarly, $L_h^{-1}(s_2+\Bi \mathbb{R})$ lies to the right of the curve $|\arg (w-\alpha^{-1}+2C_2+10)|= 5\pi/6$, 
and $L_h^{-1}(x_h+\Bi \mathbb{R})$ lies to the right of the leaf $L_h^{-1}(s_2+\Bi \mathbb{R})$. 
This implies part 4 of the lemma. 

By the second paragraph, $L_h^{-1}(s_1+ \Bi \mathbb{R})$ lies to the right of the curve $|\arg(w-2C_2-10)|= 5\pi/6$. 
By the uniform bound on $|F_h(w)-w-1|$ in Lemma~\ref{cyl-cond}, $L_h^{-1}(\Bi \mathbb{R})$ lies within 
uniformly bounded distance from $L_h^{-1}(s_1+ \Bi \mathbb{R})$.  
This implies part 5 of the lemma. 
Similarly, part 6 of the lemma follows from comparing $L_h^{-1}(x_h+ \Bi \mathbb{R})$ to the curve 
$L_h^{-1}(s_2+ \Bi \mathbb{R})$. 
\end{proof}

By Lemma~\ref{L:domain-of-E-precoordinate}-(5)-(6), the top end of $L_h^{-1}((0, x_h)+\Bi \mathbb{R})$ is contained in $A_1$.
Therefore, by Lemma~\ref{E:asymptote-of-L_h}, we have 
\begin{equation}\label{E:asymptote-of-L_h^{-1}}
\lim_{\substack{\Im w\to +\infty \\ \Re w\in (0, x_h)}} L_h^{-1}(w)-w = -\ell_h.
\end{equation} 

\begin{lem}\label{L:derivative-of-L_h^{-1}}
For all $\eps>0$ there is $M_\eps$ such that for all $\alf\in (0, r_2]$, all $h\in \QIS_\alf$, and all 
$\zeta\in ([0,x_h]+\textnormal{\Bi}\mathbb{R}) \setminus B(0, \eps)$, we have 
\[ M_\eps^{-1}  \leq |(L_h^{-1})'(\zeta)| \leq M_\eps.\]
\end{lem} 

\begin{proof}
By Lemma~\ref{L:domain-of-E-precoordinate}-3, $A_3\cap B(1/\alf, C_2)=\emptyset$ (in the definition of $A_3$, 
$L\cap R'=\emptyset$). 
Hence, by Lemma~\ref{cyl-cond}-1, on $X \setminus B(0, C_2)$, we have $|F_h'-1|< 1/4$.

By the pre-compactness of the class $\cup_{\alf\in [0,r_2]}\QIS_\alf$ and the continuous dependence of $L_h$ on 
$h$, there is $\eps'>0$ such that $L_h(X\cap B(\cp_{F_h},\eps')) \subset B(0, \eps)$.  
For the same reason, $|F_h'|$ is uniformly bounded from above and away from zero on $(B(0, C_2)\cap X)\setminus B(0, \eps')$. 

By Lemmas~\ref{L:domain-of-E-precoordinate}-2 and \ref{L:derivative-of-L_h}-3, when $s_1 \leq \Re \zeta \leq s_2$, 
$C^{-1}\leq |(L_h^{-1})'(\zeta)| \leq C$, where $C$ is a uniform constant.  

On the other hand, by Lemmas~\ref{cyl-cond}-2, \ref{L:extending-F_h}-4, and \ref{L:extending-L_h}, 
$s_1$ and $x_h-s_2$ are uniformly bounded from above. 
Thus, for $\zeta$ with $0 \leq \Re \zeta \leq x_h$, there is $j_\zeta\in \mathbb{Z}$, with $|j_\zeta|$ 
uniformly bounded from above, such that $s_1\leq \Re(\zeta- j_\zeta) \leq s_2$.
Then, the desired bounds on $|(L_h^{-1})'(\zeta)|$ follow from the functional equation 
$L_h\circ F_h= L_h+1$ in Lemma~\ref{L:extending-L_h} and the above upper and lower bounds on $|F_h'|$.
\end{proof}

Define the sets 
\begin{gather*}
B_0:=\{\zeta \in \cc \mid \Re \zeta \in [0, 1]\}, \; B_1:=\{\zeta \in \cc \mid \Re \zeta \in [x_h-1, x_h]\}.
\end{gather*}

\begin{lem}\label{L:change-of-coordinates-precoord}
For all $\alf \in (0, r_2]$, all $h \in \QIS_\alf$, and all $w \in L_h^{-1}(B_0)+1/\alf$ with $|\Im w| \geq 3C_2+5$, 
there is $l_w \in \mathbb{Z}$ such that $F_h\co{l_w}(w)$ is defined and belongs to $L_h^{-1}(B_1)$. 
Moreover, 
\begin{itemize}
\item[(1)] if $l_w\geq 0$, then for $0\leq j\leq l_w$ we have $F_h\co{j}(w)\in X$; 
\item[(2)] if $l_w < 0$, for $l_w\leq j\leq 0$, we have $F_h\co{j}(w)\in X$;  
\item[(3)] for every constant $c\geq 5$ and for every $w$ with $3C_2+5 \leq |\Im w|  \leq 3C_2+c$, 
$|l_w|$ is bounded from above by a constant depending only on $c$ (independent of $w$, $\alf$, and $h$);
\item[(4)] as $\Im w \to \pm \infty$, we have $\Im F_h \co{l_w}(w) \to \pm\infty$. 
\end{itemize} 
\end{lem}

\begin{proof}
By Lemma~\ref{L:domain-of-E-precoordinate}-3, $L_h^{-1}(B_0)+1/\alf$ lies to the left of the curve 
$|\arg (w-2C_2-10-1/\alf)|=\pi/6$, and $L_h^{-1}(B_1)$ lies to the right hand of $|\arg (w+2C_2+10-1/\alf)|=5\pi/6$. 
The curve $|\arg (w-2C_2-10-1/\alf)|=\pi/6$ intersect the left-hand boundary of $A_1$ at two points with imaginary parts 
$\pm 2C_2/(\sqrt{6}-\sqrt{2})+ C_2+5$. 
Note that $2C_2/(\sqrt{6}-\sqrt{2}) +C_2+5 \leq 3C_2+5$. 
Thus, the intersection of $L_h^{-1}(B_0)+1/\alf$ and $\{ w: |\Im w|\geq 3C_2+5\}$ is contained in $A_1$. 

By Lemma~\ref{L:extending-F_h}, the forward orbit and the backward orbit of  
every point in $A_1$ eventually leave $X$.
Combining with the above paragraph, the backward or the forward orbit of any $w\in L_h^{-1}(B_0)+1/\alf$ with 
$|\Im w| \geq 3C_2 +5$ must cross $L_h^{-1}(B_1)$. 
Moreover, the uniform estimate in Lemma~\ref{cyl-cond}-2 shows that for  
$w\in L_h^{-1}(B_0)+1/\alf$ with $3C_2+5 \leq |\Im w| \leq 3C_2+6$, $|l_w|$ is uniformly bounded from above by a 
constant depending only  on $C_2$. 

Part (4) of the lemma follows from the upper bound on $|\arg (F_h(w) -w)| \leq \pi/12$ in Equation~\eqref{E:slopes-of-X}
and the above argument. 
\end{proof}

For $\zeta$ near the top and bottom ends of $B_0$ let $w=L_h^{-1}(\zeta)+1/\alf$ and $m_\zeta= l_w$ be the integer defined 
in Lemma~\ref{L:change-of-coordinates-precoord}.  
For some of those $\zeta$, there may be more than one choice for $m_\zeta$, in which case, one may choose either one. 
Then, consider the map 
\begin{equation}\label{E:T_h}
T_h(\zeta):= L_h (F_h\co{m_\zeta} (L_h^{-1}(\zeta)+1/\alf)),
\end{equation}
near the two ends of $B_0$, with values in $B_1$. 

\begin{lem}\label{L:horn-map-inverted}
There is $\eta>1$ such that for all $\alf \in (0, r_2]$ and all $h\in \QIS_\alf$, the map 
\[T_h : \{\zeta\in B_0: |\Im \zeta| \geq \eta\}/ \mathbb{Z} \to B_1/\mathbb{Z}\] 
is defined and univalent. 
Moreover, 
\begin{itemize}
\item[(1)] $\Im T_h(\zeta)\to \pm\infty$, as $\Im \zeta\to \pm \infty$;
\item[(2)] for $\eta\leq |\Im \zeta| \leq \eta+1$, $|m_\zeta|$ is uniformly bounded from above independent of $\alf$, $\zeta$, and $h$.
\end{itemize}
\end{lem}

\begin{proof}
By the pre-compactness of the class $\QIS$ (see the proof of Lemma~\ref{L:derivative-of-L_h^{-1}}) there is 
$\eta>1$, independent of $\alf$ and $h$, such that for all $\zeta\in B_0$ with $|\Im \zeta| \geq \eta$, 
$|\Im L_h^{-1}(\zeta)| \geq 3C_2+5$. 
Combining this with Lemma~\ref{L:change-of-coordinates-precoord}, we conclude that $T_h$ is defined above the 
height $\eta$ and below the height $-\eta$. 
Moreover, by Equation~\eqref{lift-relation} and Lemma~\ref{L:extending-L_h}, $T_h(\zeta+1)=T_h(\zeta)+1$ 
when $\Re \zeta=0$. 
Therefore, $T_h$ projects to a well-defined map from $\{\zeta\in B_0: |\Im \zeta| \geq \eta\}/ \mathbb{Z}$ to $B_1/\mathbb{Z}$.  
Also, as $F_h$ and $L_h$ are univalent on $X$, $T_h$ must be univalent.

We have $\Im L_h(\zeta)\to \pm\infty$ when $\Im \zeta\to \pm \infty$ within $B_0$, 
and $\Im F_h\co{l_w}(w) \to \pm\infty$ as $\Im w\to \pm\infty$ by Lemma~\ref{L:change-of-coordinates-precoord}-4. 
This implies the asymptotic behavior of $T_h$ in part 1.
Part (2) of the lemma follows from Lemma~\ref{L:change-of-coordinates-precoord}-3 and 
the uniform bound on $|L_h'|$ in Lemma~\ref{L:derivative-of-L_h^{-1}}. 
\end{proof}

\begin{rem}
The map $T_h$ projects under $z\mapsto e^{2\pi \Bi z}$ to the inverse of $\rr'(h)$ restricted to the ball $B(0, e^{-2 \pi \eta})$, 
see Section~\ref{Inou-Shishikura-class}. 
The reason for considering this inverse here is that \textit{a priori} we do not know how large  
$\Dom T_h^{-1}$ is, that is, the values of $L_h^{-1}$ on $B_1$.   
\end{rem}

Recall the covering map $\tau_h$ and the relation $L_h= \Phi_h \circ \tau_h$ in Equations 
\eqref{E:covering-formula} and \eqref{linearize}. 
Define
\begin{equation}\label{E:y_h}
y_h:=\sup \{ t\in (0, x_h) \mid \tau_h  \text{ is univalent on } L_h^{-1}((0,t)+\Bi \mathbb{R})\}.
\end{equation}
We have $y_h\geq 1$. 
That is because, $\Phi_h$ and $L_h$ are univalent maps, and $\Phi_h(\p_h)\supset (0,1]+\Bi \mathbb{R}$. 

\begin{propo}\label{P:width-of-L_h(X)}
There is $k>0$ such that for all $\alf \in (0, r_2]$ and all $h \in \QIS_\alf$, we have 
\begin{itemize}
\item[(1)] $\alf^{-1} -k\leq y_h\leq \alf^{-1}$; 
\item[(2)] $x_h\leq \alf^{-1}+k $. 
\end{itemize}
\end{propo}

\begin{proof}
First we claim that 
\[b_h'=\sup\{ |m_\zeta|: \zeta\in B_0, |\Im \zeta| \geq \eta\}\] 
is uniformly bounded from above independent of $\alf$ and $h$. 
That is because, by Lemma~\ref{L:horn-map-inverted}, $T_h$ projects under $e^{2\pi\Bi z }$ and $e^{-2\pi\Bi z }$ to 
univalent maps on $B(0, e^{-2\pi \eta})$, denoted by $\hat{T}_{h, t}$ and $\hat{T}_{h,b}$ respectively, satisfying 
$\hat{T}_{h,t}(0)=\hat{T}_{h,b}(0)=0$. 
By the distortion Theorem~\ref{T:Distortions}, the image of any ray $\{re^{\Bi \theta} :  r\in (0, e^{-2\pi (\eta+1)})\}$, for fixed 
$\theta\in [0, 2\pi)$, under $\hat{T}_{h, t}$ and $\hat{T}_{h,b}$ have uniformly bounded spirals about $0$. 
(Indeed, using a sharp distortion theorem on the argument obtained from Loewner theory (see Thm 3.5 in \cite{Dur83}) 
the total spiral is bounded by $\log ((1+ e^{-2\pi(\eta+1)})/ (1-e^{-2\pi(\eta+1)})) \leq 2 e^{-2\pi(\eta+1)} \ll 2 \pi$. 
But we don't need this exact value here.)
In terms of the lift map $T_h$ and the integers $m_\zeta$, this means that 
$|m_\zeta-m_{\zeta'}|$ is uniformly bounded from above for $\zeta$ and $\zeta'$ in $B_0$ with 
$|\Im \zeta| \geq |\Im \zeta'| \geq \eta+1$.
On the other hand, when $|\Im \zeta|\in [\eta, \eta+1]$, $|m_\zeta|$ is uniformly bounded from above, 
independent of $\alf$ and $h$, by Lemma~\ref{L:horn-map-inverted}-2. 
This finishes the proof of the claim. 

Recall that $\tau_h$ is periodic of period $1/\alpha$. 
If $L_h^{-1}(B_0)+ 1/\alpha$ lies to the right of $L_h^{-1}(B_1)$, then $\tau_h$ is univalent on 
$L_h^{-1}((0,x_h) + \Bi \mathbb{R})$ and we have $y_h=x_h$. 
Otherwise, as we show below, $y_h$ is obtained from subtracting a uniformly bounded value from $x_h$.  

By Lemmas~\ref{L:domain-of-E-precoordinate} and \ref{L:derivative-of-L_h^{-1}}, 
$\inf \Re L_h^{-1}(\Bi[-\eta, \eta])$ is uniformly bounded from below, and 
$\sup \Re L_h^{-1}(x_h+\Bi[-\eta,\eta])- 1/\alpha$ is uniformly bounded from above, both independent of $\alpha$ and $h$. 
Then, Lemma~\ref{cyl-cond}-2 implies that there is a positive integer $j$, uniformly bounded from above, 
such that 
\[\sup \Re \big(F_h^{-j}(L_h^{-1}(x_h+\Bi [-\eta, \eta]))\big) < \sup \Re \big( (L_h^{-1}(\Bi [-\eta, \eta])+1/\alf)\big ).\]  
By the first paragraph, 
\begin{align*}
F_h^{-b_h'-2}(L_h^{-1}(x_h+\Bi [\eta, +\infty]) \cap (L_h^{-1}(\Bi \mathbb{R})+1/\alpha)= \emptyset \\
F_h^{-b_h'-2}(L_h^{-1}(x_h+\Bi (-\infty, -\eta]) \cap (L_h^{-1}(\Bi \mathbb{R})+1/\alpha)=\emptyset.
\end{align*} 
Therefore, there is $b_h>0$, uniformly bounded from above independent of $\alf$ and $h$, 
such that $L_h^{-1}(\Bi \mathbb{R})+1/\alf$ lies to the right of $L_h^{-1}(x_h-b_h+\Bi \mathbb{R})$. 
This implies that, $\tau_h$ is univalent on $L_h^{-1}((0,x_h-b_h)+\Bi \mathbb{R})$. 
That is, $y_h\geq x_h-b_h$.
Note that $x_h-b_h\geq 1$, since $L_h^{-1}(\Bi \mathbb{R})+1/\alf$ lies to the right of $L_h^{-1}(1+\Bi \mathbb{R})$. 
 
By Equation~\eqref{E:asymptote-of-L_h^{-1}}, $L_h^{-1}$ tends to a translation near the top end of 
$(0, x_h)+\Bi\mathbb{R}$, and $\tau_h$ is periodic of period $1/\alf$. 
Hence, $y_h\leq 1/\alf$. 
For the same reason, near the top end of $B_0$, $|m_\zeta - (x_h-1/\alf)| \leq 2$.
Hence, $x_h \leq 1/\alpha+2+b_h'$ and $y_h \geq x_h-b_h \geq 1/\alpha + m_\zeta - 2 -b_h \geq 1/\alpha -2 - b_h$, 
where $b_h$ is uniformly bounded from above.
\end{proof}


The argument in the proof of Proposition~\ref{P:width-of-L_h(X)} through studying $T_h$ has a key consequence 
stated in the next proposition. 

\begin{propo}\label{P:value-at-end-by-L}
There is $C'>0$ such that for all $\alpha\in (0, r_2]$ and all $h\in \QIS_\alpha$, we have 
\[| L_h^{-1}(x_h)-1/\alpha|\leq C'.\] 
\end{propo}

\begin{proof}
Let $\eta\geq 1$ be the constant introduced in Lemma~\ref{L:horn-map-inverted}. 
For $\zeta \in B_0$ with $\Im \zeta\geq \eta$, define
$E_h(\zeta):=L_h \circ F_h^{-n_\zeta} \circ  F_h\co{m_\zeta}(L_h^{-1}(\zeta)+1/\alf)$,  
where $m_\zeta$ is the integer in Equation~\eqref{E:T_h} and $n_\zeta$ is the number of backward iterates 
required to map $F_h^ {m_\zeta} (L_h^{-1}(\zeta)+1/\alf)$ into $L_h^{-1}(B_0)$. 
The integer $n_\zeta$ exists because of Lemma~\ref{L:extending-F_h}. 
By Lemma~\ref{L:horn-map-inverted}, $E_h$ is defined above the height $\eta$, and its values belong to $B_0$.

By Equations~\eqref{E:choice-of-lift} and \eqref{E:asymptote-of-L_h^{-1}},  
$\Im E_h(\zeta)- \Im \zeta\to 0$ when $\Im \zeta\to +\infty$. 
Also, by Equation~\eqref{lift-relation} and Lemma~\ref{L:extending-L_h}, $E_h(\zeta+1)=E_h(\zeta)+1$, when $\Re \zeta=0$. 
Hence, $E_h$ projects under $e^{2\pi \Bi \zeta}$ to a well defined univalent map 
$\widetilde{E}_h: B(0, e^{-2\pi \eta}) \to \mathbb{C}$ satisfying $\widetilde{E}_h(0)=0$ and $|\widetilde{E}_h'(0)|=1$. 
By the distortion Theorem~\ref{T:Distortions}, $\widetilde{E}_h (B(0,e^{-2\pi \eta}))$ contains the ball $B(0, e^{-2\pi \eta}/4)$. 
Then, applying the distortion theorem to the map $\widetilde{E}_h^{-1}: B(0, e^{-2\pi \eta}/4) \to \mathbb{C}$ we conclude that 
$|\widetilde{E}_h^{-1}(e^{-4\pi \eta})|$ is uniformly bounded from above and away from $0$. 
This implies that $|\Im E_h^{-1}(2\eta\Bi)|$ must be uniformly bounded from above and below. 

On the other hand, by the pre-compactness of $\QIS$, on any given compact subset of $B_0$, 
$|\Im L_h^{-1}(\zeta)-\Im \zeta|$ is uniformly bounded from above.  
Combining with the above paragraph, we conclude that  $|\Im L_h^{-1} \circ E_h^{-1} (2\eta \Bi)| \leq C_1$, for some 
constant independent of $\alpha$ and $h$. 

Now we have, 
\begin{align*}
|\Im L_h^{-1}(x_h)| & \leq |\Im L_h^{-1} (n_\zeta)| + 2 M_1   \\
& \leq |\Im L_h^{-1}(n_\zeta + 2\eta\Bi )|+ M_12 \eta+ 2 M_1  \\
& = |\Im F_h\co{n_\zeta}( L_h^{-1}(2\eta\Bi))| + 2M_1 (\eta+1) \\ 
& \leq |\Im F_h\co{(-m_\zeta+n_\zeta)}( L_h^{-1}(2\eta \Bi))| + |m_\zeta| \frac{1}{4} + 2 M_1 (\eta+1) \\
&\leq  |\Im \big(F_h\co{(-m_\zeta+n_\zeta)}( L_h^{-1}(2\eta\Bi))-\frac{1}{\alpha}\big)| + b_h' \frac{1}{4} + 2M_1 (\eta+1) \\
& \leq |\Im L_h^{-1} \circ E_h^{-1}(2 \eta \Bi)|+  \frac{b_h'}{4} + 2M_1 (\eta+1) \\
& \leq C_1 +  \frac{b_h'}{4} + 2M_1 (\eta+1). 
\end{align*}

In the first line of the above equation we have used $n_\zeta\in (x_h-2, x_h)$ and Lemma~\ref{L:derivative-of-L_h^{-1}} with 
$\varepsilon=1$.
In the second line we have used Lemma~\ref{L:derivative-of-L_h^{-1}} with $\varepsilon=1$. 
In the third line we have used the functional equation~\eqref{equivariance}. 
In the fourth line we have used $|F_h(w)-w-1| \leq 1/4$ from Lemma~\ref{cyl-cond}. 
In the fifth line we have used $|m_\zeta|\leq b_h'$, where $b_h'$ is the constant in the proof of 
Proposition~\ref{P:width-of-L_h(X)}, which is uniformly bounded from above.  

On the other hand, by the definition of $x_h$, $L_h^{-1}(x_h+ \Bi \mathbb{R})$ touches the right hand boundary of $X$. 
Combining with the upper bound on $|\arg L_h^{-1}(x_h+ \Bi \mathbb{R})-\pi/2|$ in Lemma~\ref{L:domain-of-E-precoordinate}, 
we conclude that $|\Re L_h^{-1}(x_h)- 1/\alpha|$ is also uniformly bounded from above. This finishes the proof of the proposition. 
\end{proof}

\subsection{A uniform bound on $|L_h^{-1}(\zeta)-\zeta|$}\label{sec:uniform-bound-L_h}

\begin{propo}\label{P:horizontal-level-sets}
There exists a constant $C_4$ such that for every $\alf$ in $(0, r_2]$, every $h$ in $\QIS_\alf$, 
and every $t \in (0, x_h)$, 
\[|L_h^{-1}(t)-t|\leq  C_4 \min \{\log (2+t), \log (2+ x_h-t)\}.\]
\end{propo}

\begin{proof}
Within this proof all the constants $D_1, D_2, D_3, \dots$ are assumed to be independent of $\alpha$ and $h$. 

Let us define 
\[x_h'= \sup \{t\geq 0 \mid (0,t) \subset L_h(X)\}.\]
By definition, $x_h'\geq x_h$. 
However, by Lemma~\ref{cyl-cond}-2 and Proposition~\ref{P:value-at-end-by-L}, there is a constant $D_1$ 
such that $x_h'\leq x_h + D_1$ and $L_h^{-1}(x_h') \in B(1/\alpha, D_1)$.  

Let $l_1$ denote the vertical line $\Re w= C_2+5$, $l_2$ denote the vertical line $\Re w= 1/(2\alpha)$, and $l_3$ denote the 
vertical line $\Re w = 1/\alpha -C_2-5$. 
By Lemma~\ref{L:extending-L_h}, the closure of the curve $L_h^{-1}(0, x_h')$ connects the left hand boundary of 
$X$ to the right hand boundary of $X$. 
Let $t_1\in (0, x_h')$ be the smallest element with $L_h^{-1}(t_1) \in l_1$ and let $t_3\in (0, x_h')$ be the largest element with 
$L_h^{-1}(t_3) \in l_3$. 
By Lemma~\ref{cyl-cond}-2 and the above paragraph, there is a constant $D_2$ such that $t_1\leq D_2$, $x_h'-t_3\leq D_2$. 
For the same reason, there is also a constant $D_3$ such that 
\[|\Im L_h^{-1}(t)| \leq D_3, \forall t\in [0,t_1] \cup [t_3, x_h'].\]
So far we have shown that the inequality in the proposition holds for $t\in [0,t_1] \cup [t_3, x_h]$.  
Below we deal with values of $t\in [t_1, t_3]$. 

For $w \in X$ lying between $l_1$ and $l_3$ we have $B(w, 5) \subset A_1 \subset A_2$. 
Then by Lemma~\ref{L:derivative-of-L_h}-1, for $t\in (t_1, t_3)$, $|\arg (L_h'(L_h^{-1}(t)))| \leq \pi /3$ and 
$2/5 \leq |L_h'(L_h^{-1}(t))| \leq 8/3$. 
It follows that there is a unique $t_2\in (t_1, t_3)$ such that $L_h^{-1}(t_2) \in l_2$, and 
\begin{equation}\label{E:apriori-bound-location}
\begin{gathered}
\Re L_h^{-1} (t)\geq \Re L_h^{-1}(t_1)+ \frac{2}{5} \cos (\frac{\pi}{3}) (t-t_1)
\geq C_2+5 + \frac{1}{5}(t-t_1), \forall t\in (t_1, t_2), \\
\Re L_h^{-1}(t) \leq \Re L_h^{-1}(t_3) - \frac{2}{5} \cos (\frac{\pi}{3})(t_3-t) 
\leq \frac{1}{\alpha}  -C_2- 5 - \frac{1}{5} (t_3-t), \forall t\in (t_2, t_3), \\
|\Im L_h^{-1}(t)| \leq  |\Im L_h^{-1}(t_1)|+ \frac{8}{3} \sin (\frac{\pi}{3}) (t_3-t_1)\leq D_3 + \frac{4}{\sqrt{3}} (t_3-t_1), 
\forall t\in (t_1, t_3).
\end{gathered}
\end{equation}
In particular, by Proposition~\ref{P:width-of-L_h(X)}, $|\alpha \Im L_h^{-1}(t)| \leq \alpha (D_3 + 1/\alpha +k+D_1)
\leq D_3 + 1 + k+D_1$. 
Let $D_4= D_3 + 1 + k+ D_1$. 

On the other hand,
\begin{equation}\label{E:radii}
\begin{gathered}
d(L_h^{-1}(t), \mathbb{Z}/\alpha) = d(L_h^{-1}(t), 0) \geq \Re L_h^{-1}(t),\forall t\in (t_1, t_2), \\
d(L_h^{-1}(t), \mathbb{Z}/\alpha) = d(L_h^{-1}(t), 1/\alpha) \geq 1/\alpha - \Re L_h^{-1}(t), \forall t\in (t_2, t_3).
\end{gathered}
\end{equation}
Therefore, 
\[B(L_h^{-1}(t), \Re L_h^{-1}(t)-C_2  ) \subset A_1 , \forall t\in (t_1, t_2),\]
\[B(L_h^{-1}(t), \frac{1}{\alpha} - \Re L_h^{-1}(t) - C_2) \subset A_1 , \forall t\in (t_2, t_3),\]
Now, it follows from Lemma~\ref{L:derivative-of-L_h}-2 and Equations~\eqref{E:apriori-bound-location} 
and \eqref{E:radii} that there is a constant $D_5$, such that 
\begin{equation} \label{E:bound-on-derivative-real}
\begin{gathered}
|\Re L_h'(L_h^{-1}(t)) -1| \leq \frac{D_5}{5 + (t-t_1)/5}, \forall t\in (t_1, t_2)  \\
|\Re L_h'(L_h^{-1}(t)) -1| \leq \frac{D_5}{5 + (t_3-t)/5}, \forall t\in (t_2, t_3)
\end{gathered}
\end{equation}
and 
\begin{equation} \label{E:bound-on-derivative-imaginary}
\begin{gathered}
|\Im L_h'(L_h^{-1}(t)) | \leq \frac{D_5}{5+ (t-t_1)/5}, \forall t\in (t_1, t_2)    \\
|\Im L_h'(L_h^{-1}(t)) | \leq \frac{D_5}{5+ (t_3-t)/5}, \forall t\in (t_2, t_3)  
\end{gathered}
\end{equation}
Integrating the inequalities in \eqref{E:bound-on-derivative-real} we conclude that there is a constant $D_6$ such that for 
$t\in (t_1,t_2)$
\begin{equation*}
|\Re L_h^{-1}(t) -t|  \leq |\Re L_h^{-1}(t_1)-t_1|+\int_{t_1}^t \frac{D_5}{5+ (t-t_1)/5}\, dt  \leq D_6+D_6 \log (1+ (t-t_1)),
\end{equation*}
and for $t\in (t_2,t_3)$
\begin{equation*}
|\Re L_h^{-1}(t) -t| \leq |\Re L_h^{-1}(t_3)-t_3| + \int_{t}^{t_3} \frac{D_5}{5+ (t-t_1)/5}\, dt 
 \leq D_6 + D_6 \log (1 + (t_3-t)).  
\end{equation*}

Similarly, integrating the inequalities in \eqref{E:bound-on-derivative-imaginary} we obtain
\begin{align*}
|\Im L_h^{-1}(t)| & \leq D_3 + D_6 \log (1+ (t-t_1)), \forall t\in (t_1, t_2), \\
|\Im L_h^{-1}(t)| & \leq D_3+ D_6 \log (1+ (t_3-t)), \forall t\in (t_2, t_3).
\end{align*}
These imply the desired inequality in the proposition for $t\in [t_1, t_3]$, with a constant $C_4$ depending only on $D_3$ 
and $D_6$.  
\end{proof} 

For $z\in \mathbb{C} \setminus\{0\}$ there is an inverse branch of the covering map $\ex$ defined on a neighborhood of $z$. 
The derivative at $z$ of any such inverse branch is well-defined and is independent of the choice of the lift and 
the neighborhood. 
We let $(\ex^{-1})'(z)$ denote this complex number. 
As a corollary of Proposition~\ref{P:horizontal-level-sets} we obtain the following estimate on the derivative of $\Phi_h$, 
which is convenient to write in terms of $(\ex^{-1})'$.  
Recall the constant $k$ defined in Proposition~\ref{P:width-of-L_h(X)}.  

\begin{propo}\label{P:derivative-at-bottom-line}
There exists a constant $C_5$ such that for every $\alf$ in $(0, r_2]$, every $h$ in $\QIS_\alf$, and every 
$t$ with $1 \leq t \leq \min\{1/(2\alpha), x_h\}$, we have 
\[\frac{1}{C_4 t} \leq ( \ex^{-1} \circ \Phi_h^{-1})'(t)  \leq \frac{C_4}{t}.\]
\end{propo} 

\begin{proof}
By Proposition~\ref{P:width-of-L_h(X)}, $x_h \geq \alpha^{-1} -k$, for some uniform constant $k$. 
For large values of $\alpha$, $x_h$ may be less than $1/(2\alpha)$. 
The condition $1 \leq t \leq \min\{1/(2\alpha), x_h\}$ guarantees that $\Phi_h^{-1}(t)$ is defined.

We assume that the constants $D_i$, $i=1,2,3,\dots$, within this proof, are independent of $\alpha$ and $h$. 

Using Lemma~\ref{L:derivative-of-L_h^{-1}} with $\eps=1$, there is $D_1$ such that 
for all $t \in [1, x_h]$, $D_1^{-1} \leq |(L_h^{-1})'(t)| \leq D_1$.  
Therefore, it is enough to show that for some constant $D_2$ we have
\begin{equation}\label{E:uniforml-bound-on-lift}
\frac{1}{D_2 t} \leq (\ex^{-1} \circ \tau_h)'(L_h^{-1}(t)) \leq \frac{D_2}{t}.
\end{equation}
There are essentially two arguments to prove the above bounds. 
To introduce these, recall the constant $C_4$ in Proposition~\ref{P:horizontal-level-sets}. 
There is $t_0 \geq 1$ such that for $t \geq t_0$ we have $C_4 \log(2+t) \leq t/2$. 

\medskip

First we consider the case $\alpha \geq \min \{1/(5k), 1/(2t_0)\}$. 
By definition, for $t \in [1, x_h]$, $L_h^{-1}(t) \in X$, and for nonzero integers $n$, $X \cap B(n/\alpha, C_2)= \emptyset$. 
On the other hand, by Lemma~\ref{L:derivative-of-L_h^{-1}} with $\eps=1/2$, the image of the strip 
$1/2 \leq \Re \zeta \leq 1$ under $L_h^{-1}$ uniformly separates $L_h^{-1}(t)$ from $0$. 
That is, there is a constant $D_3$ such that for all $t\in [1, x_h]$, $|L_h^{-1}(t)| \geq D_3$. 
Therefore, for $t\in [1, x_h]$, $L_h^{-1}(t)$ is uniformly away from the set $\mathbb{Z}/\alpha$. 
On the other hand, since $x_h \leq \alpha^{-1}+k\leq 6k$ is uniformly bounded from above, 
by the pre-compactness of the class of maps $h$, for $t\in [1, x_h]$, $|L_h^{-1}(t)|$ is uniformly bounded from 
above independent of $t$, $\alpha$, and $h$. 
These imply that $|\tau_h(L_h^{-1}(t))|$ is uniformly bounded from above and away from $0$, by constants independent of 
$t$, $\alpha$, and $h$. 

By the above paragraph, and explicit estimates of the formula for $\ex^{-1} \circ \tau_h$, there is a constant $D_4$ such that 
$D_4^{-1} \leq |(\ex^{-1} \circ \tau_h)'(L_h^{-1}(t))| \leq D_4$.
In this case, $t$ is uniformly bounded away from $0$ and from above by $1 \leq t \leq x_h\leq \alpha^{-1}+k \leq 6k$. 
Therefore, one may adjust the constant $D_4$ to some uniform constant $D_2$ so that 
Equation~\eqref{E:uniforml-bound-on-lift} holds for $t\in [1, x_h]$. 

\medskip

Now assume that $\alpha \leq \min \{1/(5k), 1/(2t_0)\}$. 
Here, by Proposition~\ref{P:width-of-L_h(X)}, $1/(2\alpha) \leq x_h$, and $t_0 \leq 1/ (2\alpha)$. 
On the uniformly bounded subset $1\leq t \leq t_0$, the above argument may be repeated to conclude that 
$t \cdot |(\ex^{-1} \circ \tau_h)'(L_h^{-1}(t))|$ is uniformly bounded from above and away from $0$ on this interval. 
It remains to consider $t \in [t_0, 1/(2\alpha)]$.

For $t \in [t_0, 1/(2\alpha)]$ define the set
\[O_t:= \left \{\xi \in \cc:  |\Im \xi| \leq t/2, |\Re \xi-t| \leq t/2\right\}.\]
Using the uniform bound $|\sigma_h| \leq C_1 \alf$ in Equation~\eqref{E:bound-on-sigma}, and 
some explicit estimates of $\ex^{-1} \circ \tau_h$, there exists a constant $D_5$ such that 
\[\frac{1}{D_5 t}\leq |(\ex^{-1} \circ \tau_f)'(t)| \leq \frac{D_5}{t}.\]
For $1 \leq t \leq 1/(2\alpha)$, the modulus of the annulus $\Phi_f(\p_f) \setminus O_t$ is uniformly 
bounded away from zero, by a constant independent of $t$, $\alf$, and $f$.
Then, one infers from the above bounds and the distortion Theorem~\ref{T:Distortions} that there is a 
constant $D_6$ such that  for all $\xi \in O_t$, 
\[\frac{1}{D_6 t}\leq |(\ex^{-1} \circ \tau_f)'(\xi)|\leq \frac{D_6}{t}.\]
By Proposition~\ref{P:horizontal-level-sets} and the choice of $t_0$, for $t_0 \leq t \leq 1/(2\alpha)$, 
$L_f^{-1}(t) \in O_t$. 
Thus, we obtain the uniform bounds in Equation~\eqref{E:uniforml-bound-on-lift} for $t \in [t_0, 1/(2\alpha)]$.
\end{proof}

\begin{propo}\label{P:nearly-translation-pre-coord}
For every $M'\in \mathbb{R}$, there is $M>0$, such that for all $\alf\in (0, r_2]$, all $h\in \QIS_\alf$, and all 
$\zeta \in [0, x_h]+\textnormal{\Bi}[M', +\infty)$ we have 
\[|L_h^{-1}(\zeta)-\zeta|\leq M \log (1+1/\alf).\]
\end{propo}

\begin{proof}
All the constants $D_1, D_2, D_3, \dots$ within this proof are assumed to be independent of $\alpha$ and $h$. 
Without loss of generality we may assume that $M' \leq 0$. 

Recall the numbers $s_1$ and $s_2$ introduced in Equation~\eqref{E:s_1-s_2}, and define
\[Y_1=[0, x_h]+\Bi[M', +\infty), \quad Y_2=[s_1, s_2]+\Bi[M', +\infty). \]
By Lemmas~\ref{cyl-cond}-2, \ref{L:extending-F_h}-4, and \ref{L:extending-L_h}, 
there is a constant $D_1$ such that $s_1 \leq D_1$ and $x_h-s_2 \leq D_1$. 
In particular, by Lemma~\ref{cyl-cond}-2, and Equation~\eqref{equivariance}, it is enough to prove the uniform 
upper bound in the proposition for values of $\zeta\in Y_2$. 
On the other hand, by the maximum principle, it is enough to prove the inequality on the boundary of $Y_2$.  

Assume that $\zeta \in \partial Y_2$ with $\Re \zeta =s_1$. 
By Lemma~\ref{L:domain-of-E-precoordinate}-(2), 
$B(L_h^{-1}(\zeta), 5) \subset A_1 \subset A_2$. 
Then, by Lemma~\ref{L:derivative-of-L_h}-(1), $|\arg L_h'(L_h^{-1}(\zeta))| \leq \pi/3$ and 
$ 2/5 \leq |L_h'(L_h^{-1}(\zeta))| \leq 8/3$. 
Hence, the function $\Im \zeta \mapsto \Im L_h^{-1}(\zeta)$ is strictly monotone, and there is a constant $D_2$ such 
that  
\[\frac{8}{3} \Im \zeta +D_2  \geq \Im L_h^{-1}(\zeta) \geq \frac{2}{5} \sin \frac{\pi}{3} \cdot (\Im \zeta) -D_2.\] 
This implies that there are $t_1 \leq t_2$, with $|t_1|$ and $\alpha |t_2|$ uniformly bounded from above, 
such that $\Im L_h^{-1}(s_1 + \Bi t_1) \geq 3.25$ and $\Im L_h^{-1}(s_1 + \Bi t_2) \geq 1/(2\alpha)$. 
By Lemma~\ref{L:derivative-of-L_h}-(2), there is a constant $D_3$ such that for all $t\in [t_1,t_2]$, 
$|(L_h^{-1})'(\zeta) - 1 | \leq D_3/\Im \zeta$.
Integrating this inequality we conclude that there is a constant $D_4$ such that for $\Im \zeta \in [t_1, t_2]$, 
$|L_h^{-1}(\zeta)- \zeta|$ is bounded from above by a uniform constant times $\log \alpha^{-1} + D_4$. 

On the other hand, by Proposition~\ref{fine-tune} with $M'=1/2$ and $r=1/2$, 
the map $L_h$ is uniformly close to a translation on vertical lines with imaginary part bigger than $1/(2\alpha)$, 
with an error bounded by $M'/r=1$. 
It follows from Lemma~\ref{L:derivative-of-L_h^{-1}} that for $t \geq t_2$, 
the function $t\mapsto L_h^{-1}(s_1+ \Bi t)$ is uniformly close  to a translation. 
Combining this with the above paragraph, we conclude there is a constant $D_5$ such that 
$|L_h^{-1}(\zeta)-\zeta| \leq D_5 \log (1+ 1/\alpha)$, for $\zeta$ with $\Re \zeta=s_1$ and $\Im \zeta\geq t_1$. 

By an identical argument, one can show that there are $t'_1$, with $|t_1'|$ uniformly bounded from above, 
and a constant $D_6$ such that $|L_h^{-1}(\zeta)-\zeta| \leq D_6 \log (1+ 1/\alpha)$, 
for $\zeta$ with $\Re \zeta=s_2$ and $\Im \zeta\geq t_1'$. 

Let $D_7= \max\{t_1, t_2\}$ and assume that $\zeta \in \partial Y_2$ with $\Im \zeta \leq D_7$.
Recall that $s_1>1$ by Lemma~\ref{L:domain-of-E-precoordinate}-(1). 
By Lemma~\ref{L:derivative-of-L_h^{-1}} with $\varepsilon=1$, Proposition~\ref{P:width-of-L_h(X)}, and 
Proposition~\ref{P:horizontal-level-sets},  
\begin{align*}
|L_h^{-1}(\zeta)-\zeta| &\leq |L_h^{-1}(\zeta) - L_h^{-1}(\Re \zeta)| +  |L_h^{-1}(\Re \zeta) - \Re \zeta| + | \Re \zeta- \zeta| \\
&\leq M_1 \max\{M',D_7\}+ C_4 \log (2+ \alpha^{-1}+ k) +\max\{M',D_7\}.
\end{align*}
This implies the desired inequality, with a uniform bound, on the boundary of $Y_2$ with imaginary part less than $D_7$.

By Equation~\eqref{E:asymptote-of-L_h^{-1}}, $|L_h^{-1}(\zeta)-\zeta|$ tends to a constant near the top end of $Y_2$. 
The limit is independent of $\zeta$, and, by the above paragraphs, it is bounded from above by a uniform constant times 
$\log (1+ \alpha^{-1})$ along $\Re \zeta=s_1$.  
Then, $|L_h^{-1}(\zeta)-\zeta|$ is bounded by a uniform constant times $\log (1+ \alpha^{-1})$ near the top end of $Y_2$. 
Combining with the above paragraphs, the desired inequality with a uniform bound is proved on the boundary of 
$Y_2$. 
This finishes the proof of the proposition. 
\end{proof}

The above proposition gives an upper bound on the asymptotic translation of $L_h^{-1}$ in 
Equation~\eqref{E:asymptote-of-L_h^{-1}}, which we state below for reference purposes.

\begin{cor}\label{C:size-of-translation}
There is $C>0$ such that for all $\alf \in (0, r_2]$ and all $h\in \QIS_\alf$, we have 
\[|\ell_h| \leq C \log (1+\alf^{-1}).\] 
\end{cor}

\begin{rem}\label{R:logarithmic-error}
When $\alf$ tends to zero, $h(z):=P\circ \phi^{-1}(\ea \cdot z)$ tends to a map $h_0$ with a parabolic fixed point at $0$. 
Then, $L_h$ tends to some univalent map $L_{h_0}$ which is the lift of the attracting Fatou coordinate of $h_0$ under 
the change of coordinate $w=-2/(h''(0)z)$. 
It is known that $L_{h_0}$ has asymptotic expansion $w+a\log w+c+o(1)$ near $+\Bi\infty$,
for some constants $a$ and $c$, for instance see~\cite[Prop. 2.2.1]{Sh00}. 
So, it seems the logarithmic bound in Propositions~\ref{P:horizontal-level-sets} and \ref{P:nearly-translation-pre-coord} 
is necessary. 
The main point used in this paper is that since $\alf \log (1+1/\alf)\to 0$ as $\alf\to0$, the logarithmic error is 
absorbed in the formula of $\tau_h$ as the rotation numbers degenerate along the renormalization tower.
\end{rem}

\subsection{Geometry of the petals}\label{sec:geometry-petals}
Recall that $L_h= \Phi_h \circ \tau_h$, where $\tau_h$ is given by the formula~\eqref{E:covering-formula}, 
$L_h$ and its domain of definition $X$ are defined in Section~\ref{SS:Fatou-coordinate-of-F_h}, and $\Phi_h$ is the 
Fatou coordinate defined in Theorem~\ref{Ino-Shi1}.
In this section we employ the estimates on $L_h$ established in Sections~\ref{sec:estimates-L_h}, \ref{sec:width-of-L_h(X)}, 
and \ref{sec:uniform-bound-L_h}, as well as some explicit estimations of $\tau_h$ to prove the geometric properties of the 
maps $\Phi_h$, as well as their domain and range, stated in Propositions~\ref{P:uniformly-bounded-width-spiral}, 
\ref{P:turning}, \ref{P:smallsector}, and \ref{P:r-ball}.

\begin{proof}[Proof of Proposition~\ref{P:uniformly-bounded-width-spiral}]
Let $r_2$ be the constant defined in Equation~\eqref{E:r_2}. 
Let us fix $h\in \QIS_\alpha$, $\alpha\in (0, r_2]$. 
By the definitions of the constants $x_h \geq y_h$ in Equations~\eqref{E:x_h} and \eqref{E:y_h}, the map 
$L_h^{-1}$ is defined on $(0, y_h)+ \Bi \mathbb{R}$, and $\tau_h$ is univalent on $L_h^{-1}((0, y_h)+ \Bi \mathbb{R})$. 
Moreover, by Lemma~\ref{L:extending-L_h}, $L_h^{-1}$ is univalent on $(0, y_h)+ \Bi \mathbb{R}$. 
We define 
\[\mathcal{P}_h=\tau_h \circ L_h^{-1}((0, y_h)+ \Bi \mathbb{R}), \Phi_h=L_h \circ  \tau_h^{-1}: 
\mathcal{P}_h \to \mathbb{C}.\] 
By Proposition~\ref{P:width-of-L_h(X)}, $y_h \geq \alpha^{-1}-k$, where $k$ is a constant independent of $\alpha$ and $h$. 
This implies part 2 of the proposition by defining $\Bk=k$. Below we prove part 1 of the proposition. 

Since $\mathcal{P}_h$ is simply connected and $0 \notin \mathcal{P}_h$, there is a continuous branch of 
argument defined on $\mathcal{P}_h$. 

The map $L_h$ has a holomorphic extension onto a neighborhood of $L_h^{-1}(\Bi \mathbb{R})$ using the functional equation 
$L_h(F_h(\zeta))= L_h(\zeta)+1$.  
Then, $L_h^{-1}(\Bi \mathbb{R})$ intersects the real axis at least once and at most at a finite number of points. 
There is $t_1\in \mathbb{R}$ such that $L_h^{-1}(t_1 \Bi)$ is the closest point to $0$ among all such intersections.
That is, $L_h^{-1}(t_1 \Bi)\in \mathbb{R}$. 
Moreover, there is $t_2> t_1$ such that for $t \geq t_2$, $L_h^{-1}(t\Bi) \notin B(0, C_2)$, where $C_2$
is the constant in Lemma~\ref{cyl-cond}. 
By the normalization of $L_h$, $L_h^{-1}(0)=\cp_{F_h} \in B(0, C_2)$, where $\cp_{F_h}$ denotes the critical point of $F_h$. 
It follows from the pre-compactness of the class of maps $h$ that $|t_1|$ and $t_2$ are 
uniformly bounded from above independent of $h$ and $\alpha$. 

Let us define the curve $\gamma_1$ as the interval $(0, L_h^{-1}(t_1 \Bi)]\subset \mathbb{R}$, 
the curve $\gamma_2$ as $L_h^{-1}(t \Bi)$ for $t\in [t_1, t_2]$, and the curve $\gamma_3$ as $L_h^{-1}(t \Bi)$ 
for $t\in [t_2, \infty)$. 
We denote the union of these curves by $\gamma$. 
Since $L_h^{-1}((0, y_h)+ \Bi \mathbb{R}) \subset X$ and $X \cap B(n, C_2)= \emptyset$, for integers $n\neq 0$, 
$\tau_h(\gamma)$ is a simple curve in $\mathbb{C} \setminus \mathcal{P}_h$.
Moreover, $\tau_h(\gamma)$ belongs to $\mathbb{C}\setminus \{0\}$ and tends to $0$ and infinity at its two ends.  
To prove the existence of a constant $\hat{\Bk}$ for part 1 of the proposition, it is enough to show that for any continuous 
branch of argument defined on $\tau_h(\gamma)$, $\sup |\arg w - \arg w'|$, for $w,w' \in \tau_h(\gamma)$, 
is uniformly bounded from above by a constant independent of $h$ and $\alpha$.

The curve $\gamma_1$ is mapped to a half-line under $\tau_h$ tending to infinity in $\mathbb{C}$. 
Hence, $\sup |\arg w - \arg w'|$, for $w,w' \in \tau_h(\gamma_1)$ is bounded by $\pi$. 
On the other hand, for every $h$, $\sup |\arg w - \arg w'|$, for $w,w' \in \tau_h(\gamma_2)$ is finite. 
Hence, by the pre-compactness of the class of maps $h$, there is a uniform upper bound on these numbers 
independent of $h$ and $\alpha$.  

Using Proposition~\ref{P:nearly-translation-pre-coord} with $M'=0$, there is $M>0$ such that for $t\geq t_2$, 
$|L_h^{-1}(t \Bi)-t\Bi | \leq M \log (1+1/\alpha)$. 
Therefore, $\gamma_3$ is contained in the set 
\[[-M \log (1+1/\alpha), M \log (1+1/\alpha)] + \Bi [-M \log (1+1/\alpha), +\infty).\] 
Recall that $\tau_h$ is periodic of period $1/\alpha$, and it maps every vertical line to an arc of a circle connecting $0$ to 
$\sigma_h$ (Each such arc segment spirals at most by $\pi$ about $0$). 
Since $\alpha \log (1+\alpha^{-1})$ is uniformly bounded from above, it follows that $\sup |\arg w - \arg w'|$, 
for $w,w' \in \tau_h(\gamma_3)$, is uniformly bounded from above by a constant depending only on $M$. 
\end{proof}

\begin{proof}[Proof of Proposition~\ref{P:turning}]
Fix $h$ with $h'(0)=\ea$ and $\alpha\in (0, r_3]$. 
By Equation~\eqref{E:asymptote-of-L_h^{-1}}, each curve $\tau_h\circ L_h^{-1}(t+\Bi \mathbb{R})$, 
$t\in [0, 1/\alf-\Bk]$, approaches zero with a well-defined tangent at $0$.
This implies that if $w\in \C_h\cup \Csh_h$ is close enough to zero, there exists a unique inverse orbit 
$w, h^{-1}(w), \dots, h^{-j}(w)$ staying near zero such that $j$ is the smallest positive integer with $h^{-j}(w)\in \p_h$. 
Comparing with the rotation of angle $\alf$, one can see that for $|w|$ small enough $\Bk+1 \leq  j \leq \Bk+2$.

By Theorem~\ref{Ino-Shi2}, $\rr (h)$ is of the form $z \mapsto P \circ \psi^{-1}(e^{2\pi \Bi/\alf}\cdot z)$, 
where $\psi:U \to \mathbb{C}$ is a univalent mapping that has a univalent extension 
onto the larger domain $V$ which contains the closure of $U$. 
By the distortion Theorem~\ref{T:Distortions}, $\rr(h)$ must be uniformly close to a rotation on $U$, 
with a constant independent of $h$ and $\alpha$. 
Moreover, we also conclude that the pre-image (under $\rr(h)$) of any ray in $P(U)$ landing at $0$ must have 
uniformly bounded spiral about zero. 
Thus, any lift of $\rr(h)$ under $\ex$ must be uniformly close to some translation, with the bound independent 
of $h$ and $\alpha$.  
This implies that  $|k_h-j|$, for any $j$ as above, is uniformly bounded from above. 
In particular, $k_h \leq |k_h-j|+ j$ is uniformly bounded from above independent of $h$ and $\alpha$. 
\end{proof}

\begin{proof}[Proof of Proposition~\ref{P:smallsector}]
Recall that for $\alpha\leq r_3$ the sector $S_f$ is defined (see Theorem~\ref{Ino-Shi2}).
Let $M$ denote the constant produced by Proposition~\ref{P:nearly-translation-pre-coord} applied with $M'=-2$. 
We consider two cases separately. 

Recall the constant $k$ from Proposition~\ref{P:width-of-L_h(X)}. 
The first case is to assume that $\alf$ is small enough such that  
\[\alf\leq 1/(3+2k) \text{ and } 3/2+M \log (1+1/\alf)\leq 1/(4\alf).\] 
Define the set
$A:= [\lfloor 1/(2\alf)\rfloor+1/2, \lfloor 1/(2\alf)\rfloor+3/2]+ \Bi [-2, +\infty)$. 
By definitions, $A=\Phi_f \circ f\co{(k_f+\lfloor 1/(2\alf)\rfloor)}(S_f )$,  and by the above condition on $\alf$, 
$A$ is contained in $[0, y_h]+ \Bi [-2/\alf, \infty)$. 
By Proposition~\ref{P:nearly-translation-pre-coord} with $M'=-2$, for all $\zeta\in A$ we have 
\[\Im L_h^{-1}(\zeta)\geq -2 -M \log (1+1/\alf),\; |\Re L_h^{-1}(\zeta)-1/(2\alf)| \leq 3/2+M \log (1+1/\alf).\] 
Now, using the second condition on $\alf$, the uniform bound $|\sigma_f| \leq C_1 \alf$ in Equation~\eqref{E:bound-on-sigma}, 
and an explicit calculation of the formula for $\tau_f$, there is a constant $M_1'$ (depending only on $M$) such that 
\begin{equation*}
\diam (f\co{(k_f+\lfloor 1/2\alf \rfloor)}(S_f)) 
 = \diam (\Phi_f^{-1}(A)) 
 = \diam (\tau_f\circ L_f^{-1}(A))
 \leq M_1'\cdot \alf.
\end{equation*}
Here, when estimating $\tau_f$, one uses that $\alf\log (1+1/\alf)$ is uniformly bounded from above on $(0,1)$. 

The second case is to assume that $\alf$ is larger than some constant, but still less than $r_3$. 
Recall the constant $s_1$ defined in Equation~\eqref{E:s_1-s_2}, and define 
$A:=[\lfloor s-1/2\rfloor, \lfloor s+1/2\rfloor] + \Bi[-2, +\infty)$.  
By Lemmas~\ref{L:domain-of-E-precoordinate}-2 and \ref{L:derivative-of-L_h}-1, for all $t\in \mathbb{R}$, 
$|\arg L_h'(L_h^{-1}(s_1+\Bi t))| \leq \pi/3$.  
Then, by Lemmas~\ref{cyl-cond}, $L_h^{-1}(A)$ must be contained within $5/4$ of the set 
$|\arg(w-2C_2-10)-\pi/2| \leq \pi/3$. 
Also, as in the previous case, $A$ lies above the line $\Im w=-2-M\log (1+1/\alf)$. 
Then, an explicit calculation on $\tau_f$ shows that for all $\zeta \in A$, $|\tau_f\circ L_f^{-1}(\zeta)|\leq M_1''$, 
for some constant $M_1''$, independent of $\alf$ and $f$. 
Here we only use that $|\sigma_f|$ is uniformly bounded from above independent of $\alpha$ and $f$ 
(Indeed, $|\sigma_f| \leq 4/27$ is assumed for $\alpha \leq r_2'$ in the proof of Lemma~\ref{cyl-cond}).
Because $\alf$ is bounded from below in this case, one can adjust the constant $M_1''$ such that 
$|\tau_f\circ L_f^{-1}(\zeta)|\leq M_1'' \alpha$ holds. 
\end{proof}

\begin{proof}[Proof of Proposition~\ref{P:r-ball}]
There are two arguments; one for values of $\alpha$ near $0$ and one for values of $\alpha$ away from zero. 
First we present the former case, where we impose a number of upper bounds on $\alpha$ along the way so that 
the proof works. 
The second case is based on a pre-compactness argument, and is presented at the end of this proof.

All the constants $D_1, D_2, D_3, \dots$ introduced within this proof are assumed to be independent of $\alpha$ 
and $f \in \QIS$. 

Consider the line segment 
\[\vtet(t):=t-(2+t/2)\Bi, \text{ for } t\in [2, 1/(2\alf)].\]
Let $\hat{\eta}_f:\p_f \to \cc$ be an arbitrary inverse branch of the covering map $\ex$, and define 
\begin{equation} \label{lift}
\chi_f:=\hat{\eta}_f \circ \Phi_f^{-1}: \Phi_f(\mathcal{P}_f)\to \mathbb{C}. 
\end{equation} 
We shall fixed the choice of the branch of $\hat{\eta}_f$ in a moment, but until then, all statements involving $\hat{\eta}_f$
are independent of the choice of the branch. 

For 
\begin{equation}\label{E:alpha-condition-1}
\alpha \leq \frac{1}{4\Bk}, 
\end{equation}
by Proposition~\ref{P:width-of-L_h(X)}, the image of $\vtet$ is contained in the domain of $\chi_f$. 

\begin{sublem}\label{long-lifts}
There is $D_1>0$ such that 
\[|\Im \chi_f(\vtet(2))|\leq D_1,\quad
\Im \chi_f(\vtet(\frac{1}{2\alf}))\geq \frac{1}{2\pi}\log \frac{1}{\alf}- D_1.\] 
\end{sublem}

\begin{proof}
There is a topological annulus $A$ in $\Phi_f(\p_f)$ which separates the pair of points $1$ and $\vtet(2)$ from a 
neighborhood of $+\Bi \infty$ and its modulus is uniformly bounded away from $0$, independent of $\alpha$ and $f$. 
Recall that $\chi_f(1)\in \mathbb{Z}$. 
The univalent map $\chi_f$ lifts $A$ to an annulus of the same modulus in $(\cc \setminus \mathbb{Z}) \cup \{\chi_f(1)\}$, 
which encloses the pair of points $\chi_f(1)$ and $\chi_f (\vtet(2))$. 
As the modulus of this annulus is uniformly bounded away from $0$, $|\chi_f(1) -  \chi_f(\vtet(2))|$ must be 
uniformly bounded from above. This implies the first inequality in the sublemma. 

By Proposition~\ref{P:nearly-translation-pre-coord} (with $M'=0$) 
and an explicit estimate on $\tau_f$, there is a constant $D_2$ such that 
$|\Phi_f^{-1}(1/(2\alf))| \leq D_2 \alf$.   
Therefore, 
\begin{equation*}
\Im \chi_f(\frac{1}{2\alf})) \geq \frac{1}{2\pi}\log \frac{1}{\alf}-\frac{1}{2\pi} \log \frac{27 D_2}{4}.
\end{equation*}
Let $\gamma$ denote the line segment connecting $1/(2\alpha)$ to $\vtet(1/ (2\alpha))$. 
Then, by Equation~\eqref{E:alpha-condition-1}, the modulus of the annulus $\Phi_f(\p_f) \setminus \gamma$ 
is uniformly bounded away from $0$ by a constant independent of $\alpha$ and $f$. 
By proposition~\ref{P:derivative-at-bottom-line} and the distortion Theorem~\ref{T:Distortions}, $|\chi_f'|$ on $\gamma$ 
is bounded from above by a uniform constant times $\alpha$. 
This implies that the length of the curve $\chi_f(\gamma)$ is uniformly bounded from above. 
Combining this with the above inequality we conclude the second inequality in the sublemma.
\end{proof}

\begin{sublem}\label{nice-derivative}
There is $D_3>0$ such that for all $t\in [2, 1/(2\alpha)]$ we have 
\[\frac{1}{D_3t}\leq |\chi_f'(\vtet(t))|\leq \frac{D_3}{t}.\]
\end{sublem}

\begin{proof}
First note that by Equation~\eqref{E:alpha-condition-1} and Proposition~\ref{P:width-of-L_h(X)},
$1/(2\alpha) \leq x_h$.  
For each $t$ in $[2, 1/(2\alpha)]$, let $\gamma_t$ denote the line segment connecting $t$ to $\vtet(t)$. 
For $\alpha$ satisfying Equation~\eqref{E:alpha-condition-1}, the modulus of the annulus $\Phi_f(\p_f)\setminus \gamma_t$
is uniformly bounded away from $0$. 
Thus, by the uniform bounds in Proposition~\ref{P:derivative-at-bottom-line} and the distortion Theorem~\ref{T:Distortions}, 
we obtain the uniform bound in the sublemma. 
\end{proof}


When
\begin{equation}\label{E:alpha-condition-3}
\alf \leq \frac{1}{4(\Bk''+\Bk+2)},
\end{equation}
for every $t\in [2, 1/(2\alpha)$, we have 
\begin{equation}\label{empty-balls}
\begin{gathered}
B(\vtet(t),t/2) \subset \{w\in \cc: 1 \leq \Re(w)\leq \alf^{-1}-\Bk-\Bk''-2, \Im w \leq -2\},\\
B(\vtet(t),t/2)+1 \subset \{w\in \cc: 1 \leq \Re(w)\leq \alf^{-1}-\Bk-\Bk''-1, \Im w \leq -2\}.
\end{gathered}
\end{equation}
In particular, $\chi_f$ is defined and univalent on $B(\vtet(t), t/2)$. 
By Sublemma~\ref{nice-derivative} and the Koebe 1/4-theorem, 
\begin{equation}\label{lift-of-big-balls} 
B \big (\chi_f(\vtet(t)), \frac{1}{8 D_3}\big ) \subseteq \chi_f \big(B(\vtet(t),t/2)\big).
\end{equation}
Let us define 
\[D_4:=\min \{1/(8D_3), 1/4\}.\]

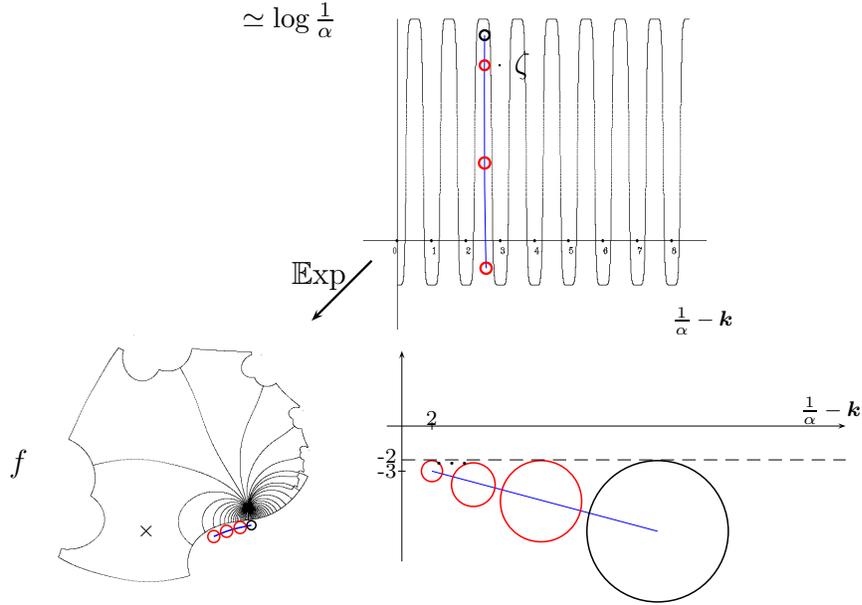
\begin{figure}
\begin{center}
\begin{pspicture}(1,-.6)(12.5,8)  
\epsfxsize=5cm
\rput(8.6,5.5){\epsfbox{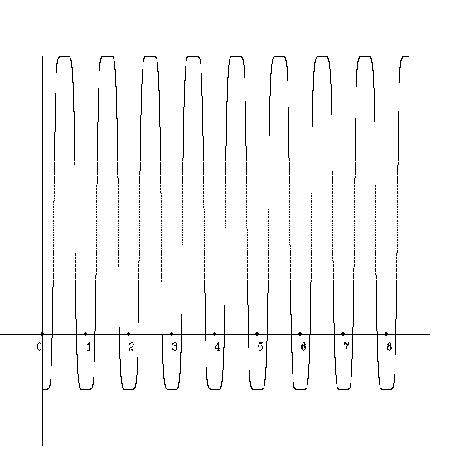}}
\psdot[dotsize=.9pt](7.9,6.8)
\rput(8.2,6.8){$\zeta$}
\rput(5.1,7.4){\small{$\simeq \log \frac{1}{\alf}$}}
\rput(10.6,3.4){\tiny{$\frac{1}{\alf}-\Bk$}}
\epsfxsize=4cm
\rput(4,1.5){\epsfbox{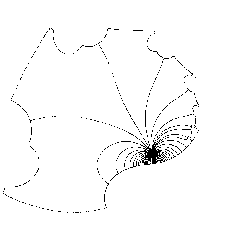}}
\rput(1.5,1.5){$f$}
\psline[linewidth=.03]{->}(6.2,4.2)(5.4,3.4)
\rput(5.5,4){$\ex$}

\psline[linewidth=.01]{->}(6.6,.2)(6.6,3)
\psline[linewidth=.01]{->}(6.4,2)(12.5,2)

\psline[linewidth=.01,linecolor=blue](7,1.4)(10,.6)
\pscircle[linewidth=.02,linecolor=red](7,1.4){.15}
\pscircle[linewidth=.02,linecolor=red](7.55,1.22){.3}
\pscircle[linewidth=.02,linecolor=red](8.45,1){.55}
\pscircle[linewidth=.02](10,.6){.95}
\rput(7.3,1.5){\small{$\dots$}}

\psline[linewidth=.005,linestyle=dashed](6.6,1.55)(12.5,1.55)
\rput(6.4,1.6){\tiny{-2}}
\rput(12.3,2.2){\tiny{$\frac{1}{\alf}-\Bk$}}
\psline[linewidth=.01](6.55,1.4)(6.65,1.4)
\psline[linewidth=.01](7,1.95)(7,2.05)
\rput(6.4,1.4){\tiny{-3}}
\rput(7,2.1){\tiny{2}}

\pscurve[linewidth=.02,linecolor=blue](4.1,.53)(4.27,.6)(4.45,.65)(4.6,.68)
\pscircle[linewidth=.02,linecolor=red](4.1,.53){.09}
\pscircle[linewidth=.02,linecolor=red](4.27,.6){.09}
\pscircle[linewidth=.02,linecolor=red](4.45,.65){.09}
\pscircle[linewidth=.02](4.6,.68){.07}

\pscurve[linewidth=.01,linecolor=blue](7.72,4.1)(7.7,5.2) (7.7,7.2)
\pscircle[linewidth=.03,linecolor=red](7.72,4.1){.09}
\pscircle[linewidth=.03,linecolor=red](7.7,5.5){.09}
\pscircle[linewidth=.03,linecolor=red](7.7,6.8){.08}
\pscircle[linewidth=.03](7.7,7.2){.08}
\rput(3.2,.6){\tiny{$\times$}}
\end{pspicture}
\caption{The figure shows a cartoon of the lift of the sectors under $\ex$, and the balls in their complement.}
\label{complementary-Balls}
\end{center}
\end{figure}

Let $\zeta_0 \in \mathbb{C}$ be an arbitrary point that satisfies the hypothesis of the proposition. 
We have 
\begin{equation}\label{E:upper-bound-Im-zeta}
\Im \zeta_0 \leq \frac{1}{2\pi}\log \frac{1}{\alf}+E.
\end{equation}
Let us assume that $\zeta_0$ also satisfies 
\begin{equation}\label{E:lower-bound-Im-zeta}
\Im \zeta_0\geq 1.
\end{equation} 
(We shall deal with $\Im \zeta_0 \leq 1$ in a moment.)

By Sublemma~\ref{long-lifts}, $\Im \chi_f (\vtet(2))\leq D_1$ and 
$\Im \chi_f (\vtet(1/(2\alpha))) \geq (2\pi)^{-1} \log \alpha^{-1} -D_1$. 
If we assume that 
\begin{equation}\label{E:alpha-condition-4}
\alf\leq e^{-4\pi D_1-2\pi},
\end{equation} 
then 
\begin{equation*}
\frac{1}{2\pi} \log \frac{1}{\alpha} -D_1 \geq D_1+1.
\end{equation*}
Then, it follows from Equations~\eqref{E:upper-bound-Im-zeta} and \eqref{E:lower-bound-Im-zeta} that 
there exists $t' \in [2, 1/(2\alpha)]=\Dom \vtet$, such that 
\begin{equation}\label{E:intermediate-height-lift}
\Im(\chi_f(\vtet(t'))) \geq 1, \quad  -D_1+1 \leq \Im \zeta_0 -\Im \chi_f(\vtet(t')) \leq D_1+E.
\end{equation}
Note that $(\Phi_f \circ  \ex)^{-1}(\vtet(t'))$ forms a $1$-periodic set of points. 
Then, there is a choice of the branch of $\hat{\eta}_f$ such that  
\begin{equation}\label{E:1/2-close-real-parts}
\left | \Re \zeta_0 - \Re \chi_f(\vtet(t'))\right| \leq 1/2.
\end{equation} 
From here on we shall fix this choice of the inverse branch $\chi_f$. 
See Figure~\ref{complementary-Balls}.

Let us define the curve 
\[\gamma(s)=(1-s) \zeta_0+s\chi_f(\vtet(t')), s\in [0,1].\]
Fix an arbitrary $\delta_1\in (0,1/16)$. 

Since $D_4\leq 1/4$, by Equations~\eqref{E:lower-bound-Im-zeta} and \eqref{E:intermediate-height-lift} 
the set $B(\gamma(1),D_4)\cup \gamma[0,1]$ is contained above the line $\Im \zeta = 3/4$. 
On the other hand, by Lemma~\ref{L:very-basic-estimates-on-IS}, $\Dom f \supset B(0,8/9)$, and hence,
\begin{equation*}
\ex \big ( \{\zeta\in \mathbb{C} : \Im \zeta\geq 0 \} \big ) \subset \Dom f \setminus \{0\}. 
\end{equation*} 
In particular,
\[\ex \Big(B_{\delta_1} \big(B(\gamma(1),D_4)\cup \gamma[0,1]\big)\Big) \subseteq \Dom f\setminus \{0\}.\]
This proves Part (1) of the proposition.   

As $D_4 \leq 1/4$, by Equation~\eqref{E:1/2-close-real-parts}, 
\[\diam (\Re (B(\gamma(1),D_4) \cup \gamma[0,1]))\leq 1/2+1/4=3/4,\]
which implies Part (3) of the proposition.  

As $\gamma$ is a straight line segment of uniformly bounded length, one may choose a uniform $\delta_2$ 
for Part (4) of the proposition. 

Recall that by Proposition~\ref{P:turning}, $k_f\leq \Bk''$. 
This implies that 
\[\p_f \cap \cup_{j=0}^{k_f-1} f\co{j}(S_f) 
\subset \Phi_f^{-1} \left (\{w\in \mathbb{C} : \Re w \in [0,1/2] \cup (\alpha^{-1} - \Bk -\Bk''-1, \alpha^{-1} - \Bk)\}\right).\]
On the other hand, 
\[\p_f \cap\cup_{j=k_f}^{k_f+ \lfloor  1/\alpha \rfloor -\Bk -1} f\co{j}(S_f)
\subset \Phi_f^{-1} \big (\{w\in \mathbb{C} : \Im w \geq -2 \}\big ).\]
Hence, by Equation~\eqref{empty-balls}, 
\[\Omega_0^0(f) \cap \Phi_f^{-1}(B(\vtet(t'), t'/2)) = \emptyset,  \quad 
\Omega_0^0(f)  \cap  f(\Phi_f^{-1}(B(\vtet(t'), t'/2))) = \emptyset.\] 
On the other hand, by Equation~\eqref{lift-of-big-balls}, 
\[\ex (B(\gamma(1), D_4)) \subset \ex (\chi_f (B(\vtet(t'), t'/2)))= \Phi_f^{-1}(B(\vtet(t'), t'/2))).\]
Hence, 
\[\ex (B(\gamma(1), D_4)) \cap \Omega_0^0(f)=\emptyset, \quad 
f( \ex (B(\gamma(1), D_4))) \cap \Omega_0^0(f)=\emptyset, \]
as desired in Part (2) of the proposition.

This finishes the proof of the proposition when $\alpha$ satisfies Equations~\eqref{E:alpha-condition-1}, 
\eqref{E:alpha-condition-3}, and \eqref{E:alpha-condition-4}, as well as $\zeta_0$ satisfies 
Equation~\eqref{E:lower-bound-Im-zeta}. 
Below we consider the remaining case. 

\medskip

By the assumption $\Im \zeta_0\leq \frac{1}{2\pi} \log \frac{1}{\alf}+E$, if any of the conditions in 
Equations~\eqref{E:alpha-condition-1}, \eqref{E:alpha-condition-3}, \eqref{E:alpha-condition-4}, and 
\eqref{E:lower-bound-Im-zeta} does not hold, there is a uniform constant $D_5\geq 1$ such that 
\[\Im \zeta_0\leq  D_5.\] 
Below we prove the position for such points $\zeta_0$ (while assuming that $\alpha \in (0, r_2]$). 

Recall the sector $S_f$ defined in Section~\ref{SS:renormalization-def}. 
Let us denote the connected component of $f^{-1}(S_f)$ which lies in $\p_f$ by $S'_f$. 
In other words, $S'_f= \Phi_f^{-1}(\Phi_f(S_f)-1)$. 
The set $S_f'$ might be contained in $\Omega_0^0(f)$, but this does not make any difference in the argument 
we present below. 

There is a constant $D_7>0$, independent of $\alpha$ and $f$, such that $\Omega_0^0(f) \cup S_f' \subset B(0,D_7)$. 
To see this, first note that there is an integer $n_f \geq 0$, uniformly bounded from above, such that 
$L_f^{-1}(\{\zeta\in \mathbb{C} \mid n_f+1/2 \leq \Re \zeta \leq x_f   \})$ is contained in $\Theta_\alpha(C_2)$, 
where $C_2$ is the constant in Lemma~\ref{cyl-cond}. 
Then, 
\[\bigcup_{i=k_f + n_f}^{k_f+ \lfloor 1/\alpha_n\rfloor -\Bk-2}  f\co{i}(S_f)   
\subseteq  \Phi_f^{-1}(\{\zeta \in \mathbb{C} \mid n_f+1/2 \leq \Re \zeta \leq x_f \})
\subset \tau_f (\Theta_\alpha(C_2)).\] 
The diameter of $\tau_f (\Theta_\alpha(C_2))$ is uniformly bounded from above, independent of $\alpha$ and $f$. 
On the other hand, for $f \in \QIS_0$ the sets $f\co{i}(S_f)$, for $i \geq 0$, are defined and compactly contained in $\Dom f$. 
Indeed,  according to \cite{IS06}, for every $f\in \QIS_0$, $S_f$ and all its forward and backward iterates are defined and
contained in $\Dom f$. 
Moreover, $f\co{i}(S_f)$, for $i\geq 0$, tend to $0$ in the attracting direction, and $f\co{i}(S_f)$ within the 
repelling Fatou coordinate, for $i \leq 0$, tend to $0$ in the repelling direction. 
For each $\alpha \in [0, r_2]$ and $f\in \QIS_\alpha$, the diameter of each $f\co{i}(S_f)$, for $0\leq i \leq k_f+n_f$, 
(let $n_f=\infty$ when $\alpha=0$), are finite. 
Similarly, the diameter of each $S_f'$ is also finite. 
Therefore, by the pre-compactness of the class of maps $\cup_{\alpha\in [0,r_2]} \QIS_\alpha$ and the continuous 
dependence of the Fatou coordinates on the maps, the diameters of these sets are uniformly bounded from above, 
independent of $\alpha$ and $f$. 
This proves the existence of the uniform constant $D_7$. 

By the pre-compactness of the class of maps $\cup_{\alpha\in [0,r_2]} \QIS_\alpha$, there is a constant $\delta>0$ 
such that 
\[B_{\delta}(\Omega^0_0(f) \cup S_f') \subset \Dom f.\]
See Equation~\eqref{well-contained} for further details. 

Let $\ex(\zeta_0)=z$. 
By the hypothesis of the proposition we have $z \in \Omega_0^0(f)$, and hence by the above paragraph, $|z| \leq D_7$. 
Moreover, because $\Im \zeta_0 \leq D_5$ here, $|z| \geq 4e^{-2\pi D_5}/27$ is uniformly bounded from below. 

Consider the smallest $r\geq 1$ such that $rz \in \partial (\Omega_0^0(f) \cup S'_f)$. 
That is, $r' z \in \Omega_0^0(f) \cup S'_f$ for all $r\in [1, r]$. 
Let $z'=rz$. 

Recall the set $\mathcal{C}_f^{-1}$ defined in Section~\ref{SS:renormalization-def}.  
We are looking for a small ball near $z'$ that is outside $\Omega_0^0(f)$ and is mapped outside $\Omega_0^0(f)$ by $f$.
However, points outside but near $\partial \mathcal{C}_f^{-1}$ may be mapped into $\Omega_0^0(f)$, due to the 
branched covering $f: \mathcal{C}_f^{-1}\to \mathcal{C}_f$. 
Also, if the topological interior of the set $S_f' \setminus \Omega_0^0(f)$ is not empty, points within this set are mapped 
into $S_f \subset \Omega_0^0(f)$. 
But, these are the only ways in which this issue occurs. 
Indeed, it follows from the definition of the sets $f\co{i}(S_f)$, $0\leq i \leq k_f+ \lfloor 1/\alpha \rfloor -\Bk-2$, 
and the way they are mapped to one another, that a given point $z_0 \in \partial \Omega_0^0(f)$ is mapped to a point on 
$\partial \Omega_0^0(f)$, unless either 
\begin{inparaenum}
\item[(a)]$z_0 \in \partial \mathcal{C}_f^{-1}$, or 
\item[(b)] both $z_0$ and $f(z_0)$ belong to the boundary of $S_f$.
\end{inparaenum}
Moreover, case (b) only occurs if $z_0$ belongs to the common boundaries of $S_f$ and $S_f'$. 
We may avoid case (b) by assuming that $z'$ is on the boundary of $\Omega_0^0(f) \cup S'_f$. 
While due to the issue arising in case (a) we need to analyze two separate cases presented below. 

Let us first assume that $z' \notin \partial \mathcal{C}_f$. 
There is a round ball $B(z'', \delta') \subset B_{\delta}(\Omega^0_0(f)) \subset \Dom f$ such that $|z'-z''| \leq \delta/2$, 
$B(z'', \delta') \cap \Omega_0^0(f)=\emptyset$, and $f(B(z'', \delta')) \cap \Omega_0^0(f)=\emptyset$. 
Indeed, by the pre-compactness of the class $\QIS$, and the continuous dependence of the sets $\mathcal{C}_f^{-i}$ on $f$, 
we may assume that $\delta'$ is uniformly bounded from below, independent of $\alpha$ and $f$. 

Define the curve $\gamma'$ as the union of the line segment connecting $z$ to $z'$ and the line segment 
connecting $z'$ to $z''$. 
By definition, $\sup_{w,w'\in \gamma'} |\arg (w/w')|< \pi$. 
Moreover, $\gamma'$ is contained in $B(0, D_7 + \delta) \setminus \{0\}$, and is uniformly away from $0$. 

Define the curve $\gamma$ as the lift of $\gamma'$ under $\ex$ which starts at $\zeta_0$. 
Since $|z''|$ is uniformly bounded from above, $|(\ex^{-1})'(z'')|$ is uniformly bounded away from $0$. 
Combining with a bounded distortion argument (Theorem~\ref{T:Distortions}), 
we conclude that the lift of $B(z'',\delta')$ under $\ex$ contains a round ball about $\ex^{-1}(z'')$ whose size is 
uniformly bounded away from $0$, say by $r^*$. 
If necessary, we reduce the size of the ball to make it less than $1/4$. 
These readily imply Parts (1) and (2) of the proposition for some $\delta_1>0$ uniformly away from $0$. 
For Part (3) of the proposition we note that 
\[\diam \Re (B_{\delta_1}(B(\gamma(1), r^*)  \cup \gamma[0,1] )) \leq  2 \delta_1 + \frac{\pi}{2\pi} + r^*,\]
where the term $\pi/(2\pi)$ comes form the bound on the spiral of $\gamma'$ in the above paragraphs. 
Then, to make the above diameter less than $1-\delta_1$, we need to choose $\delta_1 \leq 3^{-1}(1/2-r^*)$. 

Finally since the diameter of $\gamma$ is uniformly bounded from above, the modulus in Part (4) is uniformly bounded 
from above. 
So we may choose a uniform $\delta_2$ in this case. 
This finishes the proof of the proposition in the case $z' \notin \partial \mathcal{C}_f$. 

In the remaining case $z' \in \partial \mathcal{C}_f$ we only need to slightly modify the above argument. 
That is, by Proposition~\ref{P:location-critical-piece}, there is a smooth curve 
$\gamma' \subset \Omega_0^0(f) \setminus B(0, 4 e^{-2\pi D_5}/27)$ connecting $z$ to 
$z' \in \partial \Omega_0^0(f) \setminus \mathcal{C}_f$ such that the length of $\gamma'$ is uniformly bounded from above 
and $\sup_{w,w'\in \gamma'} |\arg (w/w')| \leq C< 2\pi$. 
Now by the above argument, there is a ball $B(z'', \delta') \subset B_{\delta}(\Omega^0_0(f))$, 
whose size is uniformly bounded from below, $B(z'', \delta') \cap \Omega_0^0(f)=\emptyset$, and 
$f(B(z'', \delta')) \cap \Omega_0^0(f)=\emptyset$. 

Define the curve $\gamma''$ as the union of the curve $\gamma'$ and the line segment connecting $z'$ to $z''$. 
As in the above paragraphs, lifting $\gamma'' \cup B(z'', \delta')$ under $\ex$, we obtain the curve $\gamma$ as well as 
a round ball whose size is uniformly bounded away from $0$. 
Here, one must choose $r^* \leq (1 -C/(2\pi))/2$ and $\delta_1 \leq (1-C/(2\pi) - r^*)/3$.  
\end{proof}

\subsection{Metric properties of the orbits in the renormalization tower}\label{sec:metric-properties}

\begin{proof}[Proof of Proposition~\ref{P:change of coord}]
By the pre-compactness of the class of maps $\cup_{\alpha\in [0, r_2] } \QIS_\alpha$, and the continuous dependence of 
the normalized Fatou coordinate on the map, the diameter of $\p_{\rr(f)}$ is uniformly bounded from above. 
In particular, the absolute value of $w \in \p_{\rr (f)}$ is uniformly bounded from above. 

Let $\eta_{\rr(f)}: \p_{\rr (f)} \to \Phi_f(\p_f)$ be an inverse branch of $\ex$ that satisfies 
$\Re (\eta_{\rr(f)}(\p_{\rr f})) \subset [0, \hat{\Bk}+1]$ as in Equation~\eqref{WidthofLift}.
Let $\zeta:=\eta_{\rr(f)}(w)$. 
By the above paragraph, $\Im \zeta =\frac{-1}{2\pi}\log \frac{27|w|}{4}$ must be bounded from below by a uniform constant, 
say $M'$. 
Let $M$ be the constant produced by Proposition~\ref{P:nearly-translation-pre-coord} for $M'$. 

To prove the proposition, we consider two cases; small values of $\alpha$ and large values of $\alpha$. 
Recall the constant $\Bk$ and $\hat{\Bk}$ from Proposition~\ref{P:uniformly-bounded-width-spiral}.
First we assume that $\alf$ is small enough so that 
\[\alf\leq 1/(2\Bk+2\hat{\Bk}+2), \text{ and } 1/(4\alf) \geq \hat{\Bk}+2+M\log (1+1/\alf).\] 
In this case we set $\kappa(f):=\lfloor 1/(2\alf) \rfloor$. 
By the first condition on $\alpha$ above, we have 
\[0 <\Re (\eta_{\rr(f)}(\p_{\rr f}))+ \kappa(f) \leq \hat{\Bk}+1 + 1/(2\alpha) \leq 1/\alpha -\Bk.\] 
In particular, combining with Proposition~\ref{P:uniformly-bounded-width-spiral}, we conclude that 
$\eta_{\rr(f)}(\p_{\rr(f)})+\kappa(f) \subset \Phi_f(\p_f)$. 
This implies that 
\[f\co{\kappa(f)} \circ \psi_{\rr (f)} (\p_{\rr(f)}) = 
f\co{\kappa(f)} \circ  \Phi_f^{-1} \circ \eta_{\rr (f)} (\p_{\rr(f)})= 
\Phi_f^{-1} \big(\eta_{\rr (f)} (\p_{\rr(f)}) + \kappa(f)\big ) \subset \p_f.\]
This proves part (1) of the proposition in this case. 

On the other hand, one can see that by the above conditions on $\alf$, $\zeta+\kappa(f)$ belongs to the set
$[0, 1/\alf-\Bk]+\Bi[M', +\infty)$.  
Then, by Proposition~\ref{P:nearly-translation-pre-coord}, we have 
\begin{gather*} 
|L_f^{-1}(\zeta+\kappa(f))- (\zeta+\kappa(f))| \leq M \log(1+1/\alf).
\end{gather*}
Hence, 
\begin{gather}
\Im L_f^{-1}(\zeta+\kappa(f)) \geq \Im \zeta-M \log(1+1/\alf), \label{local-bound-1}\\
|\Re L_f^{-1}(\zeta+\kappa(f))- 1/(2\alf) | \leq \hat{\Bk}+2+ M \log(1+1/\alf))\leq 1/(4\alf). \label{local-bound-2} 
\end{gather}
Using Lemma~\ref{strip-image}-2 (with $r=1/4$) and the above inequalities, 
there is a constant $M_2$ depending only on $C_1$ and $M$, such that  
\begin{align*}
|f\co{\kappa(f)}(\psi_{\rr(f)}(w))|&=|\Phi_f^{-1} (\eta_{\rr f} (w)+\kappa(f))| \\
&=|\tau_f \circ L_f^{-1}(\eta_{\rr f}(w)+\kappa(f)))| \\
&=|\tau_f(L_f^{-1}(\zeta+\kappa(f)))| \\
& \leq 4C_1 e^{2\pi} \alf e^{-2\pi\alf \Im L_f^{-1}(\zeta + \kappa(f))} \\
& \leq 4C_1 e^{2\pi} \alf e^{-2\pi\alf (\Im \zeta -M\log(1+1/\alf))}
\leq M_2\cdot \alf |w|^{\alf}.
\end{align*}
This proves Part (2) of the proposition for small values of $\alpha$. 

Now we consider larger values of $\alpha$ that do not satisfy the above conditions. 
Here we set $\kappa(f)=0$. 
Then, $f\co{\kappa(f)} \circ \psi_{\rr (f)} (\p_{\rr(f)})= \psi_{\rr (f)} (\p_{\rr(f)}) \subset \p_f$, by 
Equation~\eqref{E:psi-injection}. 
This proves Part~(1) of the proposition in this case. 

Since $\zeta \in [0, 1/\alf-\Bk]+\Bi[M', +\infty)$, as in the above argument, we must have 
$\Im L_f^{-1}(\zeta) \geq \Im \zeta-M \log(1+1/\alf)$. 
(Here we do not need the bound in \eqref{local-bound-2}.)  
Then, by an elementary estimate on $\tau_h$, there is a uniform constant $M_2$ such that 
$|\tau_f(L_f^{-1}(\zeta))|\leq M_2\cdot |w|^\alf$.
That is, $|\psi_{\rr(f)}(w)| \leq M_2 |w|^{\alpha}$. 
However, since $\alf$ is bounded from below here, one may adjust $M_2$ to accommodate the parameter $\alf$ in the formula. 
This finishes the proof of Part (2) in this case. 
\end{proof}

\begin{proof}[Proof of Proposition~\ref{P:height-control-lem}]
By Proposition~\ref{P:nearly-translation-pre-coord}, with $M'=0$, we find a constant $M$ (independent of $n$) 
such that for all $\zeta\in [0, x_h]+\Bi [0, +\infty)$ we have 
\[\Im L_{n+1}^{-1}(\zeta)\geq \Im \zeta - M \log (1+1/\alf_{n+1}).\]
Choose $D_1>0$ such that for all $\alf \in (0,1)$, we have 
\[\frac{D_1}{\alf} - M \log (1+\frac{1}{\alf}) \geq \frac{1}{4\alf}.\]
If $\Im \zeta_{n+1}\geq D_1/\alf_{n+1}$, the above equations guarantee that  
\[\Im L_{n+1}^{-1}(\zeta_{n+1})\geq \frac{1}{4\alf_{n+1}}.\]
This implies that $L_{n+1}^{-1}(\zeta_{n+1})\in \Theta_{\alf_{n+1}}(\tfrac{1}{4\alf_{n+1}})$. 
By Lemma~\ref{strip-image}-2, with $r=1/4$,   
\begin{align*}
|\tau_{n+1}(L_{n+1}^{-1}(\zeta_{n+1}))|
&\leq 4C_1 e^{2\pi} \alf_{n+1} e^{-2\pi \alf_{n+1} (\Im \zeta_{n+1}-M \log (1+1/\alf_{n+1}))}\\
&\leq C  \alf_{n+1}e^{-2\pi\alf_{n+1}\Im \zeta_{n+1}},
\end{align*} 
for some constant $C$ that depends only on $C_1$ and $M$. 
Here we have used that $\alf_{n+1}\log (1+1/\alf_{n+1})$ is uniformly bounded from above independent of $\alf_{n+1}\in (0,1)$. 
Recall that $\Phi_{n+1}(w_{n+1})=\zeta_{n+1}$, and $\Phi_{n+1}^{-1}= \tau_{n+1} \circ L_{n+1}^{-1}$. 
Hence, we have shown that $|w_{n+1}| \leq C\alf_{n+1}e^{-2\pi\alf_{n+1}\Im \zeta_{n+1}}$.

By Proposition~\ref{P:turning} and Equation~\eqref{sign-property}, $w_{n+1}$ is mapped to $z_{n+1}$ in a 
uniformly bounded number of iterates of $f_{n+1}$. 
The map $f_{n+1}$ is of the form 
\[z \mapsto P \circ \phi_{n+1}^{-1}(e^{2\pi \Bi \alpha_{n+1}} \cdot z) : (e^{-2\pi \Bi \alpha_{n+1}} \cdot \phi_{n+1}(U))  \to \cc,\]  
with $|f_{n+1}'(0)|=1$ and $\phi_{n+1}$ has univalent extension over the larger domain $V$ (see Theorem~\ref{Ino-Shi2}). 
This implies that, there exists a uniform constant $C'$ such that $|z_{n+1}|\leq C' |w_{n+1}|$. 

Recall that $\ex (\zeta_n)= z_{n+1}$. 
Combining the above two paragraphs, we have 
\begin{equation*}
\frac{4}{27}e^{-2\pi \Im \zeta_n} =|\frac{-4}{27} e^{-2\pi \Bi \overline{\zeta_n}}|
=|z_{n+1}| \leq C'C \alf_{n+1}e^{-2\pi\alf_{n+1}\Im \zeta_{n+1}}.    
\end{equation*}
Multiplying the above equation by 27/4, and then taking $\log$, we obtain
\[2\pi \alpha_{n+1} \Im \zeta_{n+1} \leq \log (\frac{27 CC'}{4}) + \log  \alpha_{n+1} + 2\pi \Im \zeta_n.\]
Then dividing through by $2\pi \alpha_{n+1}$ we obtain, 
\[\Im \zeta_{n+1} \leq \frac{1}{2\pi \alpha_{n+1}} \log (\frac{27 CC'}{4}) + \frac{1}{2\pi \alpha_{n+1}} \log  \alpha_{n+1} + 
\frac{1}{\alpha_{n+1}} \Im \zeta_n.\]
This is the desired inequality in the proposition when we define the constant $D_2=\frac{1}{2\pi} \log (\frac{27 CC'}{4})$.
\end{proof}

\begin{proof}[Proof of Proposition~\ref{P:borrowed-lem}]
First we prove that for every $D>0$ there exists $E>0$ such that if 
$\Im \zeta_{n+1}\leq D/\alf_{n+1}$ then $\Im \zeta_n\leq \tfrac{1}{2\pi} \log \tfrac{1}{\alf_{n+1}}+E$.

The map $f_{n+2}$ has the form 
\[z \mapsto P \circ \phi_{n+2}^{-1}(e^{2\pi\alf_{n+2}\Bi} \cdot z): (e^{-2\pi\alf_{n+2}\Bi} \cdot \phi_{n+2}(U)) \to \cc,\] 
with $|f_{n+2}'(0)|=1$. 
Recall that by Theorem~\ref{Ino-Shi2} the map $\phi_{n+2}$ has univalent extension onto the larger domain $V$. 
By the distortion theorem~\ref{T:Distortions}, this implies that $\Dom f_{n+2}=\phi_{n+2}(U)\cdot e^{-2\pi\alf_{n+2}\Bi}$ 
has a uniformly bounded diameter in $\cc$. 
Then, as $\ex(\zeta_{n+1})\in \Dom f_{n+2}$, $\Im \zeta_{n+1}$ must be uniformly bounded from below by a constant 
$M'$ independent of $n$. 
Using Proposition~\ref{P:nearly-translation-pre-coord}, with $M'$, we obtain a constant $M$, 
independent of $n$, such that  
$\Im L_{n+1}^{-1}(\zeta_{n+1}) \leq \Im \zeta_{n+1} + M \log (1+1/\alf_{n+1})$.  
For points $\zeta_{n+1}$ with $\Im \zeta_{n+1}\leq D/\alf_{n+1}$, we obtain 
\[\Im L_{n+1}^{-1}(\zeta_{n+1}) \leq D/\alf_{n+1} + M \log (1+1/\alf_{n+1}).\]
By an explicit estimate on the covering map $\tau_{f_{n+1}}$ given by the formula in Equation~\eqref{E:covering-formula},  
there is  a constant $C$ independent of $n$ such that 
\[|w_{n+1}|=|\tau_{f_{n+1}}(L_{n+1}^{-1}(\zeta_{n+1}))| \geq C \alf_{n+1}.\] 
The point $w_{n+1}$ is mapped to $z_{n+1}$ by a uniformly bounded number of iterates of $f_{n+1}$. 
Moreover, if $z_{n+1}$ is close to $0$, then $w_{n+1}$ must be also close to $0$. 
That is due to the covering structure of $P$ on $U$, which covers a neighborhood of $0$ only once. 
These imply that there is a constant $C'$ independent of $n$ such that $|z_{n+1}|\geq C' \alf_{n+1}$.  
Then, as $\zeta_n$ is mapped to $z_{n+1}$ by $\ex$, we obtain 
\[\Im \zeta_n \leq \frac{-1}{2\pi} \log (\frac{27 C'}{4}) + \frac{1}{2\pi} \log \frac{1}{\alpha_{n+1}}.\]
This finishes the proof of the claim by introducing $E=\frac{-1}{2\pi} \log (\frac{27 C'}{4})$. 

It is proved in Lemma~3.11 in \cite{Ch10-II} that there exists a constant $D>0$ such that given any 
$z\in \cap_{n=0}^\infty \Omega_0^n \setminus \overline{\Delta(f)}$ there are infinitely many integers $m$ 
with $\Im \zeta_m\leq D/\alf_m$. 
(Indeed, the statement of the lemma in that paper concerns $z \in \mathcal{PC}(f) \setminus \overline{\Delta(f)}$, however, 
the proof is written for $\cap_{n=0}^\infty \Omega_0^n \setminus \overline{\Delta(f)}$.) 
Combining this with the statement in the first paragraph, we conclude that there are infinitely many levels $m$ with 
$\Im \zeta_{m-1} \leq \frac{1}{2\pi} \log \frac{1}{\alpha_{m}} + E$. 
\end{proof}

\begin{proof}[Proof of Proposition~\ref{P:lift-of-integers}]
Fix $\zeta \in \ex^{-1}(\Omega_{n+1}^0)$ satisfying $\Im \zeta \leq \frac{1}{2\pi} \log \alpha_{n+1}^{-1} +E$. 
Since $\Omega_{n+1}^0$ has a uniformly bounded diameter, independent of $n$, there is a uniform constant 
$C_1$ such that $\Im \zeta \geq C_1$. 
See, for instance, the proof of Proposition~\ref{P:borrowed-lem} for further details. 

First assume that $\alpha$ is small enough so that $2 \leq 1/(2\alf_{n+1})\leq 1/\alf_{n+1}-\Bk$. 
For an inverse branch of $\ex$, denoted by $\hat{\eta}_{n+1}$, we may consider the continuous curve
$\Upsilon_{n+1}(t):= \hat{\eta}_{n+1}\circ \Phi_{n+1}^{-1}(t)$, for $1\leq t \leq \tfrac{1}{2\alf_{n+1}}$. 
By the same argument as in the proof of Sublemma~\ref{long-lifts}, there exists a uniform constant $C_2$ 
such that for every choice of $\hat{\eta}_{n+1}$, we have  
\[\frac{1}{2\pi}\log \frac{1}{\alf_{n+1}}- C_2 \leq \Im \Upsilon_{n+1}(\frac{1}{2\alf_{n+1}}) .\]
Moreover, 
\[\Im \Upsilon_{n+1}(1) = \Im \hat{\eta}_{n+1} \circ \Phi_{n+1}^{-1}(1) 
= \Im \hat{\eta}_{n+1} (-4/27)= 0.\]

On the other hand, by Sublemma~\ref{nice-derivative}, $|\Upsilon_{n+1}'(t)|$, for $1 \leq t \leq \tfrac{1}{2\alf_{n+1}}$, 
is uniformly bounded from above, say by $C_3$. 
Hence, for every integer $i\in [2, \tfrac{1}{2\alf_{n+1}}]$, the Euclidean distance $d(\Upsilon_{n+1}(i-1), \Upsilon_{n+1}(i))$ is 
at most $C_3$.
It follows that there exists a choice of the inverse branch $\hat{\eta}_{n+1}$ and an integer $i\in [1, \tfrac{1}{2\alf_{n+1}}]$
such that $\zeta':=\Upsilon_{n+1}(i)$ satisfies the desired inequalities in the proposition.  

Now let us assume that $\alpha$ is bounded from below, that is, $1/(2\alf_{n+1}) > 1/\alf_{n+1}-\Bk$. 
(We are still assuming that $\alpha \leq r_2$, so there may not be any such $\alpha$.)
Then, by the hypothesis of the proposition, $\Im \zeta$ is uniformly bounded from above. 
As in the second paragraph, it is also uniformly bounded from below. 
Then, there is an element $\zeta'$ in the set $\ex^{-1} \circ \Phi_{n+1}^{-1}(1)= \mathbb{Z}$ that satisfies the desired 
inequalities in the proposition. 
\end{proof} 

\subsection*{Acknowledgment}
I would like to thank Hiroyuki Inou, Misha Lyubich, Mitsuhiro Shishikura, and Saeed Zakeri 
for useful discussions while carrying out this research. 
I am grateful to the referee(s) for detailed comments on the presentation of the paper, and on 
the proofs. 

\bibliographystyle{smfalpha}

\bibliography{Main.bbl}
\end{document}